\documentclass[ps,preprint,twoside,numbers,natbib]{imsart}
\usepackage{amsthm,amsmath,amssymb}
\usepackage[normalem]{ulem}
\usepackage{mathrsfs}
\usepackage{chngcntr}
\usepackage{dsfont}
\usepackage[shortlabels]{enumitem}
\usepackage{graphicx}
\usepackage{hyperref}

\doi{10.1214/19-PS339}
\pubyear{2020}
\volume{17}
\issue{0}
\firstpage{478}
\lastpage{544}

\startlocaldefs
\newtheorem{theorem}[equation]{Theorem}
\newtheorem{proposition}[equation]{Proposition}
\newtheorem{lemma}[equation]{Lemma}
\newtheorem{corollary}[equation]{Corollary}
\newtheorem{property}{Property}
\theoremstyle{definition}
\newtheorem{definition}[equation]{Definition}
\newtheorem{remark}[equation]{Remark}
\newtheorem*{remarkVTEX}{Remark}
\newtheorem*{problem}{Open Problem}
\counterwithin{figure}{section}
\counterwithin{equation}{section}
\def\arxivurl#1{\href{https://arxiv.org/abs/#1}{#1}}
\endlocaldefs

\begin{document}

\begin{frontmatter}
\title{Activated Random Walks on $\mathbb{Z}^{d}$\thanksref{t1}}
\runtitle{Activated Random Walks on $\mathbb{Z}^{d}$}
\thankstext{t1}{This is an original survey paper}

\begin{aug}
\author{\fnms{Leonardo T.} \snm{Rolla}\ead[label=e1]{leorolla@dm.uba.ar}}
\address{\printead{e1}}
\runauthor{L. T. Rolla}
\end{aug}

\begin{abstract}
Some stochastic systems are particularly interesting as they exhibit critical
behavior without fine-tuning of a parameter, a phenomenon called self-organized
criticality. In the context of driven-dissipative steady states, one of
the main models is that of Activated Random Walks. Long-range effects intrinsic
to the conservative dynamics and lack of a simple algebraic structure cause
standard tools and techniques to break down. This makes the mathematical
study of this model remarkably challenging. Yet, some exciting progress
has been made in the last ten years, with the development of a framework
of tools and methods which is finally becoming more structured. In these
lecture notes we present the existing results and reproduce the techniques
developed so far.
\end{abstract}


\received{\smonth{12} \syear{2019}}

\tableofcontents

\end{frontmatter}
\section{Overview}

In the study of critical phenomena, there is a large class of non-equilibrium
lattice systems that naturally evolve to a critical state, characterized by
power-law distributions for the sizes of relaxation events. A typical example
is the occurrence of huge avalanches caused by small perturbations. In many
cases, such systems are attracted to a stationary critical state without
being specifically tuned to a critical point.

This seems to be the explanation for the emergence of random fluctuations at
a macroscopic or mesoscopic scale, and creation of self-similar shapes in a
variety of growth systems. Among attempts to explain long-ranged
spatial-temporal correlations, the physical paradigm called
\emph{self-organized criticality} is a widely accepted theory, although it is
still very poorly understood from a mathematical point of view.

For non-equilibrium steady states, it eventually became evident that
self-organized criticality is related to conventional critical behavior of a
system undergoing a phase transition. In the case of driven-dissipative
systems, it is related to an \emph{absorbing-state phase transition}. These
are systems whose dynamics drives them towards, and then maintains them at
the edge of stability. The phase transition arises from the conflict between
spread of activity and a tendency for this activity to die out, and the
critical point separates a phase with sustained activity and an absorbing
phase in which the dynamics is eventually extinct in any finite region.

The main stochastic models in this class are the Manna Sandpile Model, its
Abelian variant which we call the Stochastic Sandpile Model, and the
Activated Random Walks (ARW). Due to long-range effects intrinsic to their
conservative dynamics, classical analytic and probabilistic techniques fail
in most cases of interest, making the rigorous analysis of such systems a
major mathematical challenge. Yet some exciting progress has been made in the
last ten years. In particular, a more structured framework of tools and
methods started to develop and emerge.

In these notes we recall the existing results, and describe in a unified
framework the tools and techniques currently available. Most of the material
is devoted to the ARW. In \S \ref{sub:ssm} we mention the Stochastic Sandpile
Model, for which much less is known. The Manna model so far seems
intractable.

In \S \ref{sub:arwdef} we describe the local rules for the ARW evolution, in
\S \ref{sub:asptsoc} we discuss the relation between self-organized
criticality and absorbing-state phase transitions, and in \S
\ref{sub:predictions} we cast some of the physical predictions. Most of them
are far outside the reach of currently known mathematical techniques, as
discussed in \S \ref{sub:challenges}. Then in \S \ref{sub:results} we quote
all the known results (which, as the reader will see, leave most of the
predictions as open conjectures), and in \S \ref{sub:methods} we describe the
main methods developed in the past ten years. Finally, in \S
\ref{sub:structure} we discuss the structure and interdependence of the
remaining sections.

\subsection{Activated Random Walks}\label{sub:arwdef}

The ARW evolution is defined as follows. Particles sitting on the graph
$\mathbb{Z}^{d}$ can be in state $A$ for \emph{active} or $S$ for
\emph{sleeping}. Each active particle, that is, each particle in the $A$
state, performs a continuous-time random walk with \emph{jump rate}
$D_{A}=1$. The walks follow a translation-invariant \emph{jump distribution},
that is, they jump from $x$ to $x+z$ with probability $p(z)$ for some fixed
distribution $p(\cdot )$ on $\mathbb{Z}^{d}$. We assume that the support of
$p$ is restricted to the nearest-neighbors $\pm e_{1} , \dots , \pm e_{d}$
and spans all of $\mathbb{Z}^{d}$.

Several active particles can be at the same site. When a particle is alone,
it may fall asleep, a \emph{reaction} denoted by $A\to S$, which occurs at a
\emph{sleep rate} $0 < \lambda \leqslant \infty $. So each particle carries
two clocks, one for jumping and one for sleeping. Once a particle is
sleeping, it stops moving, \textit{i.e.}\ it has jump rate $D_{S}=0$, and it
remains sleeping until the instant when another particle is present at the
same site. At such an instant the particle which is in the $S$ state flips to
the $A$ state, giving the \emph{reaction} $A+S \to 2A$.

If the clock rings for a particle to sleep while it shares a site with other
particles, the tentative transition $A \to S$ is overridden by the
instantaneous reaction $A+S \to 2A$, so this attempt to sleep has no effect
on the system configuration.

A particle in the $S$ state stands still forever if no other particle ever
visits the site where it is located. When a site has no particles or one
sleeping particle, it is called \emph{stable}, otherwise it has one or more
particles, all active, and is called \emph{unstable}. A stable site stays
stable indefinitely, and can only become unstable if later on it is visited
by an active particle. An \emph{absorbing configuration} is one for which
every site is stable.

We note that, at the extreme case $\lambda =\infty $, when a particle visits
an unoccupied site, it falls asleep instantaneously. This case is equivalent
to \emph{internal diffusion-limited aggregation} with multiple sources.

We have described \emph{local rules} for the system to evolve. In order to
fully describe the system, we need to specify on which subset of
$\mathbb{Z}^{d}$ this dynamics will occur, what are the boundary conditions,
and the initial state at $t=0$. By ``state'' we mean a probability
distribution on the space of configurations, unless it refers to the state
$A$ or $S$ of a particle.

\subsection{Phase transition and self-organized criticality}\label{sub:asptsoc}

We consider two different dynamics which follow the above local rules.

\paragraph{Infinite-volume conservative system.}
On the infinite lattice $\mathbb{Z}^{d}$, at $t=0$ we start from a
translation-ergodic state with average density of particles $\zeta $. It
turns out that, because the dynamics does not create or destroy particles,
this average density $\zeta $ is conserved during the evolution. We say that
this system fixates, or is absorbed, if each site is visited only finitely
many times and eventually becomes stable. Otherwise, if each site is visited
infinitely many times, we say that the system stays active. This model shows
an \emph{ordinary phase transition} in the sense that, for some $\zeta _{c}$,
the dynamics a.s.\ fixates when $\zeta <\zeta _{c}$ and a.s.\ stays active
when $\zeta >\zeta _{c}$. This is called an \emph{absorbing-state phase
transition}.

\paragraph{Driven-dissipative system.}
On a finite box $V_{L}=\{-L,\dots ,L\}^{d}$, we define a system with three
components. At rate one, a new active particle is added to a site $x\in V$
chosen uniformly at random. The ARW dynamics is run with time being
accelerated by a factor of $\kappa >1$. The box has open boundary,
\textit{i.e.}\ particles are killed when they exit $V$.

We first let $\kappa \to \infty $, so that the whole box is stabilized right
after a particle is added, so we obtain a Markov chain on the space of
absorbing configurations called \emph{driven-dissipative dynamics}. We then
let $t\to \infty $ to reach a stationary state $\nu _{s}^{L}$ supported on
absorbing configurations. We finally let $L\to \infty $ to have a state $\nu
_{s}$ on $\mathbb{Z}^{d}$ with mean density $\zeta _{s}$, see \S
\ref{sub:predictions}.

\emph{Self-organized criticality} for the driven-dissipative system loosely
means the following. When the average density $\zeta $ inside the box is too
small, mass tends to accumulate. When it is too large, there is intense
activity and a substantial number of particles are killed at the boundary.
With this mechanism, the model is attracted to a steady state with an average
density given by $0<\zeta _{s}<\infty $, and this state has several features
associated to criticality. The \emph{density conjecture} says that
$\zeta_{c}$ and $\zeta _{s}$ should coincide. Moreover, the critical
exponents of the driven-dissipative system should be related to those of the
infinite-volume one.

\subsection{Predictions}\label{sub:predictions}

Consider a system running on the whole graph $\mathbb{Z}^{d}$. At $t=0$,
sites have an i.i.d.\ Poisson number of active particles with parameter
$\zeta $.

We say that the system \emph{fixates} if, for each site, there is a random
time after which the site is either vacant or contains one sleeping particle.
We say that the system \emph{stays active} if, for each site, there are
arbitrarily large times at which the site has at least one active particle.
There is a \emph{critical density} $\zeta _{c}$, which is non-decreasing in
$\lambda $, such that the system will a.s.\ fixate for $\zeta <\zeta _{c}$
and the system will a.s.\ stay active for $\zeta >\zeta _{c}$.

In this subsection we describe some aspects of the ARW behavior. Some of
these claims have been proved, most of them remain widely open.

\paragraph{Phase space.}
The critical density satisfies $\zeta _{c}<1$ for every $\lambda <\infty $
and $\zeta _{c} \to 0$ as $\lambda \to 0$. That $\zeta _{c} \leqslant 1$
should be obvious from the fact that each site can accommodate at most one
sleeping particle. But the particles should be able to sustain activity even
at densities lower than unit (except for the case $\lambda =\infty $).
Moreover, $\zeta _{c}>0$ for every $\lambda >0$ and $\zeta _{c} \to 1$ as
$\lambda \to \infty $. In particular, $\zeta _{c}=1$ when $\lambda =\infty $.
More generally, $\zeta _{c}$ is continuous and strictly increasing in
$\lambda $. See Figure~\ref{fig:predictions}.

\begin{figure}
\includegraphics{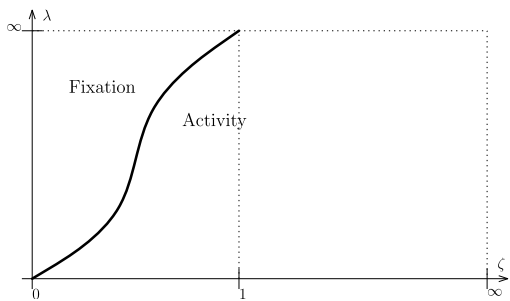}%
\caption{Prediction for the phase space.}\label{fig:predictions}
\end{figure}

\paragraph{Uniqueness of the critical density.}
The value of $\zeta _{c}$ depends on the dimension $d$, on the jump
distribution $p(\cdot )$ and on the sleep rate $\lambda $. But it does not
depend on the choice of i.i.d.\ Poisson for the initial state. More
precisely, for every translation-ergodic active initial state (see next
paragraph) with density $\zeta $, the system a.s.\ fixates if $\zeta <\zeta
_{c}$ and a.s.\ stays active if $\zeta >\zeta _{c}$. At $\zeta =\zeta _{c}$
the system should also stay active (except for the trivial case of $\lambda
=\infty $ and non-random initial condition with density $\zeta =1$).

\paragraph{Invariant measures and convergence.}
For each $\zeta < \zeta _{c}$, all the translation-ergodic stationary
distributions with average density $\zeta $ are \emph{absorbing states}, that
is, they are measures supported on absorbing configurations. Starting from
any state with such density, the evolution a.s.\ converges to an absorbing
configuration having the same density. In general, different initial states
are attracted to different absorbing states.

For each $\zeta >\zeta _{c}$, there is a unique translation-ergodic
stationary \emph{active state} (\textit{i.e.}\ a state which is not
absorbing, which by ergodicity means a positive fraction of the particles are
active) with density $\zeta $. For every translation-ergodic active initial
state with density $\zeta $, the system converges in law to this unique
stationary active state as $t\to \infty $. So in terms of basin of
attraction, the active state is stable and absorbing states are unstable, in
conflict with our terminology for stable and unstable sites at the
microscopic level.

At $\zeta =\zeta _{c}$ and active initial states, the system converges in law
to a unique absorbing state, but the system a.s.\ stays active and
convergence is in law only. So the critical case mixes features from both
phases.

\paragraph{Power laws at and near criticality.}
At $\zeta =\zeta _{c}$, the average density of activity (number of active
particles per site) at time $t$ decays as a power of $t$ as $t \to \infty $.
On the other hand, if $\zeta > \zeta _{c}$, the density of activity seen in
the stationary regime $t = \infty $ decays as a power of $\zeta -\zeta _{c}$
as $\zeta \downarrow \zeta _{c}$. Two-point correlations in space also decay
as a power of the distance $\Delta x$, and same-site time correlations decay
as a power of $\Delta t$. Outside criticality, correlation decays
exponentially with a typical correlation length for space and another one for
time, and the correlation lengths themselves diverge as powers of $|\zeta
-\zeta _{c}|$ as $\zeta \to \zeta _{c}$. More details on critical exponents
can be found in~\cite{DickmanRollaSidoravicius10}.

\paragraph{Driven-dissipative dynamics.}
Consider the driven-dissipative dynamics described in \S \ref{sub:asptsoc}.
Let $\zeta _{s}^{L}$ be the average density of particles in $\nu _{s}^{L}$,
and let $\zeta _{s}=\lim _{L} \zeta _{s}^{L}$. The \emph{density conjecture}
says that $\zeta _{s}=\zeta _{c}$. Moreover, the state $\nu _{s}$ should have
the same two-point correlation decay exponents as conservative
infinite-volume system at criticality.

A stronger version of the density conjecture is the following. Let $\zeta \in
[0,\infty )$. Consider an i.i.d.\ Poisson configuration with density $\zeta $
on the box $V_{L}$, and run the ARW dynamics with open boundary until it
reaches an absorbing configuration. Then the final state has density
concentrated around some value which depends on $\zeta $ and $L$, and this
value tends to $\min \{\zeta ,\zeta _{c}\}$ as $L\to \infty $.

\paragraph{Fixed-energy dynamics.}
Consider the ARW dynamics on a large torus
$\mathbb{Z}_{n}^{d}=(\mathbb{Z}/n\mathbb{Z})^{d}$ instead of
$\mathbb{Z}^{d}$, starting with approximately $\zeta n^{d}$ particles. Almost
surely, this system will eventually fixate if and only if it has fewer than
$n^{d}$ particles, so the question is not \emph{whether} it fixates. The
relevant quantity is \emph{how long} it will take the system to fixate. For
parameters $\lambda $ and $\zeta $ inside the active phase for
$\mathbb{Z}^{d}$ as shown in Figure~\ref{fig:predictions}, it should take a
long time to fixate (exponential in $n^{d}$), with high probability. For
parameters inside the fixating phase, the corresponding system on a large
torus should fixate in a short time (faster than any positive power of $n$).
For parameters on the critical curve, the time to fixate should be in
between. For the critical and near-critical regimes, several quantities
should decay or blow up as power laws, such as space and time correlation
lengths, activity decay, etc.

\subsection{Open problems and challenges}\label{sub:challenges}

In principle, all the statements in \S \ref{sub:predictions} that do not
appear in \S \ref{sub:results} are open problems. But most of them are far
beyond the reach of current techniques. Throughout these lecture notes we
will explicitly mention some more realistic open problems, after the
background needed to properly state each question has been introduced.

The first difficulty in studying the ARW lies in the fact that this system is
not attractive. This is overcome by considering a site-wise kind of
construction, or an explicit construction in terms of a collection of random
walks, rather than the Harris graphical construction. These frameworks allow
for different kinds of arguments which have proven to be very useful.

Another feature of this model --particle conservation-- still causes
tremendous difficulties. This has so far restrained most attempts to apply
arguments of the type ``energy vs.\ entropy.'' These arguments typically go
as follows. One first identifies some structure that is intrinsic to the
occurrence of events that conjecturally should not occur. All possible
structures are then enumerated, and their number is overwhelmed by the high
probabilistic cost needed for them to occur. One then concludes that such
events have vanishing probability. This approach has been very successful in
many branches of statistical mechanics. However, for reaction-diffusion
dynamics, the conservation of particles in the system gives rise to intricate
long-range effects, which makes it difficult to find suitable structures
within the occurrence of events of interest. In \S \ref{sec:netflow} we
present the only case where an approach involving enumeration of events and
compensation by an extreme choice of parameters has been implemented with
some success.

\subsection{Results}\label{sub:results}

We now briefly summarize known results towards the above predictions. The
results will be stated under the common assumption that the jumps are to
nearest neighbors only. In the next sections we will go through the proof of
all the results mentioned here. For bibliographical references, see \S
\ref{sec:extensions}.

\begin{figure}
\includegraphics{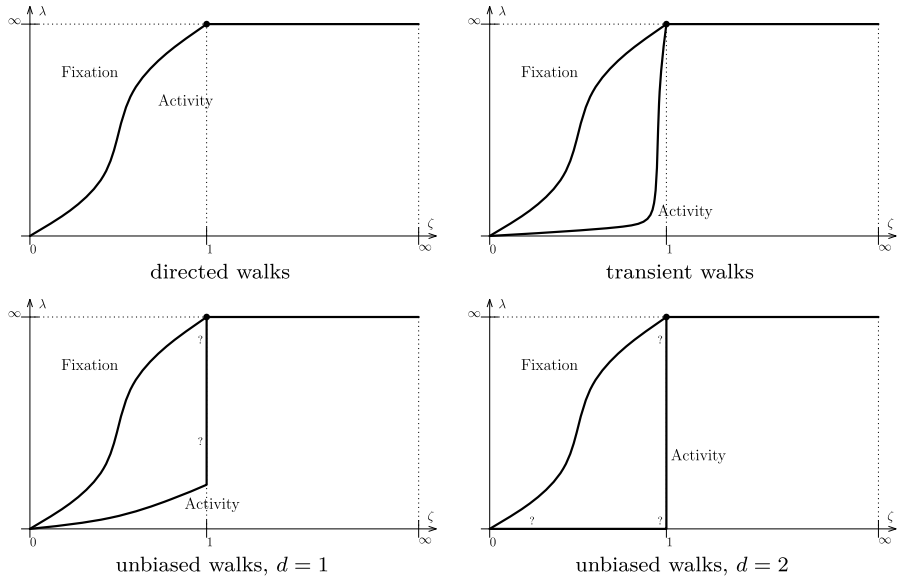}%
\caption{Results about fixation vs.\ activity on the phase space.}\label{fig:results}
\end{figure}

\paragraph{Phase space.}
Results regarding the phase space for the infinite-volume system on
$\mathbb{Z}^{d}$ are summarized in Figure~\ref{fig:results}. It is known that
$\zeta _{c} \geqslant \frac{\lambda }{1+\lambda }$ in general (\S
\ref{sec:traps} and \S \ref{sec:weak}). In particular, there is a fixation
phase for every fixed $\lambda >0$ by taking $\zeta $ small, and there is a
fixation phase for every fixed $\zeta <1$ by taking $\lambda $ large. It is
also known that, for every $\lambda \leqslant \infty $, there is no fixation
at $\zeta =1$ (\S \ref{sec:particlewise}).

In the case of unbiased walks (\textit{i.e.}\ $\sum _{y} y p(y)={
\boldsymbol{0}}$) on $d=2$, this is all we know, and proving existence of a
non-trivial active phase (\textit{i.e.}\ for some $\lambda >0$ and $\zeta
<1$) remains an open problem. On $d = 1$, for every $\zeta >0$ fixed there is
an active phase by taking $\lambda $ small (\S \ref{sec:netflow}), but it
remains open to show existence of a non-trivial active phase for every fixed
$\lambda <\infty $ and $\zeta $ close enough to $1$.

For transient walks, the picture is fairly complete: for every fixed $\zeta
>0$ the system stays active if $\lambda $ is small enough, and for every
fixed $\lambda <\infty $ the system stays active if $\zeta $ is close enough
to $1$. There are different proofs for unbiased (\S \ref{sec:weak}) and
biased (\S \ref{sec:counting}) walks. If the walks are not only biased but
directed (\textit{i.e.}\ $p(-e_{j})=0$ for every $j$), the critical curve can
be described explicitly: $ \zeta _{c} = \frac{\lambda }{1+\lambda } $, and
for $d=1$, there is no fixation at $\zeta =\zeta _{c}$ (\S
\ref{sec:counting}).

\paragraph{Fixed-energy dynamics.}
For the one-dimensional torus (\textit{i.e.}\ the ring $\mathbb{Z}_{n}$), it
has been shown that there is a slow stabilization phase and a fast
stabilization phase. Consider the average activity time $\mathcal{T}$ given
by the sum of the total time each particle is active, divided by $n$. For
every $0<\zeta <1$ fixed, if $\lambda $ is large enough then $
\mathcal{T}\leqslant C \log ^{2} n $, and if $\lambda $ is small enough then
$ \mathcal{T}\geqslant e^{cn} $, with high probability as $n\to \infty $ (\S
\ref{sec:cycle}).

\paragraph{Uniqueness of the critical density.}
The prediction given in \S \ref{sub:predictions} says that
translation-ergodic distributions with average particle density $\zeta >\zeta
_{c}$ stay active as long as they are not supported on absorbing
configurations. The partial result presented in \S \ref{sec:universality}
holds for distributions supported on completely active configurations, that
is, without any sleeping particle. Closing this gap is a major question and
would probably be an important step towards the density conjecture.

\subsection{Methods}\label{sub:methods}

Most proofs rely on the properties of the site-wise representation described
in \S \ref{sub:diaconis}. To study the phase space and establish regions of
fixation and activity, we normally check one of the conditions stated in \S
\ref{sub:conditions}. Uniqueness of the critical density also allows us to
make convenient assumptions about the initial distribution.

In practice, to check one of these conditions we need to describe a toppling
procedure for which probabilistic estimates can be obtained. \emph{Toppling}
is a one-step update of the current configuration according to the dynamics
described in \S \ref{sub:arwdef}, and the \emph{Abelian property} allows us
to choose which site should be toppled ignoring the actual order in the
continuous-time dynamics. A \emph{toppling procedure} is a recipe that
specifies the next site to be toppled, usually (but not necessarily) in terms
of the outcome of previous topplings. The simpler examples of this general
strategy are gathered in \S \ref{sec:counting} and illustrate this principle
well. From \S \ref{sec:traps} to \S \ref{sec:multiscale} all proofs follow
this common setup, each one with its own specific elements.

Another method of analysis is the use of a particle-wise construction. This
construction allows different uses of the mass transport principle, coupling,
resampling, and ergodicity. These arguments are shown in \S
\ref{sec:particlewise}.

\subsection{Structure of these lecture notes}\label{sub:structure}

The main results and some open problems are stated at the beginning of each
section. The reader may want to first have a quick glance at each section to
have a sense of what is going on, then read \S \ref{sec:definitions} skipping
the proofs, and again skim through the other sections. After that, the advice
is to read the rest of the text linearly, maybe skipping computations.

The ordering of sections was decided taking into account relevance,
difficulty and interdependence. Proofs are intended to be self-contained and
have the level of detail of a research article. Although each section is
written assuming that the reader is familiar with the material presented
before it, reading the text linearly is not a strict requirement. To follow a
section in full detail, going through \S \ref{sec:definitions} is mandatory,
and going through \S \ref{sec:counting} is highly recommended for most parts.
The exception is \S \ref{sec:cycle} which uses a result from \S
\ref{sec:netflow} and arguments from \S \ref{sec:traps}.

These notes are organized as follows. In \S \ref{sec:definitions} we describe
the site-wise representation and state the main criteria to study the
absorbing-state phase transition, establishing the main tools used in
subsequent sections. In \S \ref{sec:counting} we provide the simplest
examples of a toppling procedure being used to prove fixation and activity by
verifying the criteria provided in \S \ref{sec:definitions}. In \S
\ref{sec:traps} we give a more sophisticated toppling procedure used to prove
lower bounds for $\zeta _{c}$ on $d=1$. In \S \ref{sec:netflow} we describe a
two-scale toppling procedure and an enumerative argument to prove upper
bounds for $\zeta _{c}$ (or lower bounds for $\lambda _{c}$) for unbiased
walks on $d=1$. In \S \ref{sec:cycle} we combine arguments from \S
\ref{sec:traps} and \S \ref{sec:netflow} plus a new argument based on a
certain urn process to study fast and slow fixation on a large ring. In \S
\ref{sec:weak} we describe a toppling procedure based on the notion of weak
and strong stabilizations to obtain a general lower bound for $\zeta _{c}$
valid in any dimension as well as upper bounds on $\zeta _{c}$ valid in the
transient case $d>2$. In \S \ref{sec:universality} we use a locally-finite
infinite-step parallel-update toppling procedure to show that the value of
$\zeta _{c}$ is independent of the choice of Poisson as the initial state. In
\S \ref{sec:multiscale} we briefly sketch a recursive multi-scale estimate
based on a toppling procedure that uses ideas of decoupling to prove lower
bounds for $\zeta _{c}$ unbiased walks on $d \geqslant 2$. In \S
\ref{sec:particlewise} we depart from the framework of site-wise
representation and use a different type of construction where particles are
labeled. We then devise other properties of the ARW such as mass
conservation, and prove an averaged condition for activity. In \S
\ref{sec:constructions} we prove that the continuous-time evolution is
well-defined and can be constructed explicitly using both the site-wise as
well as particle-wise constructions. We also prove the equivalence between
fixation and a condition on the site-wise representation. We then use a
hybrid construction to prove a comparison lemma used in the proof of the
averaged condition. Finally, in \S \ref{sec:extensions}, we describe how and
when these results were first proved, then comment on some of the arguments
that extend to other graphs, initial conditions and jump distributions, and
mention some of the arguments which have meaningful counter-parts for the
Stochastic Sandpile Model.

\section{Definitions and main tools}\label{sec:definitions}

In this section we define precisely the stochastic process to be studied,
describe the site-wise representation, and give conditions for fixation
and activity. We then state mass conservation and ergodicity properties
used later on, and conclude collecting frequently used notation.

\subsection{The stochastic process and notation}\label{sub:evolution}

We will denote by $\eta _{t}(x)$ the number and type of particles at site $x$
at time $t$, as follows. Let $\mathbb{N}_{0} = \{0,1,2,\dots \}$ and
$\mathbb{N}_{\mathfrak{s}}= \mathbb{N}_{0} \cup \{\mathfrak{s}\}$. The
configuration of the ARW at time $t\geqslant 0$ is given by $\eta _{t} \in
(\mathbb{N}_{\mathfrak{s}})^{\mathbb{Z}^{d}}$. The interpretation is that, at
time $t$, site $x$ contains $\eta _{t}(x)$ active particles if $\eta _{t}(x)
\in \mathbb{N}_{0}$, or one sleeping particle if $\eta _{t}(x)=\mathfrak{s}$.

We turn $\mathbb{N}_{\mathfrak{s}}$ into an ordered set by letting
$0<\mathfrak{s}<1<2<\cdots $. We also let $|\mathfrak{s}|=1$, so $|\eta
_{t}(x)|$ counts the number of particles regardless of their state. To add a
particle to a site, we define $\mathfrak{s}+1=2$, which represents the
$A+S\to 2A$ transition. We also define $1 \cdot \mathfrak{s}= \mathfrak{s}$
and $n \cdot \mathfrak{s}= n$ for $n \geqslant 2$, which represent the
transitions $A \to S$ and $2A \to A+S \to 2A$, respectively.

The process has a parameter $0<\lambda <\infty $ and evolves as follows. For
each site $x$, a clock rings at rate $(1+\lambda ) \eta _{t}(x)
\mathds{1}_{\eta _{t}(x) \ne \mathfrak{s}}$. When this clock rings, the
system goes through the transition $\eta \to \mathfrak{t}_{x\mathfrak{s}}\eta
$ with probability $\frac{\lambda }{1+\lambda }$, otherwise $\eta \to
\mathfrak{t}_{xy}\eta $ with probability $p(y-x)\frac{1}{1+\lambda }$. The
transitions are given by
\begin{equation*}
\mathfrak{t}_{x\mathfrak{s}}\eta (z) =
\begin{cases}
\eta (x) \cdot \mathfrak{s}, & z=x,
\\
\eta (z), & z \ne x,
\end{cases}
\qquad \mathfrak{t}_{xy}\eta (z) =
\begin{cases}
\eta (x)-1, & z=x,
\\
\eta (y)+1, & z=y,
\\
\eta (z), & \mbox{otherwise},
\end{cases}
\end{equation*}
and only occur if $\eta (x) \geqslant 1$. The operator
$\mathfrak{t}_{x\mathfrak{s}}$ represents a particle at $x$ trying to fall
asleep, which will effectively happen if there are no other particles present
at $x$. Otherwise, by definition of $n \cdot \mathfrak{s}$ the configuration
will not change. The operator $\mathfrak{t}_{xy}$ represents a particle
jumping from $x$ to $y$, where possible activation of a sleeping particle
previously found at $y$ is represented by the convention that
$\mathfrak{s}+1=2$. The case $\lambda = \infty $ is left aside until \S
\ref{sub:resampling}.

Given a translation-ergodic distribution $\nu $ on
$(\mathbb{N}_{\mathfrak{s}})^{\mathbb{Z}^{d}}$ with finite mean
\begin{equation*}
\int |\eta ({\boldsymbol{0}})|\nu ({\mathrm{d}}\eta )<\infty ,
\end{equation*}
there exists a process $(\eta _{t})_{t\geqslant 0}$ with the above transition
rates and such that $\eta _{0}$ has law $\nu $, see \S
\ref{sec:constructions} for a proof. We will use $\mathbf{P}^{\nu }$ to
denote the underlying probability measure in a space where this process is
defined.

\subsection{Site-wise representation}\label{sub:diaconis}

The site-wise representation enables us to exploit an algorithmic approach to
fixation. Due to particle exchangeability, this representation extracts
precisely the part of the randomness that is relevant for the absorbing-state
phase transition, focusing on the total number of jumps and leaving aside the
order in which they take place. It is suitable for studying path traces,
total occupation times, and final particle positions. But it precludes the
analysis of quantities for which the order and instant of the jumps do
matter, such as correlation functions, time needed for fixation, invariant
measures, etc.

In this subsection we do not deal with a time evolution, and $\eta $ denotes
simply an element of $(\mathbb{N}_{\mathfrak{s}})^{\mathbb{Z}^{d}}$ on which
one can perform certain operations. We say that site $x$ is \emph{unstable
for the configuration $\eta $} if $\eta (x) \geqslant 1$. An unstable site
$x$ can \emph{topple}, by applying $\mathfrak{t}_{xy}$ or
$\mathfrak{t}_{x\mathfrak{s}}$ to $\eta $.

We consider a \emph{field of instructions}
$\mathcal{I}=(\mathfrak{t}^{x,j})_{x\in \mathbb{Z}^{d}, j\in \mathbb{N}}$.
More precisely, at each site $x$ there is a sequence or \emph{stack} of
instructions $\mathfrak{t}^{x,1},\mathfrak{t}^{x,2},\dots $, such that each
one of the $\mathfrak{t}^{x,j}$ equals either $\mathfrak{t}_{x \mathfrak{s}}$
or $\mathfrak{t}_{xy}$ for some $y$. Later on we will choose $\mathcal{I}$
random, but for now $\mathcal{I}$ denotes a field that is fixed, and $\eta $
denotes an arbitrary configuration.

We also introduce the \emph{odometer field} $h=\big (h(x);x\in
\mathbb{Z}^{d}\big )$ which counts the number of topplings already performed
at each site, usually started from $h\equiv 0$. The toppling operation at $x$
is defined by
\begin{equation*}
\Phi _{x}(\eta ,h)= \big (\mathfrak{t}^{x,h(x)+1}\eta , h+ \delta _{x} \big )
.
\end{equation*}
We say that $\Phi _{x}$ is \emph{legal for $(\eta ,h)$} or simply \emph{legal
for $\eta $} if $\eta (x) \geqslant 1$. We may write $\Phi _{x}\eta $ as a
short for $\Phi _{x}(\eta ,0)$.

Sometimes it will be convenient to topple a site $x$ when it contains any
particle at all, even if $\eta (x)=\mathfrak{s}$. To achieve that, we define
on $\mathbb{N}_{\mathfrak{s}}$ the operations $\mathfrak{s}-1=0$ and
$\mathfrak{s}\cdot \mathfrak{s}=\mathfrak{s}$. We say that toppling $x$ is
\emph{acceptable} if $\eta (x) \geqslant \mathfrak{s}$, and we say that $\Phi
_{x}$ is \emph{acceptable for $(\eta ,h)$} if $\eta (x) \geqslant
\mathfrak{s}$. One should think of this operation as first forcing the
particle to activate and then toppling the site. A legal toppling is also
acceptable. We remark that it is not possible to define the operations
$0\cdot \mathfrak{s}$ and $0-1$ while preserving the local Abelian property
stated below, so these two operations will not be considered as acceptable.

\subsubsection*{Sequences of topplings and local properties}

Let $\alpha =(x_{1},\dots ,x_{k})$ denote a finite sequence of sites, and
define the operator $\Phi _{\alpha }= \Phi _{x_{k}}\Phi _{x_{k-1}}\cdots \Phi
_{x_{1}}$. We say that $\Phi _{\alpha }$ is \emph{legal} for $(\eta ,h)$ if
$\Phi _{x_{j}}$ is legal for $\Phi _{(x_{1},\dots ,x_{j-1})}(\eta ,h)$ for
each $j=1, \dots , k$. In this case we say that $\alpha $ is a \emph{legal
sequence} of topplings for $(\eta ,h)$. We define an \emph{acceptable
sequence} of topplings analogously. When $h\equiv 0$ we may write $\eta $
instead of $(\eta ,h)$. Let $m_{\alpha }=\big (m_{\alpha }(x);x\in
\mathbb{Z}^{d}\big )$ denote \emph{the odometer of $\alpha $}, given by
$m_{\alpha }(x)=\sum _{j}{\mathds{1}}_{x_{j}=x}$, so $m_{\alpha }$ is the
field which counts how many times each site $x$ appears in $\alpha $. We
write $m_{\alpha }\geqslant m_{\beta }$ if $m_{\alpha }(x) \geqslant m_{\beta
}(x)\ \forall \ x$, and $\tilde{\eta } \geqslant \eta $ if $\tilde{\eta }(x)
\geqslant \eta (x)\ \forall \ x$. We also write $(\tilde{\eta },\tilde{h})
\geqslant (\eta ,h)$ if $\tilde{\eta } \geqslant \eta $ and $\tilde{h} = h$.

We now state the four properties that make the ARW an Abelian model. The next
three lemmas are based on these properties alone rather than on specific
details of the ARW. Let $x$ be a site in $\mathbb{Z}^{d}$ and $\eta
,\tilde{\eta }$ be configurations.

\begin{property}[Local Abelian property]\label{property1}
If $\alpha $ and $\beta $ are acceptable sequences of topplings for the
configuration $\eta $, such that $m_{\alpha }=m_{\beta }$, then $\Phi
_{\alpha }\eta = \Phi _{\beta }\eta $.
\end{property}

\begin{property}[Mass comes from outside]\label{property2}
If $\alpha $ and $\beta $ are acceptable sequences of topplings for $\eta $
such that $m_{\alpha }(x) \leqslant m_{\beta }(x)$ and $m_{\alpha }(z)
\geqslant m_{\beta }(z)$ for all $z\ne x$, then $\Phi _{\alpha }\eta (x)
\geqslant \Phi _{\beta }\eta (x)$.
\end{property}

\begin{property}[Monotonicity of stability]\label{property3}
If site $x$ is unstable for the configuration $\eta $, and if $\tilde{\eta
}(x)\geqslant \eta (x)$, then $x$ is unstable for the configuration~$
\tilde{\eta }$.
\end{property}

\begin{property}[Monotonicity of topplings]\label{property4}
If $\tilde{\eta }\geqslant \eta $ and $\Phi _{x}$ is legal for $\eta $, then
$\Phi _{x}$ is legal for $\tilde{\eta }$ and $\Phi _{x} \tilde{\eta
}\geqslant \Phi _{x} \eta $.
\end{property}

\begin{proof}
The last two properties are immediate from the previous definitions. For
convenience, define the operators $n \oplus = n+1$ on
$\mathbb{N}_{\mathfrak{s}}$, as well as $n \ominus = n-1$ and $n\odot = n
\cdot \mathfrak{s}$ on $\mathbb{N}_{\mathfrak{s}}\setminus \{0\}$. With this
notation, whenever $n \ominus $ is acceptable (\textit{i.e.}\ $n\ne 0$) we
have $n \ominus \oplus = n \oplus \ominus $. Analogously, whenever $n \odot $
is acceptable (\textit{i.e.}\ $n\ne 0$) we have $n \odot \oplus = n \oplus
\odot $. Therefore, within any acceptable sequence of operations, replacing
$\ominus \oplus $ by $\oplus \ominus $ and $\odot \oplus $ by $\oplus \odot $
yields an acceptable sequence with the same final outcome.

We first prove Property~\ref{property1} as a warm up. Suppose that $m_{\alpha
}=m_{\beta }$. Notice that $\Phi _{\alpha }\eta (x)$ is given by $\eta (x)$
followed by a sequence of $\oplus $, $\ominus $ and $\odot $'s. The number of
times each operator appears is determined by $\mathcal{I}$ and $m_{\alpha }$
only, hence it is the same number when we write $\Phi _{\beta }\eta (x)$ as
$\eta (x)$ followed by a sequence of $\oplus $, $\ominus $ and $\odot $'s.
Their actual order depends on the sequence, but the internal order of the
$\ominus $'s and $\odot $'s is determined by $(\mathfrak{t}^{x,j})_{j}$ and
is thus the same for both $\Phi _{\alpha }\eta (x)$ and $\Phi _{\beta }\eta
(x)$. As a consequence, we can apply the above identities to move the $\oplus
$'s to the left, yielding then identical sequences for $\Phi _{\alpha }\eta
(x)$ and $\Phi _{\beta }\eta (x)$, proving the claimed property.

For Property~\ref{property2} we make a similar observation. Suppose
$m_{\alpha }(x) \leqslant m_{\beta }(x)$ and $m_{\alpha }(z) \geqslant
m_{\beta }(z)$ for $z\ne x$. Again $\Phi _{\alpha }\eta (x)$ is given by
$\eta (x)$ followed by a number of $\oplus $, $\ominus $ and $\odot $'s. The
number of times that operator $\oplus $ appears depends on $m_{\alpha }(z),
z\ne x$, and is thus bigger in $\Phi _{\alpha }\eta (x)$ than in $\Phi
_{\beta }\eta (x)$. The number of times that operators $\ominus $ and $\odot
$ appear depend on $m_{\alpha }(x)$, and is thus smaller than in $\Phi
_{\beta }\eta (x)$. Pushing the $\oplus $'s to the left as before, we get
that $\Phi _{\alpha }\eta (x)$ is written in the same way as $\Phi _{\beta
}\eta (x)$, perhaps with a few extra $\oplus $'s in the beginning, and a few
missing $\ominus $ and $\odot $'s in the end, so $\Phi _{\alpha }\eta (x)
\geqslant \Phi _{\beta }\eta (x)$.
\end{proof}

\subsubsection*{Stabilization via sequential topplings}

Let $V$ denote a finite subset of $\mathbb{Z}^{d}$. We say that a
configuration $\eta $ is \emph{stable in $V$} if every $x\in V$ is stable for
$\eta $. We say that $\alpha $ is \emph{contained in $V$}, and write $\alpha
\subseteq V$, if every $x$ appearing in $\alpha $ is an element of $V$. We
say that $\alpha $ \emph{stabilizes $\eta $ in $V$} if $\alpha $ is
acceptable for $\eta $ and $\Phi _{\alpha }\eta $ is stable in $V$.

\begin{lemma}\label{lemma:lap}
If $\alpha $ is an acceptable sequence of topplings that stabilizes $\eta $
in $V$, and $\beta \subseteq V$ is a legal sequence of topplings for $\eta $,
then $m_{\beta }\leqslant m_{\alpha }$.
\end{lemma}

\begin{proof}
Let $\beta \subseteq V$ be legal and $m_{\alpha }\ngeqslant m_{\beta }$.
Write $\beta =(x_{1},\dots ,x_{k})$ and $\beta ^{(j)}=(x_{1},\dots ,x_{j})$
for $j\leqslant k$. Let $\ell =\max \{ j : m_{\beta ^{(j)}} \leqslant
m_{\alpha } \}<k$ and $y=x_{\ell +1}\in V$. Since $\beta $ is legal, $y$ is
unstable in $\Phi _{\beta ^{(\ell )}}\eta $. But $m_{\beta ^{(\ell
)}}\leqslant m_{\alpha }$ and $m_{\beta ^{(\ell )}}(y)= m_{\alpha }(y)$. By
the Properties~\ref{property2} and~\ref{property3}, $y$ is unstable for $\Phi
_{\alpha }\eta $ and therefore $\alpha $ does not stabilize $\eta $ in $V$.
\end{proof}

Let $V \subseteq \mathbb{Z}^{d}$ be a finite set. We define \emph{the
odometer of $\eta $ in $V$} by
%
\begin{equation}\label{eq:laplower}
m_{V,\eta } = \sup _{\beta \subseteq V\text{ legal}} m_{\beta }
\end{equation}
the supremum being taken over sequences $\beta $ which are legal for
$\eta $ and contained in $V$. Lemma~\ref{lemma:lap} says that
%
\begin{equation}\label{eq:lapupper}
m_{V,\eta } \leqslant m_{\alpha
}\end{equation}
for every acceptable sequence $\alpha $ stabilizing $\eta $ in $V$. These two
together provide very good sources of lower and upper bounds for $m_{V,\eta
}$. Note that the sequence $\alpha $ need not be legal, nor contained in $V$.
This allows us to choose a convenient sequence of topplings, even wake up
some particles if we wish, if we are looking for upper bounds.

\begin{lemma}[Abelian property]
If $\alpha $ and $\beta $ are both legal toppling sequences for $\eta $ that
are contained in $V$ and stabilize $\eta $ in $V$, then $m_{\alpha }=m_{\beta
}=m_{V,\eta }$. In particular, $\Phi _{\alpha }\eta =\Phi _{\beta }\eta $.
\end{lemma}

\begin{proof}
Applying~\eqref{eq:laplower},~\eqref{eq:lapupper} and Lemma~\ref{lemma:lap}:
$m_{\beta }\leqslant m_{V,\eta } \leqslant m_{\alpha }\leqslant m_{\beta }$.
\end{proof}

If there is an acceptable sequence $\alpha $ that stabilizes $\eta $ in $V$,
then there is a legal sequence $\beta $ contained in $V$ that also stabilizes
$\eta $ in $V$. Indeed, if one tries to perform legal topplings in $V$
indefinitely, on the one hand one has to eventually stop since $V$ is finite
and there is a finite upper bound $\sum _{x} m_{\alpha }(x)$ for the total
number of topplings, and on the other hand this procedure only stops if there
are no more unstable sites in $V$.

\begin{lemma}[Monotonicity]\label{lemma:monotonicity}
If $V\subseteq \tilde{V}$ and $\eta \leqslant \tilde{\eta }$, then
$m_{V,\eta }\leqslant m_{\tilde{V},\tilde{\eta }}$.
\end{lemma}
\begin{proof}
Let $\beta \subseteq V$ be legal for $\eta $. By successively applying
Properties~\ref{property3} and~\ref{property4}, $\beta $ is also legal
for $\tilde{\eta }$. Since $\beta \subseteq \tilde{V}$, the inequality follows
directly from definition~\eqref{eq:laplower}.
\end{proof}

By monotonicity, the limit
%
\begin{equation}\label{eq:odometer}
m_{\eta }= \lim _{V \uparrow \mathbb{Z}^{d}} m_{V,\eta }
\end{equation}
exists and does not depend on the particular sequence
$V \uparrow \mathbb{Z}^{d}$. The limit is also given by the supremum of
$m_{V,\eta }$ over finite $V$. A configuration $\eta $ on
$\mathbb{Z}^{d}$ is said to be \emph{stabilizable} if
$m_{\eta }(x)<\infty $ for every $x\in \mathbb{Z}^{d}$. In the next subsection
we relate the above concepts to the question of fixation vs.\ activity
for the continuous-time process.

\subsection{Criteria for fixation and activity}\label{sub:conditions}

Assume the initial configuration $\eta _{0} \in
(\mathbb{N}_{\mathfrak{s}})^{\mathbb{Z}^{d}}$ has a translation-ergodic
distribution denoted by $\nu $. Assume also that the support of the jump
distribution $p(\cdot )$ generates the group $\mathbb{Z}^{d}$ and not a
subgroup.

Consider now the field of instructions $\mathcal{I}$ being random, and
distributed as follows. For each $x\in \mathbb{Z}^{d}$ and $j\in \mathbb{N}$,
choose $\mathfrak{t}^{x,j}$ as $\mathfrak{t}_{xy}$ with probability
$\frac{p(y-x)}{1+\lambda }$ or $\mathfrak{t}_{x\mathfrak{s}}$ with
probability $\frac{\lambda }{1+\lambda }$, independently over $x$ and $j$.
Now sample $\eta _{0}$ according to $\nu $ and independently of
$\mathcal{I}$, and let $\mathbb{P}^{\nu }$ denote the probability measure in
a space where $\mathcal{I}$ and $\eta _{0}$ are defined.

Recall that $\mathbf{P}^{\nu }$ denotes the probability measure in whatever
probability space where the process $(\eta _{t})_{t\geqslant 0}$ described in
\S \ref{sub:evolution} is defined. The following result is proved in \S
\ref{sub:condbu}.

\begin{theorem}\label{thm:equivalence}
$\mathbf{P}^{\nu }\big (\text{fixation of }(\eta _{t})_{t\geqslant 0} \big
)=\mathbb{P}^{\nu }(m_{\eta _{0}}({\boldsymbol{0}})<\infty ) = 0 \text{ or }
1$.
\end{theorem}

Combining Theorem~\ref{thm:equivalence} with~\eqref{eq:odometer} and
Lemma~\ref{lemma:monotonicity} we get the following.
%
\begin{corollary}
If the condition
%
\begin{equation}\label{eq:conditionb}
\sup _{k} \inf _{V \, \text{finite}} \mathbb{P}^{\nu }( m_{V,\eta _{0}}({
\boldsymbol{0}}) \leqslant k ) > 0
\end{equation}
is satisfied, the system a.s.\ fixates. If the condition
%
\begin{equation}\label{eq:conditionu}
\inf _{k} \sup _{V \, \text{finite}} \mathbb{P}^{\nu }( m_{V,\eta _{0}}({
\boldsymbol{0}}) \geqslant k ) > 0
\end{equation}
is satisfied, the system a.s.\ stays active.
\end{corollary}

A typical usage of Condition~\eqref{eq:conditionb} is to rely on~\eqref{eq:lapupper}
and use the odometer of an acceptable stabilizing sequence of topplings
as a stochastic upper bound for $m_{V,\eta _{0}}$. Likewise, a typical
usage of Condition~\eqref{eq:conditionu} is to rely on~\eqref{eq:laplower}
and use the odometer of a legal sequence of topplings as a stochastic lower
bound for $m_{V,\eta _{0}}$.

In \S \ref{sub:conditione} we will prove the following sufficient condition
for activity.

\begin{theorem}\label{thm:conditione}
Let $M_{n}$ count the number of particles that jump out of $V_{n} =
\{-n,\dots ,n\}^{d}$ when $V_{n}$ is stabilized via legal topplings, so
particles are ignored after leaving $V_{n}$. If $\eta _{0}$ is i.i.d.\ and
the condition
%
\begin{equation}\label{eq:conditione}
\limsup _{n} \dfrac{\mathbb{E}M_{n}}{|V_{n}|}>0
\end{equation}
is satisfied, then the system a.s.\ stays active.
\end{theorem}

Similar to Condition~\eqref{eq:conditionu}, a typical usage of
Condition~\eqref{eq:conditione} is to rely on the Abelian property and use
the expected number of particles exiting $V_{n}$ during some specific legal
sequence of topplings $\beta \subseteq V_{n}$ as a stochastic lower bound for
$\mathbb{E}M_{n}$.

An important property of the ARW is that $\zeta _{c}$ has a sharp definition.

\begin{theorem}[Uniqueness of the critical density]\label{thm:universality}
Given the dimension $d$, jump distribution $p(\cdot )$, and sleep rate
$\lambda $, there is a number $\zeta _{c}$ such that, for every
translation-ergodic distribution $\nu $ supported on
$(\mathbb{N}_{0})^{\mathbb{Z}^{d}}$ with average density $\zeta $, the ARW
dynamics satisfies
\begin{equation*}
\mathbf{P}^{\nu }(\text{system stays active}) =
\begin{cases}
0,& \zeta <\zeta _{c},
\\
1,& \zeta >\zeta _{c}.
\end{cases}
\end{equation*}
\end{theorem}

This property will be proved in~\S \ref{sec:universality}. Mathematically, it
is useful because every statement about bounds for $\zeta _{c}$ can be proved
assuming an i.i.d.\ initial state with whatever marginal distribution is more
convenient. In fact, we can even take non-i.i.d.\ distributions if that
helps.

We conclude with a monotonicity property.

\begin{theorem}\label{thm:monotone}
The critical density $\zeta _{c}$ is non-decreasing in $\lambda $.
\end{theorem}

The proof is given in \S \ref{sub:lambda}.

\begin{problem}
Show that $\zeta _{c}$ is continuous and strictly increasing in
$\lambda $.
\end{problem}

\subsection{Factors, ergodicity and the mass transport principle}\label{sub:mtperg}

We will not use these tools directly until \S \ref{sec:universality}. Let
$f:\mathbb{Z}^{d} \times \mathbb{Z}^{d} \to \mathbb{R}_{+}$ be a
translation-invariant random function, that is, $f(x,y;{\omega })=f(\theta
x,\theta y;\theta {\omega })$ for every translation $\theta $ of
$\mathbb{Z}^{d}$. Here $\omega $ can be any random process with
translation-invariant distribution.

The \emph{mass transport principle} is given by the identity
\begin{equation*}
\mathbb{E}\big [\sum _{y} f(x,y)\big ]=\mathbb{E}\big [\sum _{y} f(y,x)
\big ] .
\end{equation*}
Informally, the mass transport principle says that on average the amount of
mass transmitted from a site $x$ is equal to the amount of mass entering $x$.
It may seem like nothing but an obvious identity, but its strength lies in
its versatility, since it holds for every translation-invariant function. The
proof consists in re-indexing the sum and using translation invariance.
See~\cite[Chapter~8]{LyonsPeres16} for applications and generalizations to
other settings.

\medskip
Suppose a field $\tilde{\omega } = g(\omega )$ is a translation-covariant
function of $\omega $. That is, for all $\omega $ for which $g(\omega )$ is
defined, $g(\theta \omega )$ is also defined and $g(\theta \omega )=\theta
g(\omega )$. Then we say that $\tilde{\omega }$ is a \emph{factor} of $\omega
$. Note that factors inherit properties such as translation invariance and
translation ergodicity.

\medskip
Finally, we mention how pairs of ergodic and mixing fields behave together.
Suppose $\omega ^{1}$ and $\omega ^{2}$ are independent. Suppose $\omega
^{1}$ is mixing (with respect to translations). If $\omega ^{2}$ is mixing,
then the pair $(\omega ^{1},\omega ^{2})$ is mixing. If $\omega ^{2}$ is
ergodic, then $(\omega ^{1},\omega ^{2})$ is ergodic. In particular, the pair
$(\eta _{0},\mathcal{I})$ introduced in the previous subsection is
translation-ergodic. See~\cite[Thm.~2.25]{KerrLi16}.

\subsection{Frequently used notation}

Let $\|\eta \|_{V}$ denote the total number of particles in the box $V$ in
the configuration $\eta $, given by $\|\eta \|_{V} = \sum _{x\in V} |\eta
(x)|$.

The discrete ball of radius $r$ centered at $y$ is denoted $B_{r}^{y}$ with
$B_{r}=B_{r}^{\boldsymbol{0}}$. Any choice of norm to define the balls would
work fine, but for the sake of aesthetics let us fix $B_{r}=\{-r,\dots
,r\}^{d}$.

\medskip
Normally, the letters $\eta $ and $\xi $ will denote configurations, $\zeta $
denotes density, while $\alpha $ and $\beta $ denote sequences of topplings.
The letter $\nu $ usually denotes a probability measure on
$(\mathbb{N}_{\mathfrak{s}})^{\mathbb{Z}^{d}}$ or most often on
$(\mathbb{N}_{0})^{\mathbb{Z}^{d}}$. The letters $L$ and $r$ are usually used
to denote sizes of some domains, whereas $i,j,k$ are generally used to index
sites, events, particles, etc. The letter $n$ may have either of these uses.

The boxes $V_{n} \uparrow \mathbb{Z}^{d}$ are not necessarily the same in
different parts of the text, and may in fact be random. Most of the time they
equal $B_{n}$ but we still use $V_{n}$ to indicate that they are being used
in the context of \S \ref{sub:conditions}.

The letters $M$ and $N$ usually denote random variables defined by counting
how many events of a family occur, such as how many particles do this and
that, how many steps go wrong in a certain procedure, etc. We normally use
$|\cdot |$ to denote the discrete volume of a finite subset of
$\mathbb{Z}^{d}$. The symbol $\#$ rarely appears, technically it also denotes
cardinality but we use it when such cardinality comes from counting random
objects. When we have to name events, we will mostly use the letter
$\mathcal{A}$ and maybe add indices.

Most of the time we deal with constructions based on independent random
variables defined explicitly, and we use the letter $\mathbb{P}$ for
probability (hence $\mathbb{E}$ for expectation). We use $\mathbf{P}$ to
emphasize that certain statements refer to whatever probability space where
the particle system $(\eta _{t})_{t\geqslant 0}$ is defined. We often say
``a.s.'' and ``with high probability'' without bothering about probability
space formalities that are hardly relevant.

The letters $\kappa $, $\delta $ and $\varepsilon $ usually denote a large
fixed number, a small fixed number, and an arbitrarily small number. The
letter $K$ is an integer parameter in some constructions or toppling
procedures, sometimes it is not fixed beforehand but instead made larger and
larger throughout each proof.

\section{Counting arguments}\label{sec:counting}

From the previous section, the Abelian property says that the odometer of a
given configuration in a given region can be obtained through any legal
sequence $\alpha $ of topplings. Note that such a sequence need not be given
beforehand: it can in fact be constructed algorithmically as the
configuration evolves.

This opens the door to proofs consisting of (i) the prescription of a
\emph{toppling procedure} followed by (ii) a probabilistic analysis to check
Conditions~\eqref{eq:conditionb}, \eqref{eq:conditionu}
or~\eqref{eq:conditione} in terms of such procedure. In this section we will
see some simple instances of this approach being applied.

\medskip
As a warm-up example for the use of Condition~\eqref{eq:conditionu} we
prove the following.

\begin{theorem}\label{thm:firstcount}
For $d=1$ and initial state i.i.d.\ with mean $\zeta =1$ and positive
variance, the system a.s.\ stays active.
\end{theorem}

We then show the simplest argument that uses a toppling procedure based on
the Abelian property to check Conditions~\eqref{eq:conditionb}
or~\eqref{eq:conditionu}, and prove the following. A \emph{directed walk} on
$d=1$ is the one with $p(+1)=1$.

\begin{theorem}\label{thm:zetacexact}
For $d=1$ and directed walks, $\zeta _{c}=\frac{\lambda }{1+\lambda }$. If
the initial state is i.i.d.\ with critical density $\zeta =\zeta _{c}$ and
positive variance, then the system a.s.\ stays active.
\end{theorem}

The proof of the previous theorem breaks down in case the walks are not
totally directed, because even a small probability of jumping left would
require some control on the interaction among particles which wander in the
wrong direction. The following is proved via an argument that uses
Condition~\eqref{eq:conditione} and Theorem~\ref{thm:universality}.

\begin{theorem}\label{thm:taggi}
For $d \geqslant 1$ and biased walks, $\zeta _{c}<1$ for every $\lambda
<\infty $ and $\zeta _{c}\to 0$ as $\lambda \to 0$.
\end{theorem}

In proving the above we get a quantitative upper bound for
$\zeta _{c}$ which gives the following corollary.

\begin{corollary}\label{cor:directed}
For directed walks in any dimension $d$, we have $\zeta _{c} \leqslant
\frac{\lambda }{1+\lambda }$.
\end{corollary}

Together with Theorem~\ref{thm:subcritstauffertaggi}, this implies that
$\zeta _{c} = \frac{\lambda }{1+\lambda }$ for directed walks.

\medskip
We now proceed to the proofs.

\begin{proof}[Proof of~Theorem~\ref{thm:firstcount}]
Let $\zeta =1$. By the CLT, the probability that $\eta _{0}$ contains at
least $L+2 \sqrt{L}$ particles in $V_{L}=[0,L]$ is at least $2 \delta >0$,
for all large $L$. On this event, at least $2\sqrt{L}$ particles will visit
either $x={\boldsymbol{0}}$ or $x=L$ when we stabilize $[0,L]$.

Therefore, using a union bound, translation invariance and monotonicity,
\begin{equation*}
2\delta \leqslant \mathbb{P}\big (m_{V_{L},{\eta _{0}}}({
\boldsymbol{0}}) \geqslant \sqrt{L}\big ) + \mathbb{P}\big (m_{V_{L},{
\eta _{0}}}(L) \geqslant \sqrt{L}\big ) \leqslant 2 \mathbb{P}\big (m_{
\eta _{0}}({\boldsymbol{0}}) \geqslant \sqrt{L}\big ).
\end{equation*}
Since this is true for all large $L$, Condition~\eqref{eq:conditionu} is
satisfied and therefore the system a.s.\ stays active.
\end{proof}

\begin{proof}[Proof of~Theorem~\ref{thm:zetacexact}]
Write $V_{L}=\{-L,\dots ,L\}$. We will consider $m_{V_{L},{\eta
_{0}}}({\boldsymbol{0}})$ and see in which cases it satisfies
Condition~\eqref{eq:conditionb} or~\eqref{eq:conditionu} as $L\to \infty $.

We will describe a legal sequence of topplings that stabilizes $V_{L}$ by
exhausting one site after the other, from left to right. The sequence of
topplings is itself random, since it is given by an algorithm which decides
the next site to topple in terms of outcome of the previous topplings.

We start by toppling site $x=-L$ until each of the $\eta _{0}(-L)$ particles
either moves to $x=-L+1$ or falls asleep. All particles but the last one have
to jump (possibly after a few frustrated attempts to sleep). The last one may
sleep or jump. Let $Y_{0}^{L}$ denote the indicator of the event that the
last particle remains sleeping at $x=-L$. Conditioned on $\eta _{0}(-L)$, the
distribution of $Y_{0}^{L}$ is Bernoulli with parameter $\frac{\lambda
}{1+\lambda }$ (in case $\eta _{0}(-L)=0$, sample $Y_{0}^{L}$ independently
of everything else). The number of particles which jump from $x=-L$ to
$x=-L+1$ is given by $N_{1}^{L}:=[\eta _{0}(-L)-Y_{0}^{L}]^{+}$. See
Figure~\ref{figcounting}.

\begin{figure}
\includegraphics{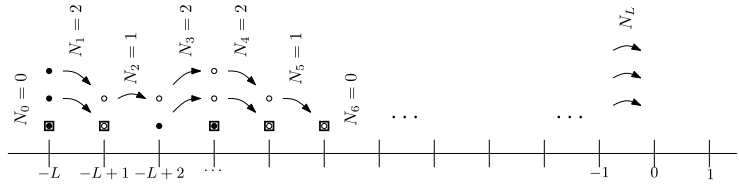}%
\caption{Stabilizing $[-L,0]$ from left to right. Disks represent particles
initially present at each site, and circles represent particles that arrive
from the left. Boxes indicate particles which fell asleep after they were
left alone.}\label{figcounting}
\end{figure}

Note that, after stabilizing $x=-L$, there are $N_{1}^{L}+\eta _{0}(-L+1)$
particles at $x=-L+1$. We now topple site $x=-L+1$ until it is stable, and
denote by $Y_{1}^{L}$ the indicator of the event that the last particle
remains sleeping at $x=-L+1$. The number of particles which jump from
$x=-L+1$ to $x=-L+2$ is given by $N_{2}^{L}:=[N_{1}^{L}+\eta
_{0}(-L+1)-Y_{1}^{L}]^{+}$. By iterating this procedure, the number
$N_{i+1}^{L}$ of particles which jump from $x=-L+i$ into $x=-L+i+1$ after
stabilizing $x=-L,-L+1,\dots ,-L+i$ is given by
\begin{equation*}
N_{i+1}^{L}=[N_{i}^{L}+\eta _{0}(-L+i)-Y_{i}^{L}]^{+},
\end{equation*}
where $N_{0}^{L}=0$. The number of particles which jump into
${\boldsymbol{0}}$ while stabilizing $V_{L}$ equals $N_{L}^{L}$. After that,
we stabilize $x={\boldsymbol{0}}$ and then $\{1,\dots ,L\}$, but the latter
no longer affects $m_{V_{L},{\eta _{0}}}({\boldsymbol{0}})$.

Note that the sequence $(N_{i}^{L})_{i=0,1,\dots ,L}$ is distributed as a
random walk on $\mathbb{N}_{0}$, with independent increments distributed as
${\eta _{0}}(x)-Y$, reflected at $0$. So the relevant quantity is
\begin{equation*}
\mathbb{E}[{\eta _{0}}(-L+k)-Y_{k}^{L}] = \zeta - \frac{\lambda }{1+\lambda }
.
\end{equation*}
If $\zeta < \frac{\lambda }{1+\lambda }$, the walk is positive recurrent.
This implies that the family $\{N_{L}^{L}\}_{L \in \mathbb{N}}$ is
stochastically bounded. Since $m_{V_{L},\eta _{0}}({\boldsymbol{0}})$ is
conditionally distributed as a sum of $N_{L}^{L}+\eta _{0}({\boldsymbol{0}})$
independent geometric variables with parameter $\frac{1}{1+\lambda }$,
Condition~\eqref{eq:conditionb} holds, and thus the system a.s.\ fixates. If
$\zeta > \frac{\lambda }{1+\lambda }$, the walk is transient, and
$\mathbb{P}\big (N_{L}^{L} \geqslant \frac{1}{2}(\zeta - \frac{\lambda
}{1+\lambda })L\big )$ is large for large $L$, so
Condition~\eqref{eq:conditionu} holds and the system a.s.\ stays active. If
$\zeta = \frac{\lambda }{1+\lambda }$, the walk is null-recurrent, so
$N_{L}^{L} \to \infty $ in probability as $L \to \infty $, which again
implies Condition~\eqref{eq:conditionu} and therefore the system a.s.\ stays
active.
\end{proof}

The proof of Theorem~\ref{thm:taggi} uses a toppling procedure in which
particles help each other progressing in the right direction. The idea is to
keep the particles spread so that they form a safe zone where the last
particle can transit without sleeping. This way the particle has a reasonable
chance of catching up with the others without falling asleep outside the safe
zone.

For ${\boldsymbol{v}}\in \mathbb{R}^{d}$, let $\mathcal{H}_{\boldsymbol{v}}=
\{ z \in \mathbb{Z}^{d} : z \cdot { \boldsymbol{v}}\leqslant 0\}$. Consider a
continuous-time random walk which starts at ${\boldsymbol{0}}$, jumps at rate
$1$ according to $p(\cdot )$, and is killed at rate $\lambda $ when it is at
$\mathcal{H}_{\boldsymbol{v}}$. Denote by $F_{\boldsymbol{v}}(\lambda )$ the
probability that this walk is ever killed. If ${\boldsymbol{v}}$ is such that
${\boldsymbol{v}}\cdot \sum _{y} yp(y)>0$, we have that
$F_{\boldsymbol{v}}(\lambda ) < 1$ for every $\lambda <\infty $, and moreover
$F_{\boldsymbol{v}}(\lambda ) \to 0$ as $\lambda \to 0$.
Theorem~\ref{thm:taggi} thus follows from the following proposition.

\begin{proposition}\label{prop:flambda}
For any dimension $d$ and any ${\boldsymbol{v}}\in \mathbb{R}^{d}$, $\zeta
_{c} \leqslant F_{\boldsymbol{v}}(\lambda )$.
\end{proposition}

\begin{proof}
By Theorem~\ref{thm:universality} we can assume that the initial state is
i.i.d.\ Bernoulli with parameter $\zeta $. We will show that
Condition~\eqref{eq:conditione} holds if $\zeta > F_{\boldsymbol{v}}(\lambda
)$.

Let $V_{n}=B_{n}$ and consider a labeling $V_{n}=\{x_{1},\ldots ,x_{r}\}$,
where $r=|V_{n}|$ and $x_{1}\cdot {\boldsymbol{v}}\leqslant \cdots \leqslant
x_{r}\cdot { \boldsymbol{v}}$. For $i=0,\ldots ,r$, let
$A_{i}=\{x_{i+1},\ldots ,x_{r}\}\subseteq V_{n}$.

The toppling procedure consists of $r$ steps. For $i=1,\ldots ,r$, Step~$i$
goes as follows. If there is a particle at $x_{i}$ then we do a sequence of
legal topplings, starting by a toppling at $x_{i}$, then at the location
where this toppling pushed the particle to, and so on until this particle
either (i) exits $V_{n}$, (ii) reaches an unoccupied site in $A_{i}$, or
(iii) falls asleep in $V_{n} \setminus A_{i}$. If case~(iii) occurs, we say
that the particle is \emph{left behind}.

By induction we can see that, for $i=1,\ldots ,r$, after Step $i-1$, each
site $x\in A_{i-1}$ (including $x_{i}$) is either unoccupied or contains
exactly one active particle (recall that we are starting with Bernoulli).
Hence, at most one particle is left behind at each step.

Since $V_{n}\setminus A_{i} \subseteq x_{i}+\mathcal{H}_{\boldsymbol{v}}$, at
each step the probability of leaving a particle behind is bounded from above
by $F_{\boldsymbol{v}}(\lambda )$. Moreover, during this procedure each
particle either ends up exiting $V_{n}$ or being left behind. Indeed, if a
given step ends due to condition (ii), the corresponding particle will be
moved again in a later step. Let $N_{n}$ denote the total number of particles
left behind, and notice that $M_{n} \geqslant \|\eta _{0}\|_{V_{n}} - N_{n}$.
Therefore, $\mathbb{E}M_{n} \geqslant \zeta \, |V_{n}| - F_{\boldsymbol{v}}(
\lambda ) \, |V_{n}|$. This completes the proof.
\end{proof}

\begin{proof}[Proof of Corollary~\ref{cor:directed}]
A jump distribution $p(\cdot )$ being directed means that, for some
${\boldsymbol{v}}\in \mathbb{R}^{d}$, we have $p(y)=0$ for every $y$ such
that ${\boldsymbol{v}}\cdot y \leqslant 0$. In this case, a random walk that
starts at ${\boldsymbol{0}}$ and jumps according to $p(\cdot )$ only visits
$\mathcal{H}_{\boldsymbol{v}}$ once (at time zero), and so it follows from
definition of $F_{\boldsymbol{v}}$ that $F_{\boldsymbol{v}}(\lambda
)=\frac{\lambda }{1+\lambda }$. The result then follows from
Proposition~\ref{prop:flambda}.
\end{proof}

\section{Exploring the instructions in advance}\label{sec:traps}

In this section we introduce a toppling procedure to prove phase transition
in dimension $d=1$ by using~\eqref{eq:lapupper} to check
Condition~\eqref{eq:conditionb}. Compared to the previous sections, it
introduces a novel element which will be used in \S \ref{sec:cycle}--\S
\ref{sec:multiscale}: the heavy usage of acceptable topplings as a way to
enforce activity and conveniently displace some particles away from their
current position.

An element which is specific to the argument presented in this section and
which will be absent in \S \ref{sec:weak}--\S \ref{sec:multiscale} is that
the toppling procedure used here is not ``Markovian.'' By this we mean that
in order to decide on the next topplings we use more information than just
the outcome of the previous ones.

\begin{theorem}\label{thm:rs12}
For $d=1$, for every $\lambda >0$, we have $\zeta _{c} \geqslant
\frac{\lambda }{1+\lambda }$.
\end{theorem}

The above lower bound is a particular case of
Theorem~\ref{thm:subcritstauffertaggi}. We find it instructive to understand
the proof given in this section, in particular because it will be used in \S
\ref{sec:cycle}, it provides Corollary~\ref{cor:depends}, and it can be
adapted to study the ARW on other graphs such as regular trees. We now move
to proving the theorem.

\subsubsection*{General strategy}

The idea is to try and stabilize all the particles from $\eta _{0}$,
following the instructions in $\mathcal{I}$, with the help of acceptable
topplings.

After describing the procedure, we will show that, whenever it is successful,
it implies that $m_{\eta _{0}}({\boldsymbol{0}})=0$. We finally show that the
procedure is successful with positive probability if $\zeta < \frac{\lambda
}{1+\lambda }$, implying Condition~\eqref{eq:conditionb}.

\subsubsection*{Description of the toppling procedure}

We will try to find a trap for one particle at a time. To find the trap, we
launch an exploration that reveals some instructions in $\mathcal{I}$ until
it identifies a suitable trap. To do that, the exploration follows the path
that the particle would perform if we always toppled the site it occupies,
and stop when the trap has been chosen. In the absence of a suitable trap, we
declare the procedure to have failed.

\begin{figure}
\includegraphics{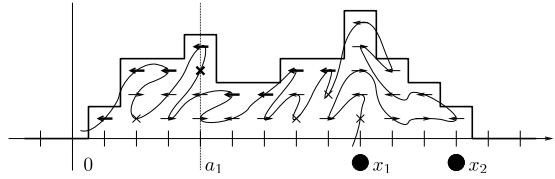}%
\caption{First exploration path. It starts at position $x_{1}$ of the first
particle and stops when it reaches the origin. The horizontal axis represents
the lattice, and above each site $x$ there is a stack of instructions
$(\mathfrak{t}^{x,j})_{j}$. The bold arrows indicate the last jump found at
each site $x \in [ 1,x_{1}-1 ]$, and the bold cross indicates a sleep
instruction found just before the last jump, this being the leftmost such
cross, whose location defines $a_{1}$.}\label{fig1arw}\vspace{-7pt}
\end{figure}

An important issue is that some of the explored instructions are actually not
going to be used by the corresponding particle which will instead remain
sleeping at the trap. In particular, the instructions revealed in one step
could interfere with the conditional distribution of subsequent steps. For
this reason, we will call \emph{corrupted} the sites where instructions have
been revealed but not used.

If there is a particle at ${\boldsymbol{0}}$, we declare the procedure
unsuccessful and stop. Otherwise, label the initial positions of the
particles on $\mathbb{Z}$ by $\cdots \leqslant x_{-3} \leqslant x_{-2}
\leqslant x_{-1} < 0 < x_{1} \leqslant x_{2} \leqslant x_{3} \leqslant \cdots
$.

Let $a_{0}=0$. We now describe the trapping of each particle. Suppose that
the first $k-1$ traps have been successfully set up at positions
$0<a_{1}<a_{2}<\cdots <a_{k-2}<a_{k-1} \leqslant x_{k-1}$, and suppose also
that the interval $[ 0,a_{k-1} ]$ contains all traps and corrupted sites
found in the previous steps.

We now launch an \emph{exploration}. Starting at $x_{k}$, the explorer
examines the next unexplored instruction at its current position, and moves
to the site indicated by such instruction. If it is a sleep instruction, the
explorer does not move. Repeat this indefinitely, and stop upon reaching
$a_{k-1}$.

Next we \emph{set up the trap}. Notice that, during the $k$-th exploration,
a.s.\ each site is visited a finite number of times. Moreover, the explorer
either stops at $a_{k-1}$ or drifts to $+\infty $. In the latter case we
simply take $a_{k}=a_{k-1}$. Suppose the former occurs. In this case, the
explorer visits every site in $D_{k}=[ a_{k-1}+1,x_{k}-1 ]$. Moreover, the
last instruction explored at each site is a jump to the left, see
Figures~\ref{fig1arw}~and~\ref{fig2arw}. Note that, for each $x\in D_{k}$,
the second last instruction may or may not be a sleep instruction. We take
$a_{k}$ as the leftmost site at which the second last instruction explored
was a sleep instruction, and call this second last instruction the $k$-th
trap. If there is no such site in $D_{k}$, we declare the entire procedure
unsuccessful and stop.\eject

\begin{figure}
\includegraphics{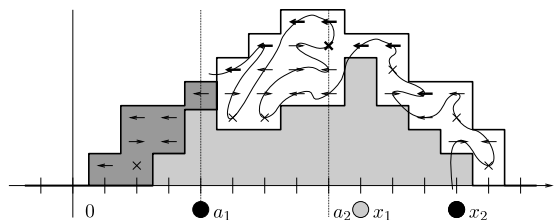}%
\caption{Second exploration path. It starts at position $x_{2}$ of the second
particle and stops when it reaches $a_{1}$. The regions in gray indicate the
instructions already examined by the first explorer. The dark gray contains
instructions examined but not used, whose locations determine the set of
corrupted sites.} \label{fig2arw}
\end{figure}

So the trap is a sleep instruction found immediately before the last
instruction, which is a jump to the left. Hence, we know that the exploration
path has not been to the right of $a_{k}$ after it revealed the instruction
that we now declare as being the trap. Hence, \emph{all the corrupted sites
will be in $[ a_{k-1}+1,a_{k}]$}, see Figure~\ref{fig2arw}. Therefore, this
process can be carried on indefinitely, as long as all the steps are
successful.

If all the previous steps are successful, we repeat a symmetric construction
on the negative half-line.

\subsubsection*{If the procedure is successful, then $m_{\eta _{0}}({\boldsymbol{0}})=0$}

We will show that, following the instructions of $\mathcal{I}$, $\eta _{0}$
is stabilized in $V_{n}=[x_{-n},x_{n}]$ with finitely many acceptable
topplings, without toppling ${\boldsymbol{0}}$.

Let us first stabilize the particle that starts at $x_{1}$. To this end, we
successively topple the sites found by the first explorer, until it reaches
the trap at $a_{1}$. At this moment the particle will fall asleep, and site
$a_{1}$ will be stable. Note that these are acceptable topplings and the
particle is following the same path that the explorer did, even if sometimes
it will be sleeping.

We also know that, after the last visit to $a_{1}$, the explorer did not go
further to the right, so when settling the first particle we use all the
instructions examined so far, except some lying in $[a_{0}+1,a_{1}]$.
Therefore, the second particle can be stabilized in the same way, as it will
find the same instructions that determined the second exploration path.

Notice also that the first particle does not visit ${\boldsymbol{0}}$, and
the second particle neither visits ${\boldsymbol{0}}$ nor $a_{1}$, thus it is
settled without activating the first particle. Likewise, the $k$-th particle
is settled at $a_{k}$, without ever visiting
$\{{\boldsymbol{0}},a_{1},a_{2},\dots ,a_{k-1}\}$, for all $k=1,\dots ,n$.
After settling the $n$ first particles in $\mathbb{Z}_{+}$, we perform the
analogous procedure for the first $n$ particles in $\mathbb{Z}_{-}$.

This means that $\eta _{0}$ can be stabilized in $V_{n}$ with finitely many
acceptable topplings, not necessarily in $V_{n}$, and never toppling the
origin. By~\eqref{eq:lapupper}, $m_{V_{n},\eta _{0}}({\boldsymbol{0}})=0$.
Since it holds for all $n\in \mathbb{N}$ and $V_{n}\uparrow \mathbb{Z}$ as
$n\to \infty $, this gives $m_{\eta _{0}}({\boldsymbol{0}})=0$.

\subsubsection*{The procedure is successful with positive probability}

For each site $x \in D_{1}$, the probability of finding a sleep instruction
just before its last jump equals $\frac{\lambda }{1+\lambda }$, and this
happens independently of the path and independently for each site. Thus,
$a_{1}-a_{0}$ is a geometric random variable with parameter $\frac{\lambda
}{1+\lambda }$ truncated at $x_{1}-a_{0}$.

Since no corrupted sites were left outside $[ a_{0}+1,a_{1} ]$ in the first
exploration, the interdistance $a_{2}-a_{1}$ is independent of $a_{1}$.
Moreover, its distribution is also geometric with parameter $\frac{\lambda
}{1+\lambda }$. The same holds for $a_{3}-a_{2}$, $a_{4}-a_{3}$, etc. By the
law of large numbers, $a_{n} \sim \frac{1+\lambda }{\lambda } n$. On the
other hand, $n \sim \zeta x_{n}$ (again by law of large numbers). Therefore,
if $\zeta < \frac{\lambda }{1+\lambda }$, the event that $a_{k} < x_{k}$ for
all $k$ has positive probability. Finally, occurrence of this event implies
that the procedure is successful, and the proof of Theorem~\ref{thm:rs12} is
finished.

\subsubsection*{Some immediate extensions and corollaries}

In the above estimates, we obtained the following.

\begin{remark}\label{remark:barriersgeometric}
The distances $\{a_{k}-a_{k-1}\}_{k\geqslant 1}$ are i.i.d.\ geometric
variables with parameter $\frac{\lambda }{1+\lambda }$.
\end{remark}

Simple modifications in the choice of the trap would give the following.

\begin{remark}
Let $X=(X_{n})_{n \geqslant 0}$ denote the Doob's $h$-transform of a walk
starting at $0$ and jumping according to $p(\cdot )$. That is, $X$ is the
walk conditioned to be positive for all $n > 0$. Let $N$ be a geometric
variable with parameter $\frac{\lambda }{1+\lambda }$ independent of $X$, and
let $Z = \max \{X_{1},\dots ,X_{N}\}$. The distances
$\{a_{k}-a_{k-1}\}_{k\geqslant 1}$ can be made i.i.d.\ and distributed as
$Z$. With this modification, the above proof gives $\zeta _{c} \geqslant
\frac{1}{\mathbb{E}Z}$ rather than just $\zeta _{c} \geqslant
\frac{1}{\mathbb{E}N}$.
\end{remark}

\begin{corollary}\label{cor:depends}
For $d=1$, if the walks are not directed, we have $\zeta _{c} > \frac{\lambda
}{1+\lambda }$.
\end{corollary}

\section{Particle flow between sparse sources}\label{sec:netflow}

In this section we prove the following.

\begin{theorem}\label{thm:subcritbhf}
For $d=1$ and symmetric walks, we have $\zeta _{c} \to 0$ as $\lambda \to 0$.
\end{theorem}

The theorem is proved by checking Condition~\eqref{eq:conditione}. Define the
region $D_{r}=\{1,\dots ,r-1\} \subseteq \mathbb{Z}^{d}$ with $d=1$, and
denote by $\eta '$ the configuration obtained after legally stabilizing $\eta
_{0}$ in $D_{r}$. The following proposition is more than enough to prove the
theorem, and will also be used in \S \ref{sec:cycle}.

\begin{proposition}\label{prop:fewstay}
For $\rho >0$, there are $\lambda >0$, $c>0$ and $C<\infty $ such that
\begin{equation*}
\mathbb{P}\big ( \|\eta ' \|_{D_{r}} \geqslant \rho r \,\big |\, \eta _{0}
\big ) \leqslant C e^{-c r}
\end{equation*}
for all $r \in \mathbb{N}$ and $\eta _{0} \in (\mathbb{N}_{0})^{D_{r}}$.
\end{proposition}

\begin{proof}[Proof of Theorem~\ref{thm:subcritbhf}]
Let $\zeta >0$. By Theorem~\ref{thm:universality}, we can assume the initial
state is i.i.d.\ with mean $\zeta $. To check
Condition~\eqref{eq:conditione}, we can work with $D_{r}$ instead of $V_{n}$.
Taking $\rho <\zeta $ and $\lambda $ as in Proposition~\ref{prop:fewstay}, we
get that $ \limsup _{r} \frac{\mathbb{E}M_{r}}{r} = \limsup _{r}
\frac{\mathbb{E}\|\eta _{0}\|_{D_{r}} -\mathbb{E}\|\eta '\|_{D_{r}}}{r}
\geqslant \zeta - \rho > 0 $. By Theorem~\ref{thm:conditione}, this implies
a.s.\ activity, which means $\zeta _{c} \leqslant \zeta $, concluding the
proof.
\end{proof}

The remainder of this section is devoted to proving Proposition~\ref{prop:fewstay}.

\subsection{General framework}

Fix some natural $K \geqslant 2 \rho ^{-1}$. Each site of the form $iK$ for
$i\in \mathbb{Z}$ is called a \emph{source}. We can suppose $r=(n+1)K$ for
some $n \in \mathbb{N}$. We can also suppose that $\eta _{0} \in
\{0,1\}^{V_{r}}$, otherwise we simply topple every site containing two or
more particles until there is no longer such a site, and start from the
resulting configuration.

We start presenting two auxiliary dynamics, the r-ARW and t-ARW. We then
analyze how each block of the t-ARW behaves individually for all possible
inputs. Later we consider global constraints given by \emph{mass balance
equations} for the flow of particles between sources, and see how the
proposition follows from these constraints and an estimate involving a single
source. We finally prove the single-source estimate.

\subsubsection*{Restricted ARW}

We introduce a toppling procedure that gives a lower bound for the activity
in the ARW. The \emph{restricted ARW procedure} (r-ARW for short) goes as
follows. We assign a different color to each source. Particles get the color
of the last source they visited (those initially located between two sources
can be assigned any of the two colors). When a particle finds a sleep
instruction, we declare it to be \emph{frozen}. When an unfrozen particle is
found at the same site, it will unfreeze the frozen particle, but \emph{only
if they have the same color}. So a site might contain two frozen particles of
different colors, or even a frozen particle of one color plus several
unfrozen particles of another color. At each step, we topple a site in
$D_{r}$ containing unfrozen particles, and we do this until all particles in
$D_{r}$ are frozen.

\smallskip

We now discuss what this procedure says about the ARW on $D_{r}$.

\smallskip

A frozen particle may be active or sleeping in the ARW. But every unfrozen
particle is also active, thus all topplings performed during this procedure
are legal for the ARW. Hence, the configuration obtained at the end of this
procedure gives an upper bound for $\|\eta '\|_{D_{r}}$.

\subsubsection*{Two-layer ARW}

\begin{figure}
\includegraphics{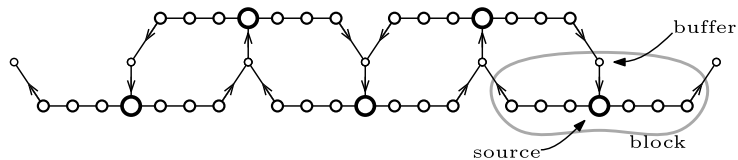}%
\caption{The two-layer graph for the t-ARW with $K=4$ and $n=5$. It has $5$
sources, at (horizontal) distance $4$ from each other, represented by big
sites. Between the two layers there are $5+2$ buffers, represented by tiny
sites. The first and last buffers correspond to sites $x={\boldsymbol{0}}$
and $x=(n+1)K=24$. Directed edges indicate that buffers receive particles
from neighboring blocks and release them to the corresponding sources.}\label{fig:tarw}
\end{figure}

We now introduce the \emph{two-layer ARW dynamics} (t-ARW for short), which
is given by the ARW dynamics on the \emph{two-layer graph} shown in
Figure~\ref{fig:tarw}. Sites are grouped into \emph{blocks} numbered
$i=1,\dots ,n$. Each block contains $2K-1$ regular sites, including a
\emph{source}, plus one \emph{buffer} site. Particles never sleep at the
buffers. For a configuration to be considered stable, all the regular sites
must be either vacant or occupied by a sleeping particle, and each buffer
that is linked to a source must be empty.

Remark that the t-ARW is equivalent to the r-ARW if we identify sites with
the same horizontal coordinate and add up their particles. Given the
horizontal position of a particle, it has two possible colors in the r-ARW,
or equivalently it is on the upper or lower layer in the t-ARW. A sleeping
particle in the t-ARW corresponds to a frozen in the r-ARW. The first and
last buffers will never be toppled, they correspond to sites
${\boldsymbol{0}}$ and $r$.
\smallskip

The t-ARW is more convenient to work with because it is Abelian.
\smallskip

To define the initial configuration $\xi $ on the two-layer graph, we
distribute each particle in configuration $\eta _{0} \in \{0,1\}^{D_{r}}$ to
one of the two layers, according to its color. In the remainder of this
section, we assume that the initial configuration $\xi $ on the two-layer
graph is fixed, and omit it in the notation. The estimates will hold
uniformly with respect to $\xi $.

\subsubsection*{Single-block dynamics}

Consider any sequence of legal topplings performed on the two-layer graph
until the configuration is stable. By Abelianness of the t-ARW, the final
configuration does not depend on the order of topplings. And by the above
considerations, it provides a stochastic upper bound for $\big \|\eta '\big
\|_{D_{r}}$.

Now notice that the interaction between a given block and the other ones is
only through the input of particles from the buffer into its source and the
output of particles from its leftmost and rightmost sites into a neighboring
buffer. In order to analyze the relation between input and output, we fix a
block $i$ and study all possible values of inflow $m=0,1,2,3,\dots $.

For $m\in \mathbb{N}_{0}$, consider the stabilization of the configuration
$\xi + m \delta _{iK}$ inside the $i$-th block. That is, $m$ particles are
added to the source $iK$ and the configuration is toppled until it is stable
in the block. By the Abelian property, it does not matter whether the $m$
particles are all added at the beginning or added one by one with some
topplings being performed in between.

We now define random functions denoted by $ L_{i}(\cdot ) , \ R_{i}(\cdot ) \
\text{and} \ S_{i}(\cdot ) $, illustrated in Figure~\ref{fig:oneblock}. Let
$L_{i}(m)$ count the number of particles that exit the block from the left
when the configuration $\xi + m \delta _{iK}$ is stabilized in the $i$-th
block, let $R_{i}(m)$ count the number of particles that exit the block from
the right, $S_{i}(m)$ the number of particles sleeping in the block. Let
$T_{i}(m)=L_{i}(m)+R_{i}(m)+S_{i}(m)$ be the total number of particles in
$\xi _{i} + m \delta _{iK}$, where $\xi _{i}$ is the restriction of $\xi $ to
the $i$-th block.

\begin{figure}
\includegraphics{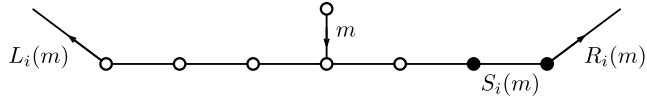}%
\caption{Illustration of the dynamics inside a single block.}\label{fig:oneblock}
\end{figure}

Remark the following about the change of these functions as $m$ increases to
$m'>m$. First, $T_{i}(m)$ equals $m$ plus the number of particles initially
found in the block, so it always increases by $m'-m$. The function $S_{i}$
assumes values on $\{0,\dots ,2K-1\}$, so it can change by at most $2K-1$.
The functions $L_{i}$ and $R_{i}$ are non-decreasing. It follows from these
observations that
%
\begin{equation}\label{eq:boundedincrease}
L_{i}(m') \leqslant L_{i}(m) + m'-m + 2K
\end{equation}
for all $0 \leqslant m < m'$.

Note that these functions are random because they depend on the field of
instructions $\mathcal{I}$, but they are independent across different blocks
$i$.

\subsubsection*{Mass balance equations and proof of active phase}

After globally stabilizing the two-layer graph, the odometer at the internal
buffer sites will be given by $\boldsymbol{m}^{*} = (m_{1}^{*},\dots
,m_{n}^{*})$. Writing $R_{0} \equiv 0$ and $L_{n+1} \equiv 0$, the vector
$\boldsymbol{m}^{*}$ satisfies the \emph{mass balance equations}
%
\begin{equation}\label{eq:buffer}
m_{i} = R_{i-1}(m_{i-1}) + L_{i+1}(m_{i+1}) \qquad \text{ for } i=1, \dots
,n.
\end{equation}
We say that a non-negative integer vector $\boldsymbol{m}$ is
\emph{realizable} if it satisfies the above system of equations. A fixed
deterministic $\boldsymbol{m}$ being realizable is a random event, because
the functions $R_{i}(\cdot )$ and $L_{i}(\cdot )$ are random. Note that the
random odometer $\boldsymbol{m}^{*}$ defined above is always realizable.

We now rewrite the above system as
%
\begin{equation}\label{eq:realizable}
L_{i}(m_{i}) = m_{i-1} - R_{i-2}(m_{i-2})
\end{equation}
for $i=1,\dots ,n+1$, where $R_{-1}\equiv 0$ and $m_{0}$ can be taken as
$L_{1}(m_{1})$.

For a non-negative vector $\boldsymbol{m}$, define
\begin{equation*}
S(\boldsymbol{m}) = \sum _{i=1}^{n} S_{i}(m_{i}) .
\end{equation*}
The total number of particles present in the blocks after global
stabilization of the two-layer graph is given by
\begin{equation*}
S^{*} = S(\boldsymbol{m}^{*}) .
\end{equation*}

\begin{lemma}[Single-block estimate]\label{lemma:singleblock}
If $\lambda $ is small enough depending on $K$, then for every $n\in
\mathbb{N}$, $i=1,\dots ,n$, and initial configuration $\xi $, we have
\begin{equation}
\nonumber
\sup _{\ell }\mathbb{E}\left [ \sum _{m} e^{S_{i}(m)} \cdot
\mathds{1}_{\{L_{i}(m)=\ell \}} \right ] \leqslant 3 .
\end{equation}
\end{lemma}

\begin{proof}[Proof of Proposition~\ref{prop:fewstay}]
Recall that $K \geqslant 2 \rho ^{-1}$ is fixed, and choose $\lambda $
according to the previous lemma. Also recall that $r=(n+1)K$. Finally, recall
that~\eqref{eq:realizable} is satisfied for $i=1,\dots ,n$ when
$\boldsymbol{m}=\boldsymbol{m}^{*}$. Therefore,
\begin{align*}
\mathbb{P}(S^{*} \geqslant \rho r) &= \sum _{\boldsymbol{m}}
\mathbb{P}( S(\boldsymbol{m}) \geqslant \rho r, \ \boldsymbol{m}^{*}
= \boldsymbol{m})
\\
& \leqslant \sum _{\boldsymbol{m}} \mathbb{P}\big ( e^{S(
\boldsymbol{m})} \geqslant e^{\rho r},\ \boldsymbol{m}
\text{ is realizable} \big )
\\
& \leqslant e^{-\rho r} \sum _{\boldsymbol{m}} \mathbb{E}\Big [ e^{S(
\boldsymbol{m})} \mathds{1}_{\{\boldsymbol{m}\text{ is realizable}
\}} \Big ]
\\
& = e^{-\rho r} \sum _{m_{0}} \mathbb{E}\bigg [ \sum _{m_{1}} \dots
\sum _{m_{n}} \prod _{i=1}^{n} e^{S_{i}(m_{i})} \mathds{1}_{\{L_{i}(m_{i})
= m_{i-1} - R_{i-2}(m_{i-2})\}} \bigg ] .
\end{align*}
Now notice that the random functions $S_{n}(\cdot )$ and $L_{n}(\cdot )$ are
independent of the family of random functions $\mathcal{X}_{n} = (L_{j}(\cdot
),R_{j}(\cdot ),S_{j}(\cdot ))_{j =1, \dots n-1}$. Hence,
\begin{align*}
\mathbb{E}\bigg [ & \sum _{m_{1}} \dots \sum _{m_{n}} \prod _{i=1}^{n}
e^{S_{i}(m_{i})} \mathds{1}_{\{L_{i}(m_{i}) = m_{i-1} - R_{i-2}(m_{i-2})
\}} \bigg | \mathcal{X}_{n} \bigg ] =
\\
& = \sum _{m_{1}} \dots \sum _{m_{n-1}} \mathbb{E}\bigg [ \sum _{m_{n}}
e^{S_{n}(m_{n})} \mathds{1}_{\{L_{n}(m_{n}) = m_{n-1} - R_{n-2}(m_{n-2})
\}} \bigg | \mathcal{X}_{n} \bigg ] \times
\\
& \qquad \qquad \qquad \qquad \times \prod _{i=1}^{n-1} \Big [ e^{S_{i}(m_{i})}
\mathds{1}_{\{L_{i}(m_{i}) = m_{i-1} - R_{i-2}(m_{i-2})\}} \Big ] .
\end{align*}
The last conditional expectation is bounded from above by
\begin{equation*}
\sup _{\ell }\mathbb{E}\bigg [ \sum _{m_{n}} e^{S_{n}(m_{n})}
\mathds{1}_{\{L_{n}(m_{n}) = \ell \}} \bigg ] .
\end{equation*}
Hence, regarding the previous chain of inequalities, we can pull the last
term in the product out of the expectation. The same reasoning works for
$i=n-1,n-2,\dots $, giving
\begin{align*}
\mathbb{P}(S^{*} \geqslant \rho r) & \leqslant e^{-\rho r} \sum _{m_{0}}
\mathbb{E}\bigg [ \sum _{m_{1}} \dots \sum _{m_{n}} \prod _{i=1}^{n} e^{S_{i}(m_{i})}
\mathds{1}_{\{L_{i}(m_{i}) = m_{i-1} - R_{i-2}(m_{i-2})\}} \bigg ]
\\
& \leqslant e^{-\rho r} \sum _{m_{0}} \prod _{i=1}^{n} \sup _{\ell
}\mathbb{E}\bigg [ \sum _{m_{i}} e^{S_{i}(m_{i})} \mathds{1}_{\{L_{i}(m_{i})
= \ell \}} \bigg ]
\\
& \leqslant r e^{-\rho r} 3^{n} \leqslant r e^{-c r} .
\end{align*}
Since $\big \|\eta '\|_{D_{r}}$ is stochastically dominated by
$S^{*}$, this concludes the proof.
\end{proof}

\subsection{Single-block estimate}

We now prove Lemma~\ref{lemma:singleblock}. Take
$M_{0}\in \mathbb{N}$ so that
%
\begin{equation}\label{eq:defm0}
e^{2K} (\tfrac{5}{9})^{j} \leqslant (\tfrac{3}{5})^{j} \quad
\text{ for all } j \geqslant M_{0}
\end{equation}
and take $\varepsilon >0$ so that
%
\begin{equation}\label{eq:defeps}
6 \varepsilon e^{2K} + (\tfrac{5}{9})^{j} \leqslant \tfrac{6}{5} (
\tfrac{3}{5})^{j} \quad \text{ for all } j \leqslant M_{0} .
\end{equation}
Now take $M_{1} \in \mathbb{N}$ such that the probability of getting at least
one tail out of $M_{1}$ fair coin tosses is at least $1-\varepsilon $, and
take $M_{2} > 2M_{1}$ such that the probability of getting at least
$M_{1}+2K+2$ tails out of $M_{2}$ fair coin tosses is at least $1-\varepsilon
$. Finally, take $\lambda $ small enough so that, with probability at least
$1-\varepsilon $, $M_{2}$ independent random walks all reach distance $2K$
before sleeping.

We now show how to put all these pieces together.

Let $\ell \in \mathbb{N}_{0}$. For the process
\begin{equation*}
\left ( \big . L_{i}(m),R_{i}(m),S_{i}(m) \right )_{m = 0,1,2,\dots },
\end{equation*}
define the following stopping times:
\begin{align*}
\mathcal{T}_{1} &= \min \{ m: L_{i}(m) \geqslant \ell - M_{1} - 2K - 2
\} ,
\\
\mathcal{T}_{2} &= \mathcal{T}_{1} + M_{1} ,
\\
\mathcal{T}_{3} &= \min \{ m: L_{i}(m) \geqslant \ell \} ,
\\
\mathcal{T}_{4} &= \min \{ m: L_{i}(m) \geqslant \ell +1 \} .
\end{align*}
From the above definitions and~\eqref{eq:boundedincrease}, we have
\begin{equation*}
\mathcal{T}_{1} < \mathcal{T}_{2} < \mathcal{T}_{3} \leqslant
\mathcal{T}_{4} .
\end{equation*}
Consider the event
\begin{equation*}
\mathcal{A}= \big \{ S_{i}(m)>0 \text{ for some } m\in [\mathcal{T}_{3},
\mathcal{T}_{4}) \big \} .
\end{equation*}
Since $S_{i}(m)<2K$ for every $m$, by definition of $\mathcal{A}$ we have,
%
\begin{equation}
\label{eq:expmoment}
\sum _{j \geqslant 0} e^{S_{i}(\mathcal{T}_{3}+j)} \cdot \mathds{1}_{
\{\mathcal{T}_{4}-\mathcal{T}_{3}>j\}} \leqslant \sum _{j
\geqslant 0} e^{2K \mathds{1}_{\mathcal{A}}} \cdot \mathds{1}_{\{
\mathcal{T}_{4}-\mathcal{T}_{3}>j\}} .
\end{equation}
We will show that
%
\begin{equation}
\label{eq:propstop}
\mathbb{P}(\mathcal{T}_{4}-\mathcal{T}_{3} > j) \leqslant (
\tfrac{5}{9})^{j} \quad \text{ for all } j \geqslant 0
\end{equation}
and
%
\begin{equation}
\label{eq:propbad}
\mathbb{P}(\mathcal{A}) \leqslant 6 \varepsilon .
\end{equation}
Let us first see how these imply the lemma.
Using~\eqref{eq:expmoment},~\eqref{eq:propstop},~\eqref{eq:propbad},~\eqref{eq:defeps},
and~\eqref{eq:defm0},
\begin{align*}
\mathbb{E}\left [ \sum _{m} e^{S_{i}(m)} \cdot \mathds{1}_{\{L_{i}(m)=
\ell \}} \right ] & \leqslant \sum _{j \geqslant 0} \mathbb{E}\left [
e^{2K \mathds{1}_{\mathcal{A}}} \cdot \mathds{1}_{\{\mathcal{T}_{4}-
\mathcal{T}_{3}>j\}} \right ]
\\
& \leqslant \sum _{j = 0}^{M_{0}} \left [ \mathbb{P}(\mathcal{T}_{4}-
\mathcal{T}_{3}>j) + e^{2K} \mathbb{P}(\mathcal{A}) \right ] +
\\
& \qquad \qquad \qquad + \sum _{j > M_{0}} \left [ e^{2K} \mathbb{P}(
\mathcal{T}_{4}-\mathcal{T}_{3}>j) \right ]
\\
& \leqslant \sum _{j = 0}^{M_{0}} \left [ 6 \varepsilon e^{2K} + (
\tfrac{5}{9})^{j} \right ] + \sum _{j > M_{0}} \left [ e^{2K} (
\tfrac{5}{9})^{j} \right ]
\\
& \leqslant \sum _{j \geqslant 0} \tfrac{6}{5} (\tfrac{3}{5})^{j} = 3 .
\end{align*}

So let us prove~\eqref{eq:propstop}. Given that $L_{i}(m)=\ell $, and given
all the information revealed when stabilizing $\xi + m\delta _{iK}$ in
the $i$-th block, we claim that the conditional probability of
$L_{i}(m+1)>\ell $ is at least $\frac{1-\varepsilon }{2}$, which is greater
than $\frac{4}{9}$. Estimate~\eqref{eq:propstop} then follows by successive
conditioning. Now to see why the claim holds true, consider the following
toppling procedure for $\xi + (m+1)\delta _{iK}$. We keep the
$(m+1)$-st particle in the buffer and let $\xi + m\delta _{iK}$ stabilize.
We then add said particle to the source, and move it until it either finds
a sleep instruction or exits the block. By the choice of $\lambda $, the
probability of exiting before finding a sleep instruction has probability
at least $1-\varepsilon $, and by symmetry the probability of leaving the
block from the left will be half of it. After that, we stabilize the remaining
active particles, if any.

\medskip
To finish the proof of Lemma~\ref{lemma:singleblock}, it remains to show~\eqref{eq:propbad}.

\medskip
We consider only $\ell \geqslant M_{1} + 2K + 2$. The case of smaller
$\ell $ uses similar but simpler arguments, and will be omitted. We will
indicate with a ``(*)'' some events which occur with conditional probability
at least $1-\varepsilon $, as a consequence of our choices of
$M_{1}$, $M_{2}$ and $\lambda $. This chain of events all together imply
$\mathcal{A}^{c}$.

Denote by $(\xi + \mathcal{T}_{1} \delta _{iK})'$ the configuration obtained
after stabilizing the configuration
$\xi + \mathcal{T}_{1} \delta _{iK}$ in the $i$-th block. We now stabilize
the configuration
$(\xi + \mathcal{T}_{1} \delta _{iK})' + M_{1} \delta _{iK}$ obtained
by adding $M_{1}$ new active particles at the source $iK$. Let each of
these $M_{1}$ new active particles move until it exits the block or finds
a sleep instruction. Suppose none of them finds a sleep instruction before
exiting (*). Suppose at least one of them exits the block from the left
(*) and at least one from the right (*).

In this case, all the sleeping particles present in
$(\xi + \mathcal{T}_{1} \delta _{iK})'$ have been activated. One by one,
let each particle move until it exits the block or finds a sleep instruction.
Suppose none of them finds a sleep instruction (*). When all the above
events occur, $S_{i}(\mathcal{T}_{2})=0$.

So we resume from $\mathcal{T}_{2}$ and assuming
$S_{i}(\mathcal{T}_{2})=0$. Suppose the next $M_{2}$ particles to be added
to the source $iK$ exit the $i$-th block before finding a sleep instruction
(*). Given this event, the conditional probability that at least
$M_{1}+2K+2$ of them exit from the left is also at least
$1-\varepsilon $.

Suppose the latter event also occurs. Then $S_{i}(m)=0$ for
$m=\mathcal{T}_{2},\mathcal{T}_{2}+1,\dots ,\mathcal{T}_{2}+M_{2}$,
and moreover
$L_{i}(\mathcal{T}_{2}+M_{2}) \geqslant L_{i}(\mathcal{T}_{2})+M_{1}+2K+2$.
It remains to check that these two events imply $\mathcal{A}^{c}$. The
last inequality implies that
$L_{i}(\mathcal{T}_{2}+M_{2}) \geqslant \ell +1$, so
$\mathcal{T}_{4} \leqslant \mathcal{T}_{2} + M_{2}$. On the other hand,
as noted after these stopping times were defined, it follows from~\eqref{eq:boundedincrease}
that $\mathcal{T}_{2}<\mathcal{T}_{3}$. Hence, when the above events
occur we have $S_{i}(m)=0$ for
$\mathcal{T}_{3} \leqslant m \leqslant \mathcal{T}_{4}$ and event
$\mathcal{A}$ cannot occur, so its probability is at most
$6 \varepsilon $.

This concludes the proof of~\eqref{eq:propbad} and hence that of Proposition~\ref{prop:fewstay}.

\section{Fast and slow phases for finite systems}\label{sec:cycle}

Consider the ARW model on the ring $\mathbb{Z}_{n} = \mathbb{Z}/ n\mathbb{Z}$
instead of $\mathbb{Z}^{d}$. When a particle jumps, it chooses one of the two
nearest neighbors according to a fair coin. The initial configuration $\eta
_{0}$ is taken as i.i.d.\ Poisson with parameter $\zeta $ and all particles
starting active. Let $\mathcal{T}= \sum _{x} m_{\mathbb{Z}_{n},\eta _{0}}(x)$
denote the total number of topplings performed during stabilization of $\eta
_{0}$.

\begin{theorem}\label{thm:fes}
Let $0<\zeta <1$. If $\lambda $ is small enough, there exist $\delta >0$ and
$\delta '>0$ such that, for all large $n$,
%
\begin{equation}\label{eq:slow}
\mathbb{P}(\mathcal{T}\geqslant e^{\delta ' n}) \geqslant 1-e^{- \delta n} .
\end{equation}
If $\lambda $ is large enough, there exist $\delta >0$ and $\kappa <\infty $
such that, for all large $n$,
%
\begin{equation}\label{eq:fast}
\mathbb{P}(\mathcal{T}\leqslant \kappa n \log ^{2} n) \geqslant 1-n^{- \delta
} .
\end{equation}
\end{theorem}

\begin{problem}
Show that~\eqref{eq:slow} or~\eqref{eq:fast} must hold for every $(\zeta
,\lambda )$ outside the critical curve of Figure~\ref{fig:predictions}.
\end{problem}

\begin{problem}
Improve the estimate~\eqref{eq:fast} to one without the
$\log ^{2} n$ term.
\end{problem}

\begin{problem}
Show that the family $m_{\mathbb{Z}_{n},\eta _{0}}({\boldsymbol{0}})$ is
tight on the fast phase.
\end{problem}

\begin{problem}
Show similar behavior in case of biased jumps.
\end{problem}

\begin{problem}
Study fast to slow transition in higher dimensions.
\end{problem}

\subsection{Slow phase}\label{sub:slow}

We first prove~\eqref{eq:slow}. For simplicity we consider the model on
$\mathbb{Z}_{2n}$ (\textit{i.e.}\ we work with an even number of sites). The
proof is based on a cyclic toppling procedure described as follows. At all
times, particles will be declared to be \emph{alive} or \emph{steady}.
Initially, declare all particles to be \emph{alive}.

\begin{figure}
\includegraphics{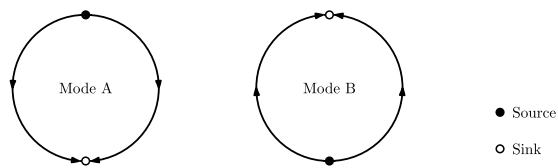}%
\caption{Two alternating modes of the toppling procedure.} \label{fig:modes}
\end{figure}

The toppling procedure consists in alternating between two modes, as
illustrated in Figure~\ref{fig:modes}. In Mode~A, site $x=0$ is called
\emph{source} and $x=n$ is called \emph{sink}. In Mode~B the roles are
reversed. During each mode, we topple all sites containing active alive
particles, except the sink. We keep toppling until there are no more such
sites. At the beginning of each mode, we declare all active particles to be
alive, and all sleeping particles to be steady. Note that particles that are
declared steady at the beginning of a mode will not move during that mode,
but they may be activated in the meantime, in which case they will be alive
in the next mode. Note also that in this process we only use legal topplings.

Fix some $\rho < \frac{\zeta }{5}$, so the expected number of particles in
$\mathbb{Z}_{2n}$ equals $10 \rho n$. If the initial configuration has fewer
than $9 \rho n$ particles, we stop the procedure (before it even starts).
Otherwise we run Mode~A.

At the end of Mode~A, we stop the procedure if~(i) more than $2\rho n$ alive
particles are sleeping, or~(ii) some site of $\mathbb{Z}_{2n}$ is not visited
by any of the alive particles during this mode. Otherwise we switch to
Mode~B. At the end of Mode~B, we stop the procedure if at least one of the
above conditions is met, otherwise we switch to Mode~A. We continue this
indefinitely until a condition to stop is met. The first run of Mode~A is
rather special: it does not start with many particles at the source, but it
starts with all particles alive. For this first run we do not check
condition~(ii).

We now argue that, at the beginning of each run of Mode~A or Mode~B, at least
$7 \rho n$ particles will be active and alive, of which at least $5 \rho n$
will be at the source. Indeed, at the end of the previous mode, no more than
$2 \rho n$ of the alive particles were found sleeping, hence at least $5 \rho
n$ of them finished at the sink. Moreover, all sites of $\mathbb{Z}_{2n}$
have been visited, hence all the previously steady particles were activated,
implying that at least $7 \rho n$ particles were active when the previous
mode ended.
\smallskip

We are ready to make the probabilistic estimates leading to~\eqref{eq:slow}.
\smallskip

A crude bound on the number of topplings is the following: at each mode we
perform at least one toppling. Hence, to get~\eqref{eq:slow} we only need to
prove that the probability (conditioned on the past) of stopping the
procedure at the end of each mode is less than $e^{-cn}$ for some $c$.

Moreover, it is enough to prove this estimate for Mode~B. Indeed, Mode~A
is analogous (and its first run is simpler), and the probability of starting
with fewer than $9 \rho n$ particles decays exponentially fast by Chernoff
bound for sums of i.i.d.\ Poisson variables.

Condition~(i) is the most laborious part. But all the work has been done in
\S \ref{sec:netflow}, and at this point it suffices to choose $\lambda $
according to Proposition~\ref{prop:fewstay}. Indeed, since the sink is never
toppled, this process is the same as the process on $D_{2n} = \{1,\dots
,2n-1\}$ and Proposition~\ref{prop:fewstay} says that condition~(i) is met
with exponentially small probability.

For condition~(ii), label $5 \rho n$ alive particles initially found at $x=0$
and see whether each one of them ends up sleeping, reaches $x=n$ in the
clockwise sense, or reaches $x=n$ in the counter-clockwise sense. For the
particles which end up sleeping, use extra randomness to complete their
tentative path which is stopped only upon reaching $x=n$. By symmetry, each
labeled particle tentatively reaches $x=n$ from either direction with
probability $\frac{1}{2}$, independently of the other labeled particles. Now
if condition~(i) is not met, it implies that at least $3 \rho n$ of these $5
\rho n$ particles perform a random walk which does make it to $x=n$, so they
will perform all their tentative paths. On the other hand, the probability
that $3 \rho n$ or more out of the $5 \rho n$ independent tentative paths
reach $n$ through the same direction is exponentially small by Chernoff bound
for sums of i.i.d.\ Bernoulli variables. This gives the estimate for the
probability of stopping the procedure at each mode, which concludes the proof
of~\eqref{eq:slow}.

\subsection{Fast phase}\label{sub:fast}

We now prove~\eqref{eq:fast}. Let
%
\begin{equation}\label{eq:choicelambda}
\frac{\lambda }{1+\lambda } > \zeta ' > \zeta .
\end{equation}
Letter $K$ denotes a constant which will be chosen large during the proof.
Letters $C$ and $c$ denote constants whose precise values are not crucial for
the arguments and may change from line to line. They depend on $\zeta $,
$\zeta '$ and $\lambda $ but neither on $K$ nor on $n$. Constants $\delta $,
$\delta '$ and $\kappa $ also depend on the choice of $K$.

Let $r = \lfloor K \log n \rfloor $ and split the ring $\mathbb{Z}_{n}$ into
\emph{arcs} containing $2r+1$ sites (we can leave a spare site between some
pairs of consecutive arcs to keep $n$ sites in total). The middle point of
each arc will be called the \emph{source}.

We consider a toppling procedure split into two stages. In Stage~1, we
force each particle to move by means of acceptable topplings, until it
reaches one of the sources. If a source gets more than $2 \zeta ' r$ particles,
we declare the procedure to have failed and stop. In Stage~2, we let the
configuration stabilize normally, by using only legal topplings. If a particle
leaves the corresponding arc during this stage, we declare the procedure
to have failed. Otherwise it is successful.

We will show that the probability of failure is bounded by
$n^{-\delta }$ for large $n$. Moreover, with exponentially high probability,
on the event that the procedure succeeds, the total odometer will be bounded
by $Cnr^{2}$, which proves~\eqref{eq:fast}.

\subsubsection*{Estimate for the total odometer}

We start by stating a crude bound on exit times for random walks. Consider
a collection of $n$ or fewer independent lazy walks, each one stopped upon
reaching distance $3r$. The number of steps each walk makes is stochastically
dominated by $C r^{2} X$, where $X$ is a geometric random variable with
mean $2$. Indeed, by taking $C$ large, the probability of exiting
$[-3r,3r]$ within $C r^{2}$ steps is larger than $\frac{1}{2}$ regardless
of the starting point, so if it fails to exit in $Cr^{2}$ steps we can
start over. Hence, by enlarging $C$ and using Chernoff bound for sums of
i.i.d.\ geometric variables, the probability that the overall total number
of jumps exceeds $C n r^{2}$ is less than $e^{-cn}$.

We can assume that $\|\eta _{0}\| < n$, otherwise the procedure will necessarily
fail at Stage~1. Now during this stage, each particle performs a lazy random
walk stopped upon reaching a source. The laziness comes from the fact that
sleep instructions keep the particle at the same site. So by the previous
paragraph, the total odometer produced during this stage is bounded by
$C n r^{2}$ with probability at least $1-e^{-cn}$.

During Stage~2, each particle performs a random path which is stopped earlier
than reaching distance $r$, unless the procedure fails. In order to compare
the path performed by the particles with a collection of independent lazy
walks, we extend the paths by using extra randomness, so that the resulting
collection of paths is distributed as a collection of independent lazy
symmetric walks. So the same argument applies to the total odometer obtained
at this stage, concluding the estimates on the total odometer.

\subsubsection*{Estimates for Stage~1}

Let $x$ be a given source. Particles which reach $x$ during Stage~1 must
have started at sites between the sources immediately to the left and right
of $x$. Conditioning on the initial configuration, for each site $y$, a
given particle starting at $y$ will, independently of other particles,
reach $x$ before another source with a probability $p_{y}$ which is proportional
to the distance between its initial location and the other source.

Now the i.i.d.\ Poisson assumption about the initial configuration
$\eta _{0}$ simplifies the argument, in that the number of particles reaching
$x$ will be a Poisson variable with parameter
$\sum _{y} \zeta p_{y}$, independently over $y$. Since the probabilities
$p_{y}$ increase linearly from $0$ to $1$ and the nearest sources are at
distance less than $2r+2$, this parameter is less than
$2 \zeta r + 2$. So by Chernoff bound for a Poisson variable with large
parameter, the probability that $x$ gets more than $2 \zeta 'r$ particles
during Stage~1 is bounded by $e^{-cr}$ for large $r$. Summing over all
sources, the probability that some source gets more than
$2 \zeta 'r$ particles is bounded by $\frac{n}{2r+1} e^{-cr}$ for large
$n$. By choosing $K$ large enough, this is less than $n^{-\delta }$ for
all large $n$.

\subsubsection*{Estimates for Stage~2}

The central statement about this stage is the following.

\begin{proposition}\label{prop:arwidla}
For $\zeta '<\frac{\lambda }{1+\lambda }$, consider the symmetric ARW on
$\mathbb{Z}$ starting with $m \leqslant 2 \zeta ' r$ particles at
${\boldsymbol{0}}$ and no other particles elsewhere. Then the probability
that some particle ever leaves the interval $[-r,r]$ is less than $e^{-cr}$
for all large $r$.
\end{proposition}

Note that, unless the procedure failed during Stage~1, there are fewer
than $2 \zeta 'r$ particles at each source. Therefore, applying the proposition
and summing over all arcs we can bound the probability of failure during
Stage~2 by $\frac{n}{2r+1} e^{-cr}$ for large $n$, with a possibly different
constant $c$. By further enlarging $K$, again we can find $\delta $ such
that this bound stays below $n^{-\delta }$ for all large $n$.
\smallskip

It remains to prove the proposition.
\smallskip

First, recall the toppling procedure presented in \S \ref{sec:traps}. Translating
that procedure by $-r$, we initially have a barrier at $a_{0}=-r$, and
a new barrier $a_{j}>a_{j-1}$ is added when an exploration reaches the
previous barrier $a_{j-1}$. A modification that we make in order to prove
Proposition~\ref{prop:arwidla} is to confine the explorations from both
directions. Namely, there is a second barrier initially at
$b_{0}=+r$, and a new barrier $b_{i}<b_{i-1}$ is added each time the exploration
reaches the previous barrier $b_{i-1}$. We can carry this exploration
$m$ times, as long as the condition $a_{j}<0<b_{i}$ is preserved. If some
of the $m$ explorations fails to find a suitable trap, we declare the procedure
to have failed and stop, otherwise it is successful.
\smallskip

We now show that this procedure is successful with high probability.
\smallskip

Recalling Remark~\ref{remark:barriersgeometric}, the distances
$a_{j}-a_{j-1}$ are i.i.d.\ until the moment $j_{*}$ when the procedure
fails due to the explorer hitting $a_{j_{*}-1}$ but being unable to find
a suitable trap. Moreover, even the failure to find the trap can be coupled
to the event that $a_{j_{*}-1}+Y \geqslant x_{0}$ for a geometric random
variable $Y$, where $x_{0}=0$ is the starting position of the explorer.
For convenience, after the procedure is finished we continue sampling independent
geometric variables just so that we get two independent i.i.d.\ sequences
$(a_{j}-a_{j-1})_{j\in \mathbb{N}}$ and
$(b_{i-1}-b_{i})_{i\in \mathbb{N}}$.

Define $J(0)=I(0)=0$. When the $(k+1)$-th explorer starts, the barriers
are at $a_{J(k)}$ and $b_{I(k)}$. If the explorer hits $a_{J(k)}$ before
$b_{I(k)}$, we set $J(k+1)=J(k)+1$ and $I(k+1)=I(k)$, otherwise we set
$I(i+1)=I(k)+1$ and $J(k+1)=J(k)$. This way $J(k)+I(k)=k$ throughout the
whole procedure. The goal then is to show that
%
\begin{equation}\label{eq:urnestimate}
\mathbb{P}(k_{*} \leqslant m ) \leqslant e^{-c r},
\end{equation}
where $k_{*} = \min \{ k : a_{J(k)} \geqslant 0 \text{ or } b_{I(k)}
\leqslant 0 \}$.

\medskip
If we were assuming $\zeta <\frac{1}{2}$, we could choose $\frac{\lambda
}{1+\lambda }>2\zeta '$ in~\eqref{eq:choicelambda}. In this case, the
conclusion of Proposition~\ref{prop:arwidla} would follow immediately from
the analysis made in \S \ref{sec:traps}, with the use of Chernoff bound for
sums of i.i.d.\ geometric variables. However, we want to show a stronger
result that extends to arbitrary $\zeta <1$, and this requires a more
delicate argument. We would like to argue that about half of the particles
will go to each direction. But there is an inconvenient reinforcement here:
if many explorations have chosen left, it increases the chances that the next
explorations will make the same choice. Fortunately, this effect is not
strong enough to produce a macroscopic unbalance between $I$ and $J$, as
shown below.

\medskip
The law of $(-a_{J(k)},b_{I(k)})_{k=0,\dots ,k_{*}}$ can be described as
follows. Consider an urn containing $X_{0} = r$ purple and
$Z_{0} = r$ yellow balls. At each turn $k$, a ball is sampled uniformly
from the urn. The sampled ball is returned and a random number
$Y_{k}$ of balls with opposite color are destroyed, where $Y_{k}$ has geometric
distribution with parameter
$\zeta '' = \frac{\lambda }{1+\lambda } > \zeta '$. That is,
\begin{equation*}
(X_{k},Z_{k})-(X_{k-1},Z_{k-1}) =
\begin{cases}
(0,-Y_{k}) ,& \text{with probability } \frac{X_{k}}{X_{k}+Z_{k}},
\\
(-Y_{k},0) ,& \text{with probability } \frac{Z_{k}}{X_{k}+Z_{k}}.
\end{cases}
\end{equation*}
Finally, the urn process is stopped when one of the colors disappears,
that is at step
$k_{*} = \min \{ k : X_{k} \leqslant 0 \text{ or } Z_{k} \leqslant 0\}$.

\begin{lemma}\label{lemma:urn}
For $0<\zeta '<\zeta ''<1$, we have $\mathbb{P}(k_{*} \leqslant 2 \zeta ' r )
\leqslant e^{-c r}$.
\end{lemma}

The proof is given in Section~3.4 of~\cite{BasuGangulyHoffmanRichey19} using
a decoupling of the urn process as in~\cite{KingmanVolkov03} and applying
estimates for sums of non-i.i.d.\ geometric and exponential random variables
from~\cite{Janson18}.

Lemma~\ref{lemma:urn} implies~\eqref{eq:urnestimate}, which in turn implies
Proposition~\ref{prop:arwidla}.

\section{Weak and strong stabilization}\label{sec:weak}

In this section we prove the following for the ARW on~$\mathbb{Z}^{d}$.

\begin{theorem}\label{thm:subcritstauffertaggi}
For any jump distribution in any dimension, $ \zeta _{c} \geqslant
\frac{\lambda }{1+\lambda } $.
\end{theorem}

\begin{theorem}\label{thm:supcritstauffertaggi}
If $d>2$, then $\zeta _{c}<1$ for every $\lambda <\infty $ and $\zeta _{c}
\to 0$ as $\lambda \to 0$.
\end{theorem}

\begin{problem}
Prove a similar statement for unbiased walks on $\mathbb{Z}^{2}$.
\end{problem}

\begin{problem}
At least prove that $\zeta _{c}<1$ for some $\lambda $.
\end{problem}

We will consider toppling sequences for any finite
$V \subseteq \mathbb{Z}^{d}$, but changing the stability condition at
site ${\boldsymbol{0}}$. In what follows, we will define weak and strong
stabilizations of a configuration $\eta $ in the box $V$.

\subsection{Definitions and first corollaries}\label{sub:weakcor}

We say that ${\boldsymbol{0}}$ is \emph{w-stable} if
$\eta ({\boldsymbol{0}}) \leqslant 1$, and we say that
${\boldsymbol{0}}$ is \emph{s-stable} if
$\eta ({\boldsymbol{0}})=0$. Otherwise we say that
${\boldsymbol{0}}$ is \emph{w-unstable} or \emph{s-unstable}. For
$y \ne {\boldsymbol{0}}$ we say that $y$ is stable, w-stable, and s-stable
if $\eta (y) \leqslant \mathfrak{s}$.

Toppling a site $z$ will be called \emph{w-legal} if $z$ is w-unstable,
and \emph{s-legal} if $z$ is s-unstable (recall from \S
\ref{sub:diaconis} that toppling a site containing a sleeping particle
is a well-defined operation). We say that a sequence $\alpha $ of acceptable
topplings \emph{weakly stabilizes} $\eta $ in $V$ if
$\Phi _{\alpha }\eta $ is w-stable in $V$. We say that it
\emph{strongly stabilizes} $\eta $ in $V$ if $\Phi _{\alpha }\eta $ is s-stable
in $V$.

Let us restate Lemma~\ref{lemma:lap} in this context.
%
\begin{lemma}\label{lemma:lap2}
If $\alpha $ is an acceptable sequence of topplings that w-stabilizes $\eta $
in $V$, and $\beta \subseteq V$ is a w-legal sequence of topplings for $\eta
$, then $m_{\beta }\leqslant m_{\alpha }$. The same holds replacing `w' by
`s'.
\end{lemma}

\begin{proof}
It is the very same proof as that of Lemma~\ref{lemma:lap}.
\end{proof}

Define
\begin{equation*}
m^{w}_{V,\eta } = \sup _{\beta \subseteq V\text{ w-legal}} m_{\beta },
\qquad m^{s}_{V,\eta } = \sup _{\beta \subseteq V\text{ s-legal}} m_{\beta }.
\end{equation*}
Notice that w-legal topplings are always legal, which are always s-legal,
which in turn are always acceptable. In particular, we have by inclusion
\begin{equation}
\nonumber
m^{w}_{V,\eta } \leqslant m_{V,\eta } \leqslant m^{s}_{V,\eta } .
\end{equation}

Now let $\eta _{0}$ be random and independent of $\mathcal{I}$, as described
in \S \ref{sub:conditions}. Let $\eta '_{V}$ be the resulting configuration
obtained by stabilizing $\eta _{0}$ in $V$ with legal topplings, and
$\eta _{V}^{w}$ be the result of weakly stabilizing $\eta _{0}$ in
$V$ with w-legal topplings. For finite $V$, these are a.s.\ well-defined
since a stable or w-stable configuration is a.s.\ achieved after finitely
many topplings. The Abelian property implies that neither
$\eta '_{V}$ nor $\eta ^{w}_{V}$ depends on the order at which these legal
or w-legal topplings are performed.

\begin{proof}[Proof of Theorem~\ref{thm:subcritstauffertaggi}]
By Abelianness, one way to stabilize $\eta _{0}$ in $V$ is to first weakly
stabilize $\eta _{0}$ in $V$, and then stabilize $\eta ^{w}_{V}$ in
$V$.

With the procedure in mind, we claim that
\begin{equation}
\nonumber
m_{V,\eta _{0}}({\boldsymbol{0}}) \geqslant 1 \quad \Longrightarrow
\quad \eta ^{w}_{V}({\boldsymbol{0}})=1 .
\end{equation}
If there is ever a particle at ${\boldsymbol{0}}$, weak stabilization
does not let this particle leave, and
$\eta ^{w}_{V}({\boldsymbol{0}})=1$. Otherwise, ${\boldsymbol{0}}$ is
never visited during weak stabilization and
$\eta ^{w}_{V}({\boldsymbol{0}})=0$, in which case $\eta ^{w}_{V}$ is
not only w-stable but also stable, so $\eta '_{V}=\eta ^{w}_{V}$ and
$m_{V,\eta _{0}}({\boldsymbol{0}})=0$, which proves the claim. On the
other hand,
\begin{equation}
\nonumber
\mathbb{P}\left ( \big . \eta '_{V}({\boldsymbol{0}}) =
\mathfrak{s}\right ) \geqslant \tfrac{\lambda }{1+\lambda } \,
\mathbb{P}\left ( \big . \eta _{V}^{w}({\boldsymbol{0}})=1 \right ) .
\end{equation}
Indeed, if $\eta ^{w}_{V}({\boldsymbol{0}})=1$ then
${\boldsymbol{0}}$ is the only site in $V$ where $\eta ^{w}_{V}$ is unstable,
so stabilization of $\eta ^{w}_{V}$ starts with a toppling at
${\boldsymbol{0}}$. In this case, with probability
$\frac{\lambda }{\lambda +1}$, stabilization is achieved immediately, with
a particle sleeping at ${\boldsymbol{0}}$.

From these two observations we get
\begin{equation}
\nonumber
\mathbb{P}\left ( \big . \eta '_{V}({\boldsymbol{0}}) =
\mathfrak{s}\right ) \geqslant \tfrac{\lambda }{1+\lambda } \,
\mathbb{P}\left ( \big . m_{V,\eta _{0}}({\boldsymbol{0}})
\geqslant 1 \right ) .
\end{equation}

Now suppose $m_{\eta _{0}}({\boldsymbol{0}}) = \infty $ a.s.\ and let
$\varepsilon >0$. Take $r$ such that
$\mathbb{P}(m_{V,\eta _{0}}({\boldsymbol{0}}) \geqslant 1)
\geqslant 1-\varepsilon $ for every $V \supseteq B_{r}$. Take $R$ such
that $|B_{R-r}| \geqslant (1-\varepsilon ) |B_{R}|$. From the previous
inequality,
$\mathbb{E}\left [ \#\{x \in B_{R-r}:\eta '_{B_{R}}(x) =
\mathfrak{s}\} \right ] \geqslant \frac{\lambda }{1+\lambda }(1-
\varepsilon )^{2}|B_{R}|$. On the other hand, all particles in
$\eta _{B_{R}}'$ were initially present in $B_{R}$, therefore
$\frac{\lambda }{1+\lambda }(1-\varepsilon )^{2}|B_{R}| \leqslant
\mathbb{E}\|\eta _{0}\|_{V} = \zeta |B_{R}|$. Since $\varepsilon $ was
arbitrary, we must have $\zeta \geqslant \frac{\lambda }{1+\lambda }$.
\end{proof}

\subsection{Stabilization via successive weak stabilizations}

Let $d > 2$, and define
\begin{equation*}
1 < G = \mathbb{E}\left [
\text{visits of a random walk to its starting point} \big . \right ] <
\infty .
\end{equation*}
In this subsection we let $V \ni {\boldsymbol{0}}$ and
$\eta \in (\mathbb{N}_{0})^{V}$ be finite and fixed.

\medskip
In the following arguments we consider a toppling procedure for obtaining
\emph{stabilization and strong stabilization via successive weak stabilizations},
shown in Figure~\ref{fig:flowchart}. Let $T_{V}$ and $T_{V}^{s}$ count
the number of rounds needed for stabilization and strong stabilization
to be achieved, respectively (weak stabilization is always achieved in
the first round). From this definition we have
%
\begin{equation}
\label{eq:sssvwscond}
T_{V}=1 \quad \Longleftrightarrow \quad \eta ^{w}_{V}({
\boldsymbol{0}})=0 \quad \Longleftrightarrow \quad T_{V}^{s}=1 .
\end{equation}

\begin{figure}
\includegraphics{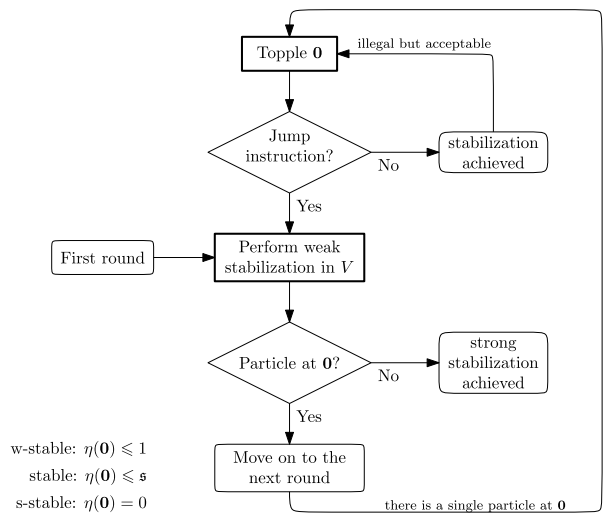}%
\caption{Flow diagram showing a way to obtain stabilization by alternating
between weak stabilizations and legal topplings at
${\boldsymbol{0}}$. If stabilization is achieved with a sleeping particle
at ${\boldsymbol{0}}$, the process can be continued using acceptable topplings
until strong stabilization is also achieved.}\label{fig:flowchart}
\end{figure}

Before we start using this procedure and studying $T_{V}$, we need a couple
of lemmas. We define the ``jump odometer'' $\bar{m}_{V,\eta }$ by counting
only the number of jump instructions performed at each site when $\eta $ is
stabilized in $V$. Define $\bar{m}_{V,\eta }^{s}$ and $\bar{m}_{V,\eta }^{w}$
similarly. Let $\eta ^{+} = \eta + \delta _{{\boldsymbol{0}}}$ denote the
result of adding an active particle at ${\boldsymbol{0}}$ to a configuration
$\eta $.

\begin{lemma}[Strong$-$weak$=$extra particle]\label{lemma:almostcommute}
We have $ \bar{m}^{s}_{V,\eta } = \bar{m}^{w}_{V,\eta ^{+}} $.
\end{lemma}
\begin{proof}
A sequence of topplings $\beta $ is w-legal for $\eta ^{+}$ if and only
if it is s-legal for $\eta $. This proves the lemma.
\end{proof}

\begin{lemma}[Getting rid of the extra particle]\label{lemma:addg}
We have
\begin{equation*}
\mathbb{E}\big [ \bar{m}^{w}_{V,\eta ^{+}}({\boldsymbol{0}}) \big ]
\leqslant G + \mathbb{E}\big [ \bar{m}^{w}_{V,\eta }({\boldsymbol{0}})
\big ] .
\end{equation*}
\end{lemma}
\begin{proof}
Consider the following toppling procedure. First, make the extra particle
at ${\boldsymbol{0}}$ jump until it leaves $V$. Then, weakly stabilize
the resulting configuration (which is $\eta $). Since the resulting configuration
is weakly stable, by Lemma~\ref{lemma:lap2} the jump odometer of this procedure
gives an upper bound for
$\bar{m}^{w}_{V,\eta ^{+}}({\boldsymbol{0}})$. Now the expected number
of visits to ${\boldsymbol{0}}$ in the first stage is bounded by
$G$ (in fact it tends to $G$ as $V$ increases), proving the inequality.
\end{proof}

From the two previous lemmas, we get the following.

\begin{corollary}
We have $ \mathbb{E}\left [ \bar{m}^{s}_{V,\eta }({\boldsymbol{0}}) -
\bar{m}^{w}_{V,\eta }({\boldsymbol{0}}) \Big . \right ] \leqslant G $.
\end{corollary}

We finally derive estimates for $T_{V}$ and $T^{s}_{V}$.

\begin{lemma}
We have $ \bar{m}^{s}_{V,\eta }({\boldsymbol{0}}) \geqslant \bar{m}^{w}_{V,
\eta }({\boldsymbol{0}}) + T^{s}_{V} - 1 $.
\end{lemma}
\begin{proof}
Consider the strong stabilization of $\eta $ on $V$ via successive weak
stabilizations as shown in Figure~\ref{fig:flowchart}, run until strong
stabilization is achieved. The first round starts in the middle of the
diagram and consists of a weak stabilization. For each of the other
$T_{V}^{s}-1$ rounds, a jump instruction is eventually performed at
${\boldsymbol{0}}$ before strong stabilization is achieved.
\end{proof}

From the two previous statements, we get the following.

\begin{corollary}\label{cor:etvbound}
We have $ \mathbb{E}T_{V} \leqslant \mathbb{E}T^{s}_{V} \leqslant 1 + G
\leqslant 2G $.
\end{corollary}

\subsection{Main estimates}

Theorem~\ref{thm:supcritstauffertaggi} will follow from two propositions,
based on the previous properties of stabilization and strong stabilization
via successive weak stabilizations. We now take the initial configuration
$\eta _{0}$ random and i.i.d.

We use the following decomposition:
%
\begin{equation}\label{eq:decomposition}
\mathbb{P}\left ( \big . \eta _{V}'({\boldsymbol{0}}) = \mathfrak{s}\right )
= \sum _{n=2}^{\infty }\mathbb{P}\left ( \big . \eta
'({\boldsymbol{0}})=\mathfrak{s}, T_{V} = n \right ) ,
\end{equation}
where $n=1$ is excluded by~\eqref{eq:sssvwscond}.

\begin{proposition}\label{prop:lowlambda}
$ \mathbb{P}(\eta '_{V}({\boldsymbol{0}}) = \mathfrak{s}) \leqslant 4 G \,
\sqrt{ \lambda } \text{ for every } V \text{ finite} $.
\end{proposition}

\begin{proof}
We want to control the summand in~\eqref{eq:decomposition}. Given that after
$n-1$ rounds there is an active particle at ${\boldsymbol{0}}$, the
conditional probability that the next instruction at ${\boldsymbol{0}}$ is a
sleep or jump instruction remains unaffected. In case it is a sleep
instruction, the procedure stops at round $n$ and $\eta
_{V}'({\boldsymbol{0}}) = \mathfrak{s}$. In case it is a jump instruction,
the procedure continues and might reach round $n+1$. So by induction on $n$
we get, for $n \geqslant 2$,
\begin{align*}
\mathbb{P}(\eta _{V}'({\boldsymbol{0}}) = \mathfrak{s}, T_{V} = n) &
\leqslant \tfrac{\lambda }{1+\lambda } \left ( \tfrac{1}{1+\lambda }
\right )^{n-2}.
\end{align*}
We now split the sum in~\eqref{eq:decomposition} at an arbitrary point
$n_{0}$ to get
\begin{align*}
\mathbb{P}(\eta _{V}'({\boldsymbol{0}}) = \mathfrak{s}) &
\leqslant \sum _{n=2}^{n_{0}} \tfrac{\lambda }{1+\lambda } \left (
\tfrac{1}{1+\lambda } \right )^{n-2} + \sum _{n=n_{0}+1}^{\infty }
\mathbb{P}(T_{V} = n) .
\end{align*}
By simple Markov inequality, minimization over $n_{0}$ and Corollary~\ref{cor:etvbound},
\begin{align*}
\mathbb{P}(\eta _{V}'({\boldsymbol{0}}) = \mathfrak{s}) &
\leqslant \lambda \cdot n_{0} + \frac{\mathbb{E}T_{V}}{n_{0}+1}
\leqslant 2 \sqrt{\lambda \mathbb{E}T_{V}} \leqslant 2 \sqrt{2 G\,
\lambda } ,
\end{align*}
which completes the proof.
\end{proof}

To boost the previous estimate we need a more careful analysis.
\begingroup\allowdisplaybreaks

\begin{proposition}\label{prop:alllambda}
$ \mathbb{P}(\eta '_{V}({\boldsymbol{0}}) = \mathfrak{s}) \leqslant 1 - (2 +
2\lambda )^{-4G} \text{ for every } V \text{ finite} $.
\end{proposition}

\begin{proof}
Let $n \geqslant 2$. We start observing that
%
\begin{equation}\label{eq:sharpweak}
\mathbb{P}\left ( \eta _{V}'({\boldsymbol{0}}) = \mathfrak{s}, \big . T_{V} =
n \,\middle |\, T_{V} \geqslant n \right ) = \tfrac{\lambda }{1+\lambda } .
\end{equation}
We want an upper bound for $\mathbb{P}( \eta _{V}'({\boldsymbol{0}}) =
\mathfrak{s}, T_{V} = n)$ which remains smaller than $1$ when summed over
$n$, even for large $\lambda $. For that we will relate stabilization to
strong stabilization in a way that retains independence.

Let $k \geqslant 2$ be fixed. The event $T^{s}_{V} \geqslant k$ is equal
to the event that, on rounds $1,\dots ,k-1$ of strong stabilization via
successive weak stabilizations shown in Figure~\ref{fig:flowchart}, the
answer to ``Particle at ${\boldsymbol{0}}$?'' is ``Yes.''

Now the main observation is that this event is \emph{independent} of the
number of times the upper cycle (the one where ${\boldsymbol{0}}$ is toppled
persistently until a jump instruction is found) is performed at each of
the rounds $1,\dots ,k-1$. Indeed, inserting or removing sleep instructions
on the stack $(\mathfrak{t}^{{\boldsymbol{0}},j})_{j}$ at the origin
may only add or remove cycles in the upper part of the flow diagram, but
has no effect on the outcome of the lower part.

Hence, for $n<k$ we have
$ \mathbb{P}\left ( \big . T_{V} > n \,\middle |\, T_{V} \geqslant n,T^{s}_{V}
\geqslant k \right ) = \tfrac{1}{1+\lambda } $, which corresponds to the
answer to ``Jump instruction?'' being ``Yes'' right at the beginning of
round $n$. Indeed, this ensures that $T_{V}>n$ because
$T_{V}^{s} \geqslant k$ implies that the next question is also being answered
affirmatively. From this identity, we get
\begin{align*}
\mathbb{P}\left ( \big . T_{V} \geqslant n+1 \,\middle |\, T^{s}_{V}
\geqslant k \right ) = \tfrac{1}{1+\lambda } \, \mathbb{P}\left (
\big . T_{V} \geqslant n \,\middle |\, T^{s}_{V} \geqslant k \right ) .
\end{align*}

Iterating the above equality for $n-1,n-2,\dots $ and using~\eqref{eq:sssvwscond}
yields
\begin{align*}
\mathbb{P}\left ( \big . T_{V} \geqslant n + 1 \,\middle |\, T^{s}_{V}
\geqslant k \right ) = \left ( \tfrac{1}{1+\lambda } \right )^{n-1} \,
\mathbb{P}\left ( \big . T_{V} \geqslant 2 \,\middle |\, T^{s}_{V}
\geqslant k \right ) = \left ( \tfrac{1}{1+\lambda } \right )^{n-1} .
\end{align*}

Finally, taking $k=n+1$, and since $T_{V} \geqslant n+1$ implies
$T^{s}_{V} \geqslant k$, the equality becomes
\begin{equation*}
\mathbb{P}\left ( \big . T_{V} \geqslant n + 1 \right ) = \left (
\tfrac{1}{1+\lambda } \right )^{n-1} \mathbb{P}\left ( \big . T^{s}_{V}
\geqslant n+1 \right ) .
\end{equation*}

Substituting this and~\eqref{eq:sharpweak} into~\eqref{eq:decomposition}
gives
\begin{align*}
\mathbb{P}(\eta _{V}'({\boldsymbol{0}}) = \mathfrak{s}) & = \sum _{n=2}^{
\infty } \tfrac{\lambda }{1+\lambda } \, \mathbb{P}\left ( \big . T_{V}
\geqslant n \right )
\\
& = \tfrac{\lambda }{1+\lambda } \sum _{n=0}^{\infty } \left (
\tfrac{1}{1+\lambda } \right )^{n} \mathbb{P}\left ( \big . T^{s}_{V}
\geqslant n + 2 \right ) .
\end{align*}
Splitting the sum at
$n_{0} = \lfloor 4G \rfloor \geqslant 2 \mathbb{E}[T_{V}^{s}-2]$ by Corollary~\ref{cor:etvbound},
we get
\begin{align*}
\mathbb{P}(\eta _{V}'({\boldsymbol{0}}) = \mathfrak{s}) &
\leqslant \tfrac{\lambda }{1+\lambda } \sum _{n=0}^{n_{0}-1} \left (
\tfrac{1}{1+\lambda } \right )^{n} + \tfrac{\lambda }{1+\lambda } \sum _{n=n_{0}}^{
\infty } \tfrac{1}{2} \left ( \tfrac{1}{1+\lambda } \right )^{n}
\\
& = 1 - \tfrac{1}{2} (1+\lambda )^{-n_{0}} \leqslant 1 - (2 + 2
\lambda )^{-4G} ,
\end{align*}
which concludes the proof.
\end{proof}\endgroup

\begin{proof}[Proof of Theorem~\ref{thm:supcritstauffertaggi}]
Assume that $\zeta > 4 G \sqrt{ \lambda }$. The number $N$ of particles that
exit $V$ during stabilization of $V$ equals $\|\eta _{0}\|_{V}$ minus the
number of sites $z$ such that $\eta '_{V}(z) = \mathfrak{s}$. Using
Proposition~\ref{prop:lowlambda}, $EN \geqslant (\zeta - 4 G \sqrt{\lambda
})\times |V|$, so Condition~\eqref{eq:conditione} is satisfied and therefore
the system a.s.\ stays active. The same argument works assuming $\zeta > 1 -
(2 + 2\lambda )^{-4G}$.
\end{proof}

\section{Uniqueness of the critical density}\label{sec:universality}

In this section we prove Theorem~\ref{thm:universality}, or the following
equivalent formulation.

\begin{theorem}\label{thm:equivalent}
Let $d$, $p(\cdot )$ and $\lambda $ be given. Let $\nu _{1}$ and $\nu _{2}$
be two spatially ergodic distributions on
$(\mathbb{N}_{0})^{\mathbb{Z}^{d}}$, with respective densities $\zeta
_{1}<\zeta _{2}$. If the ARW system is a.s.\ fixating with initial state
${\nu _{2}}$, then it is also a.s.\ fixating with initial state ${\nu _{1}}$.
\end{theorem}

Below we briefly sketch the proof of Theorem~\ref{thm:equivalent}, and
then give the complete proof in three parts: embedding the initial configuration
into another one with higher density, stabilization of the embedded configuration,
and finally stabilization of the original configuration. An interesting
feature that distinguishes the toppling procedure used here is that it
is not a sequential procedure for ever-growing finite domains as in the
previous sections, but rather a sequence of parallel-update type of operation
that topples infinitely many sites at once, in order to use ergodicity
and mass conservation.

\begin{problem}
Suppose $\nu $ is a translation-ergodic active state (active means $\nu $ is
supported on $(\mathbb{N}_{\mathfrak{s}})^{\mathbb{Z}^{d}} \setminus \{0,
\mathfrak{s}\}^{\mathbb{Z}^{d}}$) with density $\zeta > \zeta _{c}$. Show
that the ARW with initial state $\nu $ a.s.\ stays active.
\end{problem}

\begin{problem}
Suppose $\nu $ is a translation-ergodic absorbing state (absorbing means
supported on $\{0,\mathfrak{s}\}^{\mathbb{Z}^{d}}$) with density
$\zeta > \zeta _{c}$. Show that the ARW with initial state $\nu $ plus
one active particle at the origin stays active with positive probability.
Note that this property is false if the graph is a tree~\cite{JohnsonRolla19}.
\end{problem}

\begin{problem}
Under which conditions besides $\zeta <\zeta _{c}$ does
$\mathbb{E}[m({\boldsymbol{0}})] < \infty $?
\end{problem}

\begin{problem}
When does $\mathbb{P}(m({\boldsymbol{0}}) \geqslant 1)=1$ imply
$\mathbb{P}(m({\boldsymbol{0}}) = \infty )=1$?
\end{problem}

Let us give a brief sketch before moving to the proof.

\medskip
The proof is algorithmic and has two stages, both stages being infinite.
The idea is very simple and is related to what is sometimes called decoupling.
Let $\eta _{0}$ and $\xi _{0}$ be independent and distributed as
$\nu _{1}$ and $\nu _{2}$. In the first stage, we evolve $\eta $ starting
from $\eta _{0}$ until it gives a configuration
$\eta _{0}' \leqslant \xi _{0}$. In the second stage, we use the same set
of instructions to evolve both systems. Since the evolution of
$\xi $ starting from $\xi _{0}$ a.s.\ fixates, so does the evolution of
$\eta $ starting from $\eta _{0}'$, concluding the proof. More precisely,
in the first stage we force each particle in the system $\eta $ to move
(by waking it up if needed) until it meets a particle of $\xi _{0}$; once
they meet, they are paired and will not be moved until the second stage.
Even if it takes infinitely many steps to finish pairing globally, a.s.\ every
particle in the system $\eta $ will eventually be paired, and the resulting
odometer will be a.s.\ finite at every site (if the odometer were infinite
somewhere, by ergodicity it would be infinite everywhere, so every particle
in $\xi _{0}$ would be paired, implying
$\zeta _{2} \leqslant \zeta _{1}$ by mass conservation). This yields a
configuration $\eta _{0}' \leqslant \xi _{0}$. In the second stage, we
simply evolve the system using the remaining instructions. Since they are
independent of $\xi _{0}$, $\eta _{0}$, and of the instructions used in
the first stage, by assumption the remaining instructions a.s.\ stabilize
$\xi _{0}$ leaving a locally-finite odometer. By monotonicity of the final
odometer with respect to the configuration, the same set of remaining instructions
also stabilizes $\eta _{0}'$, again with a locally-finite odometer. Adding
the odometer of both stages would give the final locally-finite odometer
given by stabilization of $\eta _{0}$, except that in the first stage we
have not followed the toppling rules correctly. But it still gives an upper
bound due to Lemma~\ref{lemma:lap}.

\medskip
We now turn to the proof.

\medskip
To make the argument precise we will not exactly move particles as in the
previous sketch, since the embedding requires an infinite number of topplings.
We instead explore the instructions and define a sequence of configurations
in terms of $\eta _{0}$, $\xi _{0}$ and $\mathcal{I}$. We end up concluding
that a.s.\ the result of this exploration implies that $\eta _{0}$ is stabilizable,
which in turn implies the statement of the theorem.

\subsubsection*{Embedding of the smaller configuration}

Without loss of generality, we assume that $\nu _{1}$ or $\nu _{2}$ is
not only ergodic but also mixing (otherwise consider $\nu _{3}$ as i.i.d.\ Poisson
with mean $\frac{\zeta _{1}+\zeta _{2}}{2}$ which is mixing, and apply
the result from $\nu _{2}$ to $\nu _{3}$ and from $\nu _{3}$ to
$\nu _{1}$). So suppose $\nu _{2}$ is mixing. Recall that
$\mathcal{I}$ is an i.i.d.\ field, thus the pair
$(\xi _{0},\mathcal{I})$ is mixing and hence the triple
$\omega =(\eta _{0},\xi _{0},\mathcal{I})$ is ergodic, see \S
\ref{sub:mtperg}.

Let $\eta _{0}$, $\xi _{0}$ and $\mathcal{I}$ be given, and take
$h_{0} \equiv 0$. Fix some $k=0,1,2,3,4,\dots $ and suppose
$\eta _{k}$ and $h_{k}$ have been defined as a factor of $\omega $ (see
\S \ref{sub:mtperg}). Denote by $A_{k}$ the set given by
\begin{equation*}
A_{k} = \{ x : \eta _{k}(x)>\xi _{0}(x) \} ,
\end{equation*}
and consider an arbitrary enumeration
\begin{equation*}
A_{k}=\{x_{1}^{k},x_{2}^{k},x_{3}^{k},\dots \}.
\end{equation*}
Let
%
\begin{equation}\label{eq:etahlimit}
(\eta _{k}^{j},h_{k}^{j}) = \Phi _{(x_{1}^{k},x_{2}^{k},\dots ,x_{j}^{k})}
(\eta _{k},h_{k}) \quad \text{and} \quad (\eta _{k+1},h_{k+1}) = \lim _{j}
(\eta _{k}^{j},h_{k}^{j})
\end{equation}
in case $A_{k}$ is infinite -- in case it is finite, by ergodicity it is
a.s.\ empty in which case we let $(\eta _{k+1},h_{k+1}) = (\eta _{k},h_{k})$.
Note that the condition $\eta _{k}(x)>\xi _{0}(x)$ is also satisfied when
$\eta _{k}(x)=\mathfrak{s}$ and $\xi _{0}(x)=0$, so this operation may
require waking up particles. That is, in going from $(\eta _{k},h_{k})$ to
$(\eta _{k+1},h_{k+1})$, every site in $A_{k}$ is toppled once. These
topplings are legal when $\eta _{k}(x) \geqslant 1$, and they are acceptable
but illegal in case $\eta _{k}(x) = \mathfrak{s}> 0 = \xi _{0}(x)$.

As we go through $j=1,2,3,\dots $ in~\eqref{eq:etahlimit}, for each
$j$ the field $h_{k}^{j}$ is increased by one unit at $x_{j}^{k}$, so
$h_{k+1}$ is well-defined and satisfies
\begin{equation*}
h_{k+1}(x) = h_{k}(x) + \mathds{1}_{A_{k}}(x).
\end{equation*}
To see that $\eta _{k}$ is also well-defined in the limit~\eqref{eq:etahlimit},
observe that, for each site $x$, the sequence
$(\eta _{k}^{j}(x))_{j}$ decreases for at most one value of $j$. In case
it decreases, it may send one particle to another site $z \ne x$. Thus,
by a standard use of the mass transport principle, the configuration at
each site $x$ increases a finite number of times (the expected number of
times is less than one, as can be seen by taking $f(x,y)$ as the indicator
of the event that $x \in A_{k}$ and toppling $x$ sends a particle to
$y$), so the limit $\eta _{k}$ is a.s.\ finite. Since there are a.s.\ finitely
many sites $z \in A_{k}$ that send a particle to $x$ when toppled, it follows
from the local Abelian property that the limit $(\eta _{k},h_{k})$ does
not depend on the enumeration of $A_{k}$, so it is a factor of
$\omega $.

Let $k \to \infty $ and define
\begin{equation*}
h_{0}'(x)=\lim _{k} h_{k}(x).
\end{equation*}
We now prove that, if
$\mathbb{P}(h_{0}'({\boldsymbol{0}})=\infty )>0$ then we must have
$\zeta _{1} \geqslant \zeta _{2}$.

First, we claim that
$\mathbb{P}(h_{0}'({\boldsymbol{0}})=\infty )=0 \text{ or } 1$. Since
$(\eta _{k},h_{k})$ is a factor of $\omega $, it is translation-ergodic.
Moreover, a.s.\ the event that $h_{0}'({\boldsymbol{0}})=\infty $ implies
the event that $h_{0}'(z)=\infty $ for every $z$ such that $p(z)>0$. These
two facts together imply that, either $h_{0}'(x)=\infty $ a.s.\ for every
$x$, or $h_{0}'(x)<\infty $ a.s.\ for every $x$, see \S
\ref{sub:condbu}. This proves the zero-one law.

Suppose $h_{0}'({\boldsymbol{0}})=\infty $ with positive probability.
By the zero-one law we have $h_{0}'({\boldsymbol{0}})=\infty $ a.s., which
means that
$\mathbb{P}(\limsup _{k} \{{\boldsymbol{0}}\in A_{k}\})=1$. But if
${\boldsymbol{0}}\in A_{k_{0}}$ for some $k_{0}$, then necessarily
$\eta _{k_{0}-1}({\boldsymbol{0}}) > \xi _{0}({\boldsymbol{0}})$, thus
$\eta _{k}({\boldsymbol{0}}) \geqslant \xi _{0}({\boldsymbol{0}})$ for
all $k \geqslant k_{0}$ by definition of $A_{k}$, and therefore
$\liminf _{k} |\eta _{k}({\boldsymbol{0}})| \geqslant |\xi _{0}({
\boldsymbol{0}})|$. On the other hand, from the mass transport principle
we have
$\mathbb{E}|\eta _{k}({\boldsymbol{0}})| = \mathbb{E}| \eta _{k-1}({
\boldsymbol{0}}) | = \dots = \mathbb{E}| \eta _{0}({\boldsymbol{0}})
| = \zeta _{1}$ (to show the first identity, we let $f_{k}(x,y)$ be the
indicator that, on step $k$, $x$ sends a particle to $y$, and let
$f_{k}(x,x)$ be the number of particles that were present at $x$ at the
beginning of stage $k$ and stayed at $x$). By Fatou's Lemma,
$\zeta _{1} \geqslant \zeta _{2}$.

Since we are assuming $\zeta _{1}<\zeta _{2}$, we must have
$h_{0}'({\boldsymbol{0}})<\infty $ a.s. Now, as we go through
$k=1,2,3,\dots $, the value of $\eta _{k}({\boldsymbol{0}})$ can decrease
only when ${\boldsymbol{0}}\in A_{k}$, \textit{i.e.}\ only when
$h_{k}({\boldsymbol{0}})$ increases. Hence,
$(\eta _{k}({\boldsymbol{0}}))_{k}$ is a.s.\ eventually non-decreasing,
so it converges. Its limit $\eta _{0}'({\boldsymbol{0}})$ satisfies
$\eta _{0}'({\boldsymbol{0}}) \leqslant \xi _{0}({\boldsymbol{0}})$,
otherwise ${\boldsymbol{0}}$ would be in $A_{k}$ for all large enough
$k$ and $h_{0}'({\boldsymbol{0}})$ would be infinite. By translation invariance,
a.s.\ $h_{0}'(x)<\infty $ and
$\eta _{0}'(x) = \lim _{k}\eta _{k}(x) \leqslant \xi _{0}(x)$ for every
$x$.

\subsubsection*{Stabilization of the original configuration}

In the previous stage we obtained a pair $(\eta _{0}',h_{0}')$ a.s.\ satisfying
$h_{0}'(x)<\infty $ and $\eta _{0}'(x) \leqslant \xi _{0}(x)$ for every
$x \in \mathbb{Z}^{d}$. Let $\tilde{\mathcal{I}}$ be the set of instructions
given by
\begin{equation*}
\tilde{\mathfrak{t}}^{x,j} = \mathfrak{t}^{x,h_{0}'(x)+j}, \qquad x
\in \mathbb{Z}^{d}, j\in \mathbb{N}.
\end{equation*}
that is, the field obtained by deleting the instructions used in the embedding
stage described above. Since the first $h_{0}'(x)$ instructions have been
deleted at each site $x$, stabilizing a system with the instructions in
$\tilde{\mathcal{I}}$ instead of $\mathcal{I}$ is equivalent to starting
with odometer at $h_{0}'$ instead of $h_{0} \equiv 0$.

Now note that the collection of instructions
$\big ( \mathfrak{t}^{x,j} : x\in \mathbb{Z}^{d}, j > h_{0}'(x)
\big )$ played no role in the construction of $\eta _{0}'$ and
$h_{0}'$, so they are independent of $\xi _{0}$ and $h_{0}'$. Hence,
$\tilde{\mathcal{I}}$ is an i.i.d.\ field just like $\mathcal{I}$, and
it is also independent of $\xi _{0}$.

Therefore,
$\mathbb{P}\big (\xi _{0}
\text{ is $\tilde{\mathcal{I}}$-stabilizable}\big ) = \mathbb{P}\big (
\xi _{0} \text{ is $\mathcal{I}$-stabilizable}\big )$, and the latter equals
$1$ by assumption. Since $\eta _{0}' \leqslant \xi _{0}$, we have
$\mathbb{P}\big (\eta _{0}'
\text{ is $\tilde{\mathcal{I}}$-stabilizable}\big ) \geqslant\break
\mathbb{P}\big (\xi _{0}
\text{ is $\tilde{\mathcal{I}}$-stabilizable}\big ) = 1$. This means that
a.s.\ there exists $h_{1}'$ such that, for all finite
$V \subseteq \mathbb{Z}^{d}$ and all $x\in \mathbb{Z}^{d}$,
$m_{V,\eta _{0}';\tilde{\mathcal{I}}}(x) \leqslant h_{1}'(x) <
\infty $.

\subsubsection*{Stabilization of the original configuration}

Given the properties of the two previous stages, we now give the (perhaps
tedious) proof that $\eta _{0}$ is a.s.\ $\mathcal{I}$-stabilizable. More
precisely, we will show that
\begin{equation*}
m_{\eta _{0};\mathcal{I}}(x) \leqslant h_{0}'(x) + h_{1}'(x) <
\infty , \quad \forall \ x \in \mathbb{Z}^{d}.
\end{equation*}

In the first stage, the limits $\eta _{0}'$ and $h_{0}'$, which are determined
by $\eta _{0}$, $\xi _{0}$ and $\mathcal{I}$, almost surely exist and
satisfy $h_{0}'<\infty $ and $\eta _{0}' \leqslant \xi _{0}$. Suppose this
event occurs, and let $V$ be a fixed finite set.

If we start from $(\eta _{0},h_{0})$ and perform all topplings in
$V$ as well as particle additions to $V$ (coming from $V^{c}$), following
the same order as in the first stage, only a finite number of operations
will be performed, and we end up with a configuration that equals
$(\eta _{0}',h_{0}')$ on $V$.

By the local Abelian property, we can add the particles first and then
topple the sites in $V$ as in the first stage, obtaining the same result.
This means that there is some $\bar{\eta }_{V} \geqslant \eta _{0}$ and an
acceptable sequence $\alpha _{V}=(x_{1},\dots ,x_{n})$ for
$(\bar{\eta }_{V},h_{0})$ such that $m_{\alpha _{V}} = h_{0}'$ on $V$ and
\begin{equation*}
\Phi _{\alpha _{V}} (\bar{\eta }_{V},h_{0}) = (\eta _{0}',h_{0}')
\text{ on } V.
\end{equation*}

Now, in the second stage, we showed that a.s.\ there exists
$h_{1}'(x) < \infty $ such that
$m_{V',\eta _{0}';\tilde{\mathcal{I}}}(x) \leqslant h_{1}'(x)$ for every
finite $V'$. Suppose this event occurs.

Notice that
$m_{V,\eta _{0}',h_{0}';{\mathcal{I}}}(x) = m_{V,\eta _{0}';
\tilde{\mathcal{I}}}(x)$, that is, to stabilize $\eta _{0}'$ in $V$ using
the shifted field of instructions is the same as stabilize
$\eta _{0}'$ in $V$ using the original field of instructions and shifted
odometer. Therefore, there exists
$\beta _{V}=(x_{n+1},\dots ,x_{m})$ contained in $V$ such that
$m_{\beta _{V}} \leqslant h_{1}'$ on $V$ and
$\Phi _{\beta _{V}} (\eta _{0}',h_{0}')$ is stable in $V$.

By the above identity,
$ \Phi _{\beta _{V}} \circ \Phi _{\alpha _{V}} (\bar{\eta }_{V},h_{0}) =
\Phi _{\beta _{V}} (\eta _{0}',h_{0}') $, on $V$. Since the latter is
stable in $V$, by Lemma~\ref{lemma:lap} we have
\begin{equation*}
m_{V,\bar{\eta }_{V};\mathcal{I}}(x) \leqslant m_{\alpha _{V}}(x) + m_{
\beta _{V}}(x), \quad \text{ for all } x\in \mathbb{Z}^{d}.
\end{equation*}
Thus, by monotonicity,
\begin{equation*}
m_{V,\eta _{0};\mathcal{I}}(x) \leqslant m_{V,\bar{\eta }_{V};
\mathcal{I}}(x) \leqslant m_{\alpha _{V}}(x) + m_{\beta _{V}}(x)
\leqslant h_{0}'(x) + h_{1}'(x) < \infty .
\end{equation*}
We now note that the above bound does not depend on $V$, so
\begin{equation*}
m_{\eta _{0};\mathcal{I}}(x) = \sup _{V \text{ finite}} m_{V,\eta _{0};
\mathcal{I}}(x) \leqslant h_{0}'(x) + h_{1}'(x) < \infty , \quad
\forall \ x \in \mathbb{Z}^{d},
\end{equation*}
which means that $\eta _{0}$ is stabilizable, concluding the proof of Theorem~\ref{thm:equivalent}.

\section{A recursive multi-scale argument}
\label{sec:multiscale}

In this section we comment on a multi-scale argument used to prove the
following.

\begin{theorem}
If the jumps are unbiased, $\zeta _{c}>0$ for every $\lambda >0$.
\end{theorem}

We give an overview of the general strategy, referring the reader to~\cite{SidoraviciusTeixeira17}
for the complete argument. Note that the above theorem is a particular
case of Theorem~\ref{thm:subcritstauffertaggi}.

\medskip
The main step is to show that an initial configuration restricted to a
very large box stabilizes within a slightly larger box, with high probability.
This is then used to show Condition~\eqref{eq:conditionu}. This is proved
by recursion on the scale of the box, and in fact the proof does not rely
much on specific details of the actual ARW dynamics. In a sense, this kind
of approach fits to our intuition that no matter how big a defect is, it
will only affect a neighborhood of comparable size.

The box at scale $k$ is a cube $V_{k}$ of side length $L_{k}$, defined
as follows. Let $\delta =\frac{1}{10}$, $L_{0}=10^{4}$ and
\begin{equation*}
L_{k+1} = \lfloor L_{k}^{\delta }\rfloor ^{2} L_{k}.
\end{equation*}
Notice that $L_{k}$ increases as a doubly exponential of $k$. We also define
$R_{k+1} = \lfloor L_{k}^{\delta }\rfloor L_{k}$ as an intermediate scale
between $L_{k}$ and $L_{k+1}$. In Figure~\ref{fig:stbox}, we see an
\emph{inner box} $V_{k+1}'$, an \text{intermediate box}, and a
\emph{full box} $V_{k+1}$ of level $k+1$.

Let $p_{k}$ denote the probability that, starting from a Poisson configuration
in $V_{k}'$, some particle exits $V_{k}$. That $p_{k} \to 0$ fast as
$k\to \infty $ follows from the recursion relation\vspace{-3pt}
\begin{equation*}
p_{k+1} \leqslant \frac{L_{k+1}^{2d}}{L_{k}^{2d}} \ {p_{k}}^{2} + e_{k+1},
\end{equation*}
consisting of a combinatorial term, the probability ${p_{k}}^{2}$ that
stabilization fails twice at scale $k$, and the probability
$e_{k+1}$ that something goes wrong at scale $k+1$. Indeed, if
$p_{k_{0}}$ is small enough and $e_{k} \to 0$ fast enough, then the square
power above beats the $1+2\delta $ power in the definition of
$L_{k}$, and $p_{k}$ vanishes doubly-exponentially fast in $k$.

Let us describe some aspects of this recursion step, depicted in Figure~\ref{fig:renorm}.

\begin{figure}
\includegraphics{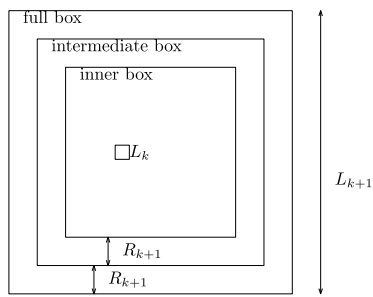}
\caption{Boxes and scales; $L_{k} \ll R_{k+1} \ll L_{k+1}$.}
\label{fig:stbox}\vspace{-9pt}
\end{figure}

Configurations in light gray have Poisson product distribution with the
right density. They are restricted to the inner box of level $k+1$ for
the initial configuration, and the inner boxes of level $k$ for the ``sieved
configurations''. Configurations in dark gray are absorbing configurations,
typically attained by the dynamics. Configurations in gray with a grid
are ``balanced configurations''. Thick arrows represent typical events,
while thin arrows represent events of low probability, either
$e_{k+1}$ or $p_{k}$.

\textbf{Starting fresh.} To let the dynamics run on boxes of the previous
scale and use recursion, it is important to start with a Poisson product
distribution within their inner boxes. This is achieved by a sieving procedure
described below.

\textbf{Worst case scenario.} If the dynamics does not to stabilize all
the $\frac{L_{k+1}^{d}}{L_{k}^{d}}$ boxes of level $k$, the configuration
inside these boxes is no longer i.i.d.\ Poisson. In the absence of any
useful knowledge about the resulting distribution of particles in this
case, we use only the fact that the total number of particles within each
box is still a Poisson random variable, and thus cannot be much larger
than its mean. A \emph{balanced configuration} is such that the number of
particles within each box of level $k$ is appropriately bounded.

\textbf{Sieving procedure.} Starting from a balanced configuration, we let
each particle move for a certain time, so as to uniformize its relative
position within whichever level-$k$ box contains it. After performing all
these jumps, if the particle happens not to be in the inner box of the
level-$k$ box containing it, we repeat the procedure again, as many times
as needed. This \emph{reshuffling} with \emph{sieving} results in a state
that with high probability, can be coupled with an i.i.d.\ Poisson configuration.
This is one of the heaviest parts in~\cite{SidoraviciusTeixeira17}. In
order for this coupling to be possible, a slight increase in the density
is necessary, analogous to the \emph{sprinkling} technique in percolation
(this increase should decay just fast enough to be summable over $k$).

\textbf{The chain of events.} By hypothesis, we start with a Poisson product
measure inside the inner box $V_{k+1}'$. Such a configuration is typically
balanced, that is, each box of level $k$ has an appropriately bounded number
of particles. We then let these particles move around using acceptable
topplings so that their distribution is now close to i.i.d.\ Poisson inside
the inner boxes of level $k$ (the sieving procedure). During this procedure
they cannot exit the intermediate box. We now let the evolution run normally
within each box of level $k$, and typically each box stabilizes nicely
without letting particles leave. It may happen however that some of these
boxes of level $k$ is not stabilized as intended (which is atypical). In
spite of failing to stabilize as intended, the resulting configuration
is still balanced. So we let the particles move around again, now obtaining
a sieved configuration in the full box $V_{k+1}$. Typically, the resulting
configuration is properly sieved. The system is then given a second chance
to stabilize. This second attempt will typically be successful, but may
fail again if some of the level-$k$ boxes does not stabilize as expected.

In the proof there are many aspects to keep under control, and many delicate
statements that we omit here. The above description is not intended to
serve as a sketch of proof, but hopefully gives a general flavor of the
main argument. For all the details, the reader is referred to~\cite{SidoraviciusTeixeira17}.

\begin{figure}
\includegraphics{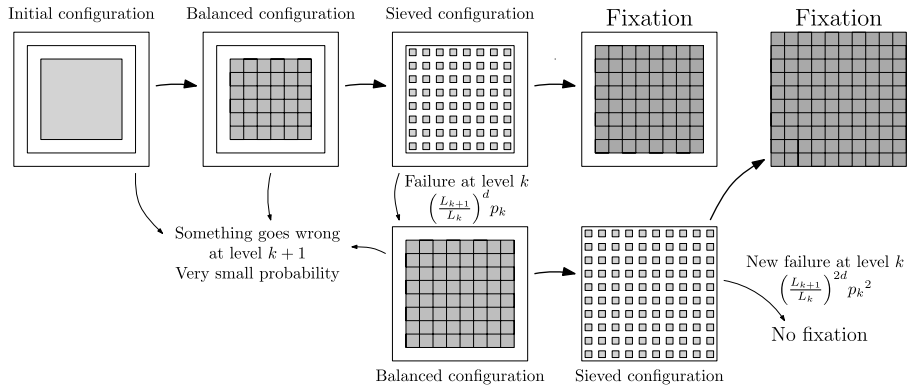}
\caption{Illustrative diagram of events for the recursion relation.}
\label{fig:renorm}
\vspace{-6pt}
\end{figure}

\section{Arguments using particle-wise constructions}
\label{sec:particlewise}

The techniques presented so far used the site-wise representation and its
properties, as described in~\S \ref{sec:definitions}. In the particle-wise
construction, the randomness of the jumps is not attached to the sites,
but to the particles.\eject

In this section we use a particle-wise construction to prove Theorem~\ref{thm:conditione}
and the following ones. We always assume that
$\mathbb{E}|\eta _{0}({\boldsymbol{0}})|<\infty $.

\begin{theorem}[Mass conservation]\label{thm:conservation}
Consider the ARW on $\mathbb{Z}^{d}$ with i.i.d.\ initial configuration
$\eta _{0}$. If the system a.s.\ fixates, then
$\mathbb{E}|\eta _{\infty }({\boldsymbol{0}})| = \mathbb{E}|\eta _{0}({
\boldsymbol{0}})|$, where
$\eta _{\infty }(x) = \lim \limits _{t\to \infty } \eta _{t}(x)$.
\end{theorem}

Now if the system fixates, then
$\eta _{\infty }({\boldsymbol{0}})\in \{0,\mathfrak{s}\}^{\mathbb{Z}^{d}}$,
hence by mass conservation
$\zeta = \mathbb{E}|\eta _{0}({\boldsymbol{0}})| = \mathbb{E}|
\eta _{\infty }({\boldsymbol{0}})| \leqslant 1$. This gives the following
corollary.

\begin{corollary}
\label{cor:muclessone}
$\zeta _{c} \leqslant 1$ for every $\lambda $.
\end{corollary}

Combining the above with Theorems~\ref{thm:monotone} and~\ref{thm:subcritstauffertaggi}
we get the following.

\begin{corollary}
For $\lambda =\infty $, $\zeta _{c} = 1$.
\end{corollary}

The process with $\lambda =\infty $ is discussed in \S
\ref{sub:resampling}, where we also consider an equivalent model to prove
the following.

\begin{theorem}
\label{thm:abactive}
For i.i.d.\ $\eta _{0}$ with density $\zeta =1$ and positive variance,
for all $\lambda \in (0,\infty ]$, the ARW a.s.\ stays active.
\end{theorem}

The requirement of positive variance cannot be waived. Indeed, if
$\lambda =\infty $ and $\eta _{0} \equiv 1$ then the configuration is already
stable at time zero.

\begin{problem}
Prove Theorem~\ref{thm:conservation} without using the particle-wise construction.
\end{problem}

\begin{problem}
Prove Theorems~\ref{thm:conservation} and~\ref{thm:abactive} replacing
the the i.i.d.\ assumption by translation ergodicity.
\end{problem}

\begin{remarkVTEX}
We highlight once more that, for unbiased walks on $\mathbb{Z}^{2}$, Corollary~\ref{cor:muclessone}
is the best bound we have, even for small $\lambda $.
\end{remarkVTEX}

We now give a formal description of the particle-wise construction, and
then move on to proving the above results.

\subsection{Description and basic properties}
\label{sub:pwconstr}

We want to revisit the description of \S \ref{sub:arwdef} and now interpret
that \emph{particles are labeled}. Each existing particle at time
$t=0$ is assigned a label $(x,j)$, where $x\in \mathbb{Z}^{d}$ denotes
its starting position and $j=1,\dots ,|\eta _{0}(x)|$ distinguishes particles
starting at the same site $x$. Let
$Y^{x,j}=(Y^{x,j}_{t})_{t \geqslant 0}$ be given by the position of particle
$(x,j)$ at each time $t$. Let
$\gamma ^{x,j}=(\gamma ^{x,j}(t))_{t \geqslant 0}$ be given by
$\gamma ^{x,j}(t) = 1$ if particle $(x,j)$ is active at time $t$ or
$\gamma ^{x,j}(t) = \mathfrak{s}$ if it is sleeping. Write
$\mathbf{Y}=(Y^{x,j})_{x,j}$ and
$\boldsymbol{\gamma }=(\gamma ^{x,j})_{x,j}$. Then the triple
$\boldsymbol{\eta }=(\eta _{0},\mathbf{Y},\boldsymbol{\gamma })$ describes
the whole evolution of the system.

Whereas the process $(\eta _{t})_{t \geqslant 0}$, given by
\begin{equation*}
\eta _{t}(z)=\sum _{x} \sum _{j \leqslant |\eta _{0}(x)|} \delta _{Y^{x,j}(t)}(z)
\cdot \gamma ^{x,j}(t),
\end{equation*}
only counts the number of particles at a given site at a given time, having
each particle labeled gives a lot more information and allows different
techniques to be employed.

For a system whose initial configuration contains finitely many particles,
the evolution described above is always well-defined. It is a simple continuous-time
Markov chain on a countable space, and many different explicit constructions
will produce $\boldsymbol{\eta }$ with the correct distribution. We now
describe one which can be extended to infinite initial configurations and
has proved particularly useful.

\subsubsection*{The particle-wise construction}

Assign to each particle $(x,j)$ a continuous-time walk
$X^{x,j}=(X^{x,j}_{t})_{t \geqslant 0}$, independently of anything else,
as well as a Poisson clock
$\mathcal{P}^{x,j}\subseteq \mathbb{R}_{+}$ according to which the particle
will try to sleep. $X^{x,j}$ is the path of the particle parameterized
by its \emph{inner time}, which may be slowed down with respect to the system
time, depending on the interaction with other particles (denoting the inner
time of particle $(x,j)$ at time $t$ by $\sigma ^{x,j}(t)$ we have
$Y^{x,j}_{t}=X^{x,j}_{\sigma ^{x,j}(t)}$). So $X^{x,j}$ will be called
the \emph{putative trajectory} of particle $(x,j)$. Write
$\mathbf{X}=(X^{x,j})_{x,j}$ and
$\boldsymbol{\mathcal{P}}=(\mathcal{P}^{x,j})_{x,j}$.

For a deterministic initial configuration $\xi $ containing finitely many
particles, $\boldsymbol{\eta }$ is a.s.\ determined by
$(\xi ,\mathbf{X},\boldsymbol{\mathcal{P}})$ in the obvious way. The
construction for infinite $\eta _{0}$ is done via limits over the sequences
of balls $(B^{y}_{n})_{n\in \mathbb{N}}$, centered at each site
$y\in \mathbb{Z}^{d}$. This family of sequences is countable and translation-invariant.
For $\eta \in (\mathbb{N}_{\mathfrak{s}})^{\mathbb{Z}^{d}}$ and finite
$V \subseteq \mathbb{Z}^{d}$, let
$\eta ^{V} = \eta \cdot \mathds{1}_{V}$ denote the restriction of
$\eta $ to $V$.

\begin{definition}[Well-definedness]
\label{def:pwwd}
We say that the above construction is \emph{well-defined} if: (i) for each
$x,y\in \mathbb{Z}^{d}$, $j\in \mathbb{N}$ and $t>0$, both
$(Y^{x,j}_{s})_{s\in {[0,t]}}$ and
$(\gamma ^{x,j}_{s})_{s\in {[0,t]}}$ are the same in the systems
$(\eta _{0}^{B^{y}_{n}},\mathbf{X},\boldsymbol{\mathcal{P}})$ for all
but finitely many $n$; (ii) the limiting process
$\boldsymbol{\eta }=(\eta _{0},\mathbf{Y},\boldsymbol{\gamma })$ does not
depend on $y$.
\end{definition}

The benefit of requiring the limit not to depend on $y$ is that
$\boldsymbol{\eta }$ is a factor of $\omega $ (see \S \ref{sub:mtperg}).
In particular, the system with labeled particles is translation-ergodic,
satisfies the mass transport principle, and probabilities of local events
can be approximated by finite systems regardless of which construction
is used. This last property implies that different explicit constructions
all yield a process $(\eta _{t})_{t\geqslant 0}$ with the same law.

\begin{theorem}
\label{thm:welldefined}
If $\sup _{x} \mathbb{E}|\eta _{0}(x)| < \infty $, then the above particle-wise
construction is a.s.\ well-defined.
\end{theorem}

The proof is deferred to \S \ref{sub:particleexistence}.

\subsubsection*{Fixation equivalence and mass conservation}

Fixation as defined in \S \ref{sub:asptsoc} concerns the state of sites.
Now that the particles are being labeled, it makes sense to consider fixation
of particles. We say that \emph{particle $(x,j)$ stays active} if
$|\eta _{0}(x)| \geqslant j$ and $\gamma ^{x,j}(t)$ is not eventually
$\mathfrak{s}$. Likewise, we say that \emph{site $x$ stays active} if
$\eta _{t}(x)$ is not eventually constant.

\begin{theorem}
\label{thm:agg}
Suppose $\eta _{0}$ is i.i.d. The following are equivalent:
\begin{enumerate}[$(i)$]
\item $\mathbf{P}( \text{some site stays active})>0$;
\item $\mathbf{P}( \text{all sites stay active})=1$;
\item $\mathbb{P}( \text{all particles stay active})=1$;
\item
\label{item:none}%
$\mathbb{P}( \text{some particle stays active})>0$.
\end{enumerate}
\end{theorem}

The proof is given in \S \ref{sub:agg}. We are ready to show mass conservation.

\begin{proof}[Proof of Theorem~\ref{thm:conservation}]
We use the construction provided by Theorem~\ref{thm:welldefined}. If particle
$(x,j)$ fixates, then $\sigma ^{x,j}(t)$ and $Y^{x,j}_{t}$ are eventually
constant, and we say that particle $(x,j)$ fixates at site
$Y^{x,j}_{\infty }$. Let $A(x,j,y)$ denote the event that particle
$(x,j)$ fixates at site $y$, and let
$f(x,y)=\sum _{j} \mathds{1}_{A(x,j,y)}$.

Assume that a.s.\ all sites fixate. Note that
$\eta _{\infty }({\boldsymbol{0}})=\mathfrak{s}$ if and only if we have
$\sum _{y} f(y,{\boldsymbol{0}})=1$, otherwise
$\eta _{\infty }({\boldsymbol{0}})=0$ and
$\sum _{y} f(y,{\boldsymbol{0}})=0$. On the other hand, by Theorem~\ref{thm:agg}
a.s.\ no particles stay active, whence
$\sum _{y} f({\boldsymbol{0}},y)$ equals
$|\eta _{0}({\boldsymbol{0}})|$. Applying the mass transport principle
concludes the proof.
\end{proof}

\subsection{Averaged condition for activity}
\label{sub:conditione}

In this subsection we prove Theorem~\ref{thm:conditione}. We are assuming
that $\eta _{0}$ is i.i.d. We can moreover assume that
$\mathbb{E}|\eta _{0}({\boldsymbol{0}})| < \infty $ (otherwise truncate
$\eta _{0}$ and use Corollary~\ref{cor:muclessone} combined with Lemma~\ref{lemma:monotonicity}).

The variable $M_{n}$ in Condition~\eqref{eq:conditione} refers to the site-wise
representation of a finite system restricted to $V_{n}$. This is equivalent
to a particle-wise construction with particles being killed when they exit
$V_{n}$. To show that Condition~\eqref{eq:conditione} implies non-fixation,
we consider a different variable $M_{n}^{*}$ which counts how many labeled
particles start in $V_{n}$ and ever visit $V_{n}^{c}$ during the evolution
of the infinite system without killing. More precisely, we consider the
system with labeled particles as provided by Theorem~\ref{thm:welldefined}.

\begin{lemma}
\label{lem:monotone_particle}
$\mathbb{E}M_{n} \leqslant \mathbb{E}M_{n}^{*}$.
\end{lemma}
The proof of Lemma~\ref{lem:monotone_particle} is deferred to \S
\ref{sub:monotone_particle}. Define
\begin{equation*}
\widetilde{V}_{n}=V_{n-L_{n}},
\end{equation*}
where $L_{n}$ is an integer sequence (e.g.~$\lfloor \log n\rfloor $) such
that
\begin{equation*}
L_{n} \to \infty \qquad \text{ but } \qquad
\frac{|V_{n}\setminus \widetilde{V}_{n}|}{|V_{n}|}\to 0.
\end{equation*}

For $n\in \mathbb{N}$, introduce the event
\begin{equation*}
\mathcal{A}_{n} = \text{``} \textstyle \sup _{t} |Y^{{\boldsymbol{0}},1}_{t}|
\geqslant L_{n}\text{''} =
\text{``particle $({\boldsymbol{0}},1)$ reaches distance $L_{n}$''},
\end{equation*}
where the requirement that
$|\eta _{0}({\boldsymbol{0}})|\geqslant 1$ is implicit.

Let $\widetilde{M}_{n}^{*}$ be the number of labeled particles starting
in $\widetilde{V}_{n}$ which ever exit $V_{n}$. By translation invariance
and particle exchangeability, for every $K$,
\begin{align*}
\mathbb{E}\widetilde{M}_{n}^{*} & = \sum _{x\in \widetilde{V}_{n}}
\sum _{i\in \mathbb{N}} \mathbb{P}(
\text{particle $Y^{x,i}$ exits $V_{n}$})
\\
& \leqslant \sum _{x\in \widetilde{V}_{n}} \sum _{i\in \mathbb{N}}
\mathbb{P}(\eta _{0}(x)\geqslant i
\text{ and particle $Y^{x,i}$ reaches distance $L_{n}$ from $x$})
\\
& = |\widetilde{V}_{n}|\sum _{i\in \mathbb{N}}\mathbb{P}(\eta _{0}({
\boldsymbol{0}})\geqslant i
\text{ and particle $Y^{{\boldsymbol{0}},i}$ reaches distance
$L_{n}$ from ${\boldsymbol{0}}$})
\\
& = |\widetilde{V}_{n}|\sum _{i\in \mathbb{N}}\mathbb{P}(\eta _{0}({
\boldsymbol{0}})\geqslant i
\text{ and particle $Y^{{\boldsymbol{0}},1}$ reaches distance
$L_{n}$ from ${\boldsymbol{0}}$} )
\\
& \leqslant |\widetilde{V}_{n}|\sum _{1\leqslant i\leqslant K}
\mathbb{P}(\mathcal{A}_{n})+|\widetilde{V}_{n}|\sum _{i> K}
\mathbb{P}(\eta _{0}({\boldsymbol{0}})\geqslant i)
\\
& = |\widetilde{V}_{n}| K\, \mathbb{P}(\mathcal{A}_{n})+|
\widetilde{V}_{n}|\,\mathbb{E}[(|\eta _{0}({\boldsymbol{0}})|-K)^{+}]
\end{align*}
\noindent%
Hence,
\begin{multline*}
\mathbb{E}M_{n}^{*} \leqslant \mathbb{E}\widetilde{M}_{n}^{*} +
\zeta \, |V_{n}\setminus \widetilde{V}_{n}| \leqslant
\\
\leqslant K \, |\widetilde{V}_{n}| \, \mathbb{P}(\mathcal{A}_{n}) + |
\widetilde{V}_{n}| \, \mathbb{E}[(|\eta _{0}({\boldsymbol{0}})|-K)^{+}]
+ \zeta \, |V_{n}\setminus \widetilde{V}_{n}| ,
\end{multline*}
and using Lemma~\ref{lem:monotone_particle},
\begin{multline*}
\mathbb{P}(\text{particle $({\boldsymbol{0}},1)$ stays active}) =
\lim _{n} \mathbb{P}(\mathcal{A}_{n}) \geqslant
\\
\geqslant \frac{1}{K}\Big (\limsup _{n}
\frac{\mathbb{E}M_{n}}{|\widetilde{V}_{n}|}-\mathbb{E}[(\eta _{0}({
\boldsymbol{0}})-K)_{+}]\Big )
\end{multline*}
which is positive provided $K$ is chosen large enough.

From Theorem~\ref{thm:agg}, we conclude that a.s.\ all sites stay active,
which finishes the proof of the theorem.

\subsection{Resampling}
\label{sub:resampling}

In this subsection we prove Theorem~\ref{thm:abactive}.

By Theorem~\ref{thm:monotone}, we can assume $\lambda =\infty $. In this
case, the sleep Poisson clocks $\boldsymbol{\mathcal{P}}$ play no role
in the previous construction: a particle is sleeping if and only if there
are no other particles at the same site.

We consider the following dynamics instead of the ARW.
\smallskip

\emph{The particle-hole model.} Particles perform continuous-time random
walks independently of each other. Sites not containing any particle are
called \emph{holes}. When a particle is alone at some site, it
\emph{settles} there forever, filling the corresponding hole. After the
hole has been filled, the site becomes available for other particles to
go through. If a site is occupied by several particles at $t=0^{-}$, we
choose one of them uniformly to fill the hole at $t=0$, and the other particles
remain free to move. This is well-defined as in Theorem~\ref{thm:welldefined},
with the same proof.

This model is very similar to the ARW with $\lambda =\infty $. In both
models, once a site has at least one particle, it will always retain one
particle. The differences are (i) sites with $n \geqslant 2$ particles
are toppled at rate $n-1$ instead of $n$ and (ii) each site retains forever
the first particle to arrive there, whereas in the ARW the particles can
take turns replacing each other. Nevertheless, both models have the same
site-wise representation (described in \S \ref{sub:lambda}) and same fixation
properties in the sense of Theorem~\ref{thm:equivalence}.
\smallskip

We will show that, under the assumption of site fixation,
\begin{equation*}
\mathbb{P}( {\boldsymbol{0}}\mbox{ is never visited} ) > 0.
\end{equation*}
This implies that $\zeta <1$ by Theorem~\ref{thm:conservation}, therefore
proving Theorem~\ref{thm:abactive}.

Now assuming site fixation, necessarily there exists
$k\in \mathbb{N}$ such that
\begin{equation*}
\mathbb{P}\big (\text{the number of particles which ever visit } {
\boldsymbol{0}}\text{ equals }k\big )>0 .
\end{equation*}
Hence, there exist $x_{1},\dots ,x_{k}\in \mathbb{Z}^{d}$ such that
$\mathbb{P}(\mathcal{A})>0$, where
\begin{equation*}
\mathcal{A}=
\text{``the particles which ever visit ${\boldsymbol{0}}$ are initially
at the sites $x_{1},\dots ,x_{k}$.'' }
\end{equation*}

Consider two systems $\omega $ and $\tilde{\omega }$, coupled as follows.
We take $\tilde{\mathbf{X}}=\mathbf{X}$, and
$\tilde{\eta }_{0}(x)=\eta _{0}(x)$ for
$x \not \in \{x_{1},\dots ,x_{k}\}$. For
$x \in \{x_{1},\dots ,x_{k}\}$, we sample $\tilde{\eta }_{0}$ and
$\eta _{0}$ independently. Now notice that
\begin{multline*}
\mathbb{P}( \mathcal{A}\mbox{ occurs for } \tilde{\omega },
\mbox{ and } \eta _{0}(x_{1})=\cdots =\eta _{0}(x_{k})=0 ) =
\\
= \mathbb{P}( \mathcal{A}\mbox{ occurs for } \tilde{\omega } )
\times \mathbb{P}( \big . \eta _{0}(x_{1})=\cdots =\eta _{0}(x_{k})=0
\mbox{ for } \omega ) >0 .
\end{multline*}
To conclude we claim that, on the above event, no particle ever visits
${\boldsymbol{0}}$ in the system $\omega $. Indeed, on the above event,
the initial configuration of $\omega $ is the same as that of
$\tilde{\omega }$ except for the deletion of the particles present in
$\{x_{1},\dots ,x_{k}\}$. In particular, all the particles which visit
the origin in $\tilde{\omega }$ are deleted in $\omega $. Recalling that
$\omega $ and $\tilde{\omega }$ share the same putative trajectories, by
following how the effect of deleting such particles propagates in the system
evolution, one can see that in the system $\omega $ no particle can possibly
visit ${\boldsymbol{0}}$, finishing the proof.

\subsection{Fixation equivalence}
\label{sub:agg}

In this subsection we prove Theorem~\ref{thm:agg}. We use the construction
provided by Theorem~\ref{thm:welldefined}. Three implications are immediate:
$(i) \Rightarrow (ii)$ by the $0$-$1$ law in Theorem~\ref{thm:equivalence},
$(ii) \Rightarrow (iii)$ because if some particle fixates then it has to
fixate at some site, $(iii) \Rightarrow (iv)$ is trivial, so we only have
to show $(iv) \Rightarrow (i)$.

Let $\mathcal{A}^{x,j}$ denote the event that particle $(x,j)$ stays active,
and write\break $\mathcal{A}^{x}=\mathcal{A}^{x,1}$. Assuming $(iv)$ holds,
$a:=\mathbb{P}(\mathcal{A}^{\boldsymbol{0}})>0$. Indeed, $(iv)$ implies
that $\mathbb{P}(\mathcal{A}^{{\boldsymbol{0}},j})>0$ for some
$j$ and, by interchangeability of particles, we have
$\mathbb{P}(\mathcal{A}^{{\boldsymbol{0}},j}) = \mathbb{P}(
\mathcal{A}^{{\boldsymbol{0}}}, |\eta _{0}({\boldsymbol{0}})|
\geqslant j) \leqslant \mathbb{P}( \mathcal{A}^{{\boldsymbol{0}}})$.

We make a side remark before giving more details. By the mass transport
principle, the number $N_{t}$ of particles which stay active and are present
at site ${\boldsymbol{0}}$ at time $t$ satisfies
$\mathbb{E}N_{t} \geqslant a$, hence
$\liminf _{t} \mathbb{E}N_{t} > 0$. But to show site activity we need
$\limsup _{t} \mathbb{P}(N_{t} \geqslant 1)>0$ instead. The idea is to
introduce extra randomness so as to spread out the effect of these particles.

\medskip
Since the system $\boldsymbol{\eta }$ is a measurable function of the randomness
$\omega $, for each $\varepsilon >0$ there is $k\in \mathbb{N}$ such that
the event $\mathcal{A}^{\boldsymbol{0}}$ can be $\varepsilon $-approximated
by some event $\mathcal{A}_{\varepsilon }^{\boldsymbol{0}}$ that depends
only on
$\left ( \eta _{0}(x),X^{x},\mathcal{P}^{x} \right )_{\|x\|
\leqslant k}$. Let $\mathcal{A}_{\varepsilon }^{x}$ denote the corresponding
translation of the event
$\mathcal{A}_{\varepsilon }^{\boldsymbol{0}}$. When
$\mathcal{A}_{\varepsilon }^{x}$ occurs, we say that particle $(x,1)$ is
a \emph{candidate}. It is a \emph{good candidate} if
$\mathcal{A}^{x}$ also occurs, otherwise it is a
\emph{bad candidate}.

Fix $t>0$. Let $n\in \mathbb{N}$ be a large number. The trick is to add
more randomness to the system by choosing $Z^{y}$ uniformly among the first
$n$ different sites in the putative trajectory $X^{y,1}$ after time
$t$, independently over $y$. Define $\mathcal{C}(y,x)$ as the event that
$\mathcal{A}_{\varepsilon }^{y}$ occurs and $Z^{y}=x$. Let
\begin{equation*}
q(y,x)=\mathbb{P}\left ( \mathcal{C}(y,x) \,\middle |\,\big . \omega
\right ) \qquad \text{and} \qquad Q(x)=\sum _{y} q(y,x).
\end{equation*}
By the mass transport principle,
\begin{equation*}
\mathbb{E}[Q({\boldsymbol{0}})] = \sum _{y} \mathbb{P}(
\mathcal{C}(y,{\boldsymbol{0}}) ) = \sum _{y} \mathbb{P}(
\mathcal{C}({\boldsymbol{0}},y) ) = \mathbb{P}(\mathcal{A}_{\varepsilon }^{{\boldsymbol{0}}}) =: b > a - \varepsilon .
\end{equation*}
Notice that $q(y,x) \leqslant \frac{1}{n}$. Notice also that
$q(y,x)$ and $q(z,x)$ are independent if $\|y-z\| > 2k$. Using these two
facts, it follows that $ \mathbb{V}[ Q({\boldsymbol{0}}) ] $ becomes
small when $n$ is large, and thus $Q({\boldsymbol{0}})$ converges to
$b$ in distribution.

Let $N(x)=\sum _{y} \mathds{1}_{\mathcal{C}(y,x)}$ count the number of
candidates for which $Z^{y}=x$. Then
\begin{equation*}
\mathbb{P}\left ( N({\boldsymbol{0}})=0 \,\big .\middle |\, \omega
\right ) = \prod _{y} (1-q(y,{\boldsymbol{0}})) \leqslant e^{-Q({
\boldsymbol{0}})} \to e^{-b}
\end{equation*}
in probability as $n \to \infty $. Also, let
$\tilde{N}(x) = \sum _{y} \mathds{1}_{\mathcal{C}(y,x) \setminus
\mathcal{A}^{y}}$ count the number of bad candidates for which
$Z^{y}=x$. Then, using the mass transport principle,
\begin{equation*}
\mathbb{E}[\tilde{N}({\boldsymbol{0}})] = \sum _{y} \mathbb{P}({
\mathcal{C}(y,{\boldsymbol{0}}) \setminus \mathcal{A}^{y}}) = \sum _{y}
\mathbb{P}({\mathcal{C}({\boldsymbol{0}},y) \setminus \mathcal{A}^{{
\boldsymbol{0}}}}) = \mathbb{P}( \mathcal{A}_{\varepsilon }^{{
\boldsymbol{0}}} \setminus \mathcal{A}^{{\boldsymbol{0}}} )
\leqslant \varepsilon .
\end{equation*}
Let $\mathcal{D}^{x}$ denote the event that there exists a good candidate
$(y,1)$ such that $Z^{y}=x$. Using the two last estimates we get
\begin{equation*}
\mathbb{P}(\mathcal{D}^{{\boldsymbol{0}}}) \geqslant \mathbb{P}(N({
\boldsymbol{0}}) \geqslant 1) - \mathbb{P}(\tilde{N}({
\boldsymbol{0}}) \geqslant 1 ) \geqslant 1 - e^{-a+\varepsilon } -
\delta _{n} - \varepsilon ,
\end{equation*}
where $\delta _{n} \to 0$ as $n\to \infty $. Choosing $\varepsilon $ small
and $n$ large, we have
$\mathbb{P}(\mathcal{D}^{\boldsymbol{0}})>\frac{a}{2}$.

To conclude, notice that, on the event
$\mathcal{D}^{\boldsymbol{0}}$, there is a particle $(y,0)$ which stays
active, and some inner time $s>t$ such that
$X^{y,1}_{s} = {\boldsymbol{0}}$, implying that site
${\boldsymbol{0}}$ is visited by an active particle after time $t$. Letting
$t\to \infty $, we get
$\mathbb{P}(\text{site } {\boldsymbol{0}}\text{ stays active})
\geqslant \frac{a}{2} > 0$, concluding the proof that
$(iv) \Rightarrow (i)$.

\section{Analysis of explicit constructions}
\label{sec:constructions}

An evolution $(\eta _{t})_{t\geqslant 0}$ starting with only finitely many
particles can be constructed explicitly in innumerous ways, using Poisson
processes, exponential variables, random walks, tossing some coins, etc.
Then we want to say that a system starting from an infinite random configuration
$\eta _{0}$ exists and can be approximated in distribution by finite ones.
Namely, denoting by $\mathbf{P}^{\nu }_{V}$ the law of the process starting
from the finite truncation $\eta _{0} \cdot \mathds{1}_{V}$ for finite
$V \subseteq \mathbb{Z}^{d}$, we wonder whether
%
\begin{equation}
\label{eq:approx}
\mathbf{P}^{\nu }\Big ((\eta _{t})_{t\geqslant 0} \in \mathcal{A}\Big ) =
\lim _{V\uparrow \mathbb{Z}^{d}} \mathbf{P}^{\nu }_{V} \Big ((\eta _{t})_{t
\geqslant 0} \in \mathcal{A}\Big )
\end{equation}
for every \emph{local event} $\mathcal{A}$, \textit{i.e.}\ every event
$\mathcal{A}$ whose occurrence is determined by
$(\eta _{s}(x))_{x\in B_{k}, s \in [0,t]}$ for some finite $t$ and
$k$.

Assume the limit on the right-hand side exists for some construction which
is consistent with the rates specified in \S \ref{sub:evolution}. Then
the limit is obviously the same for any other construction consistent with
\S \ref{sub:evolution}. So if there exists a process
$(\eta _{t})_{t\geqslant 0}$ on a space $\mathbf{P}^{\nu }$ whose distribution
satisfies~\eqref{eq:approx}, then its distribution is unique.

There are at least three ways to show existence of a
$\mathbf{P}^{\nu }$ satisfying~\eqref{eq:approx}. One is to consider a certain
norm on a subset of $(\mathbb{N}_{\mathfrak{s}})^{\mathbb{Z}^{d}}$ and
use abstract theory of generators and semigroups adapted to non-compact
spaces. Such a norm has to be more restrictive than product topology (indeed,
one can always make many particles visit ${\boldsymbol{0}}$ in short time
by placing enough particles far away), but it still gives~\eqref{eq:approx}
with a good level of generality on~$\nu $. Another way is to consider the
particle-wise construction described in \S \ref{sub:pwconstr}, which is
well-defined as we prove in \S \ref{sub:particleexistence}. The third way
is to add Poisson clocks to the site-wise representation of \S
\ref{sub:conditions}, which is done in \S \ref{sub:swdefined}. These constructions
work under the assumption that
$\int |\eta (x)|\nu ({\mathrm{d}}\eta ) \leqslant \zeta _{\mathrm{max}}$
for some finite $\zeta _{\mathrm{max}}$ uniformly over $x$.

\subsection{Conditions for fixation and activity}
\label{sub:condbu}

In this subsection we prove Theorem~\ref{thm:equivalence} assuming that
an explicit site-wise construction of the process satisfies~\eqref{eq:fixationstabilizable},
and that~\eqref{eq:noblowups} holds.

We start with the $0$-$1$ law. For almost every $\mathcal{I}$, if
$m_{\eta }({\boldsymbol{0}})=\infty $ for a given configuration
$\eta $, then $m_{\eta }(y)=\infty $ for all $y$ with $p(y)>0$. Write
$W_{p}=\{z:p(z)>0\}\subseteq \mathbb{Z}^{d}$. Let us omit the tedious
proof of the following fact: if the elements of a set $W$ generate the
group $(\mathbb{Z}^{d},+)$, then as a semigroup they generate a set that
contains some $w+U\cap \mathbb{Z}^{d}$, where
$U\subseteq \mathbb{R}^{d}$ is a cone with non-empty interior. By the
previous remark, if $z\in -(w+U)$ and $m_{\eta _{0}}(z)=\infty $, then
$m_{\eta _{0}}({\boldsymbol{0}})=\infty $. Assume that
$\mathbb{P}^{\nu }\big (m_{\eta _{0}}({\boldsymbol{0}})=\infty \big )>0$.
As $-w-U$ contains balls of arbitrarily large radius, since
$(\eta _{0},\mathcal{I})$ is translation-ergodic, the
$\mathbb{P}^{\nu }$-probability of finding a site $z\in -(w+U)$ with
$m_{\eta _{0}}(z)=\infty $ is equal to $1$, and therefore
$\mathbb{P}^{\nu }\big (m_{\eta _{0}}({\boldsymbol{0}})=\infty \big )=1$.

\medskip
We now prove that
$\mathbf{P}^{\nu }\big (\text{fixation of }(\eta _{t})_{t\geqslant 0}
\big )=\mathbb{P}^{\nu }(m_{\eta _{0}}({\boldsymbol{0}})<\infty )$. Let
$h_{t}(x)$ denote the number of topplings at site $x$ during the time interval
$[0,t]$, meaning any action performed at $x$, including unsuccessful attempts
to sleep. Write $h_{\infty }(x)=\lim _{t\to \infty }h_{t}(x)$. This limit
exists as $h_{t}(x)$ is non-decreasing in $t$.

The core of the proof is to add some Poisson clocks to
$\mathbb{P}^{\nu }$ and use it to construct $\mathbf{P}^{\nu }$ explicitly,
so that
%
\begin{equation}
\label{eq:fixationstabilizable}
\mathbf{P}^{\nu }\big (h_{\infty }(x) \geqslant k\big ) = \mathbb{P}^{\nu
}\big (m_{\eta _{0}}(x)\geqslant k\big ) \qquad \mbox{for each} \quad k>0
\end{equation}
and use it to show that
%
\begin{equation}
\label{eq:noblowups}
\mathbf{P}^{\nu }\big (h_{t}(x) \geqslant k \big )\to 0 \quad
\mbox{ as }\quad k\to \infty \qquad \mbox{for each fixed }t,
\end{equation}
which is done in the next subsection.

Let us show that these imply the theorem. Assume
$\mathbb{P}^{\nu }\big (m_{\eta _{0}}(x)<\infty \big )=1$. It follows from~\eqref{eq:fixationstabilizable}
that
$\mathbf{P}^{\nu }\big (h_{t}(x)\mbox{ eventually constant}\big )=1$, thus
$x$ is eventually stable in $\eta _{t}$ and in particular
$\eta _{t}(x)$ remains bounded for large $t$. But $\eta _{t}(x)$ can only
decrease when $x$ is unstable, so
$\mathbf{P}^{\nu }\big (\eta _{t}(x)\mbox{ converges}\big )=1$.

Otherwise, $\mathbb{P}^{\nu }\big ( m_{\eta _{0}}(x)=\infty \big )=1$ by the
$0$-$1$ law, then~\eqref{eq:fixationstabilizable} gives
$\mathbf{P}^{\nu }\big (h_{t}(x)\to \infty \mbox{ as }t\to \infty \big )=1$.
Now by~\eqref{eq:noblowups} we know that
$\big (h_{t}(x)\big )_{t\geqslant 0}$ cannot blow up in finite time, whence
for each $x$, the value of $\eta _{t}(x)$ changes for arbitrarily large
times, and the system stays active.

\subsection{The site-wise construction}
\label{sub:swdefined}

In this subsection we provide a coupling~$\mathbb{P}^{\nu }$ that produces
$\mathbf{P}^{\nu }$ and all $\mathbf{P}^{\nu }_{V}$ on the same probability
space, and show that~\eqref{eq:approx} holds. We also show that this coupling
satisfies~\eqref{eq:fixationstabilizable} and~\eqref{eq:noblowups}. The
translation-invariant distribution $\nu $ is fixed and will be omitted
in the notation.

Start by adding Poisson clocks to the site-wise representation. More precisely,
sample $\mathcal{I}$ following the distribution described in \S
\ref{sub:conditions}, sample $\eta _{0}$ according to the distribution
$\nu $, and sample an i.i.d.\ collection of Poisson point processes with
intensity $(1+\lambda ){\mathrm{d}}t$, all independently. Let
$\mathbb{P}$ denote the underlying probability.

\medskip
For a finite deterministic initial configuration $\xi $, the evolution
is constructed as follows. At $t=0$, let $L_{0}(x)=0$ for all $x$. Fix
$\eta _{t}'(x)=\xi (x)$ for all small $t$, and let $L_{t}(x)$ increase
by
$\tfrac{{\mathrm{d}}}{{\mathrm{d}}t}L_{t}(x) = \big (\eta _{t}(x)\big )
\mathds{1}_{\eta _{t}(t) \ne \mathfrak{s}}$. Denote the Poisson point
process at each site $x$ by $(T_{n}(x))_{n}$ where $T_{0}=0$ and
$T_{n+1}-T_{n}$ are i.i.d.\ exponentials with parameter $1+\lambda $. Writing
$h_{t}'(x)=\max \{n\in \mathbb{N}_{0}:L_{t}(x)\geqslant T_{n}(x)\}$ for
all $t$, let $\eta _{t}'(x)$ remain constant until the moment
$t_{1}$ of the first jump of $h_{t}'$, which happens a.s.\ at a unique
site $y_{1}$ that must be unstable for $\xi $. At this point, take
$\alpha _{1} = (y_{1})$ and $\eta _{t_{1}}'=\Phi _{y_{1}}\xi $. Notice
that $h_{t_{1}}'=m_{\alpha _{1}}$. Continue evolving $L_{t}$ with the same
rule, keeping $\eta _{t}'=\eta _{t_{1}}'$, until the moment $t_{2}$ of
the next jump of $h_{t}'$, which happens a.s.\ at a unique site
$y_{2}$, that again must be unstable for $\eta _{t_{1}}'$. As before, take
$\alpha _{2} = (y_{1},y_{2})$ and
$\eta _{t_{2}}'=\Phi _{y_{2}}\eta _{t_{1}}'=\Phi _{\alpha _{2}}\xi $. Again
$h_{t_{2}}'=m_{\alpha _{2}}$. Carry this procedure until
$\eta _{t}'(x)=0 \text{ or } \mathfrak{s}$ for all $x$. After this time,
$L_{t}$ will be constant and the configuration will no longer change.

In this construction, at each time $t \geqslant 0$, $\eta _{t}'$ is given
by $\Phi _{\alpha _{j}}$ for some $j$, $\alpha _{j}$ is a legal sequence
of topplings for $\xi $, and $m_{\alpha _{j}}=h_{t}$. Hence, by the local
Abelian property, $\eta _{t}'$ can be read from $\xi $,
$\mathcal{I}$ and $h_{t}'$. Moreover, the occupation times
$L_{t}(x)$, and thus the toppling counter $h_{t}'(x)$, are increasing in
the initial configuration $\xi $ (proof below). We use
$\eta _{t}^{V}$ and $h_{t}^{V}$ to denote the processes obtained by taking
$\xi = \eta _{0}^{V} = \eta _{0} \cdot \mathds{1}_{V}$. Then for each
fixed $x$ and $t$ the counter $h_{t}^{V}(x)$ will be increasing in
$V$, so it has a limit $h_{t}(x)$ that does not depend on the particular
increasing sequence $V\uparrow \mathbb{Z}^{d}$.

Below we will show that, denoting by $N_{t}^{V}(x)$ the number of times
that a particle jumps from some $z \ne x$ into $x$ in the process
$(\eta _{s}^{V})_{s \in [0,t]}$, we have
$\mathbb{E}[N_{t}^{V} (x)] \leqslant \zeta _{\mathrm{max}} \times t$.

Hence, the set of sites $z$ such that $\mathfrak{t}^{z,k}=x$ for some
$k \leqslant h_{t}(z)$ is finite and $h_{t}(x)$ is also finite. Now for
each $z$ in this set, $h_{t}^{V}(z)$ eventually equals $h_{t}(z)$ as
$V \uparrow \mathbb{Z}^{d}$, and the same holds for $h_{t}^{V}(x)$. It
thus follows from the local Abelian property that $\eta _{t}^{V}(x)$ will
also be eventually constant, and we take $\eta _{t}(x)$ as the limit. Hence,
almost surely, convergence holds simultaneously for all
$x\in \mathbb{Z}^{d}$ and $t\in \mathbb{Q}_{+}$, and taking
$\eta _{s}(x) = \lim _{t \downarrow s,t\in \mathbb{Q}} \eta _{t}(x)$ gives~\eqref{eq:approx}.

\medskip
We now move on to the proof of~\eqref{eq:fixationstabilizable}. First,
for every $k$,
\begin{equation}
\nonumber
\mathbb{P}\big (h^{V}_{t}(x)\geqslant k\big ) \mathop{\longrightarrow }_{V
\uparrow \mathbb{Z}^{d}} \mathbb{P}\big (h_{t}(x)\geqslant k\big )
\mathop{\longrightarrow }_{t\to \infty } \mathbb{P}\big (h_{\infty }(x)
\geqslant k\big ).
\end{equation}
We now show, on the other hand, that
\begin{equation}
\nonumber
\mathbb{P}\big (h_{t}^{V}(x)\geqslant k\big ) \mathop{\longrightarrow }_{t
\to \infty } \mathbb{P}\big (m_{\eta _{0}^{V}}(x)\geqslant k\big )
\mathop{\longrightarrow }_{V \uparrow \mathbb{Z}^{d}} \mathbb{P}
\big (m_{\eta _{0}}(x)\geqslant k\big ) .
\end{equation}
The second limit follows from
$ m_{\eta _{0}} \geqslant m_{\eta _{0}^{V}} \geqslant m_{\eta _{0},V}
\to m_{\eta _{0}} $. Let us prove the first limit. Since finite configurations
are a.s.\ stabilized after a finite number of topplings, there is some
$t_{*}\geqslant 0$ such that $h_{t}^{V}=h_{t_{*}}^{V}$ for all
$t\geqslant t_{*}$, and moreover $\eta _{t_{*}}^{V}$ is stable. Since
$h_{t}^{V}$ counts the number of topplings performed at each site up to
time $t$, all of which are legal for $\eta _{0}^{V}$, we have
$h_{t_{*}}^{V} = m_{\eta _{0}^{V}}$ by the Abelian property. By monotonicity
the above limits commute, proving~\eqref{eq:fixationstabilizable}.

\medskip
Now the tedious proof that $L_{t}(x)$ is non-decreasing in $\xi $. Let
$\xi _{1} \leqslant \xi _{2}$ be finite configurations. In order to show
that these yield $L^{1}_{t}(x)$ and $L^{2}_{t}(x)$ satisfying
$L^{1} \leqslant L^{2}$, we will show that the set
$ G = \big \{ t \geqslant 0 : L_{s}^{1} \leqslant L_{s}^{2}
\mbox{ for all } s\leqslant t \big \} \ni 0 $ is both open and closed on
$[0,\infty )$. Since $t\mapsto L_{t}$ is continuous, $G$ is closed. Let
$t\in G$ and fix $x\in \mathbb{Z}^{d}$. If
$L_{t}^{1}(x)<L_{t}^{2}(x)$, by continuity there is some
$\varepsilon _{x}>0$ such that $L_{s}^{1}(x)<L_{s}^{2}(x)$ for all
$s\leqslant t+\varepsilon _{x}$. Otherwise
$L_{t}^{1}(x)=L_{t}^{2}(x)$, which means $h_{t}^{1}(x)=h_{t}^{2}(x)$. At
the same time, $h_{t}^{1}(z) \leqslant h_{t}^{2}(z)$ for $z \neq x$. Hence,
by Property~\ref{property2} in \S \ref{sub:diaconis},
$\eta _{t}^{1}(x) \leqslant \eta _{t}^{2}(x)$. Since
$t \mapsto \eta _{t}^{1}(x)$ and $t \mapsto \eta _{t}^{2}(x)$ are piecewise
constant and right-continuous, there is some $\varepsilon _{x}>0$ such
that, for all $s \in (t, t+\varepsilon _{x}]$,
$\eta _{s}^{1}(x) \leqslant \eta _{s}^{2}(x)$, hence
$\frac{{\mathrm{d}}}{{\mathrm{d}}s}L^{1}_{s}(x) \leqslant
\frac{{\mathrm{d}}}{{\mathrm{d}}s}L^{2}_{s}(x)$, and thus
$L_{s}^{1}(x) \leqslant L_{s}^{2}(x)$. Finally, since there are a.s.\ finitely
many sites $x$ such that $L_{s}^{1}(x)>0$ for some $s$, we can take the
$\varepsilon >0$ as the smallest $\varepsilon _{x}$ over all such
$x$, so that $t+\varepsilon \in G$.

\medskip
The last missing statement is that
$\mathbb{E}[N_{t}^{V} (x)] \leqslant \zeta _{\mathrm{max}} \times t$. This
was used in the proof of~\eqref{eq:approx} and it also implies~\eqref{eq:noblowups}.
Instead of making a rather convoluted argument, we resort to the finite
particle-wise construction of \S \ref{sub:pwconstr} in terms of walks
$X$ and $Y$, which produces a process
$(\eta _{t}^{V})_{t \geqslant 0}$ with the same distribution. Using translation
invariance and re-indexing sums,
\begin{align*}
\mathbb{E}\big [N_{t}^{V} (x)\big ] &= \mathbb{E}\bigg [ \sum _{y}
\sum _{j=1}^{\eta _{0}(y)} \mbox{number of jumps of } Y^{y,j}
\mbox{ into } x \mbox{ during } [0,t] \bigg ]
\\
&\leqslant \mathbb{E}\bigg [ \sum _{y} \sum _{j=1}^{\eta _{0}(y)}
\mbox{number of jumps of } X^{y,j} \mbox{ into } x \mbox{ during } [0,t]
\bigg ]
\\
&= \textstyle \sum _{y}\mathbb{E}|\eta _{0}(y)| \times \mathbb{E}
\Big [ \mbox{jumps of } X^{y,1} \mbox{ into } x \mbox{ during } [0,t]
\Big ]
\\
&\leqslant \textstyle \zeta _{\max } \, \sum _{y} \mathbb{E}\Big [
\mbox{number of jumps of } X^{y,1} \mbox{ into } x \mbox{ during } [0,t]
\Big ]
\\
&= \textstyle \zeta _{\max } \, \mathbb{E}\Big [ \sum _{y}
\mbox{number of jumps of } X^{x,1} \mbox{ into } y \mbox{ during } [0,t]
\Big ]
\\
&= \zeta _{\max } \times t,
\end{align*}
which concludes the proof.

\subsection{Well-definedness of the particle-wise construction}
\label{sub:particleexistence}

In this subsection we prove Theorem~\ref{thm:welldefined}. The triple
$(\eta _{0},\mathbf{Y},\boldsymbol{\gamma })$ describes the system by specifying
the location and state of each labeled particle. An equivalent description
is to specify the labels and states of all particles present at each site.

Consider a realization of the walks $\mathbf{X}$ and clocks
$\boldsymbol{\mathcal{P}}$. For a finite initial configuration
$\xi $, we consider the evolution obtained from
$(\xi ,\mathbf{X},\boldsymbol{\mathcal{P}})$, and define
\begin{equation*}
\bar{\eta }_{t}(z;\xi ) = \Big \{ (x,j,i)\in \mathbb{Z}^{d}\times
\mathbb{N}\times \{1,\mathfrak{s}\}\,:\, j=1,\dots ,|\xi (x)|, Y^{x,j}_{t}=z,
\gamma ^{x,j}_{t}=i \Big \} .
\end{equation*}
Dependence of $Y$'s and $\gamma $'s on $\xi $ is omitted in the notation.
Our goal is to show that, almost surely, for each
$z,y\in \mathbb{Z}^{d}$ and $t>0$,
$(\bar{\eta }_{s}(z;\eta _{0}^{B_{n}^{y}}))_{s \in {[0,t]}}$ is the same
for all but finitely many $n$, and that the limiting process
$(\bar{\eta }_{t})_{t \geqslant 0}$ does not depend on $y$. Since each walk
$X^{x,j}$ only visits finitely many sites by time $t$, this implies well-definedness
in the sense of Definition~\ref{def:pwwd}.

\medskip
A convenient observation is the following. It suffices to prove that, for
every sequence of finite sets $W_{n} \uparrow \mathbb{Z}^{d}$, a.s.\ the
process $(\bar{\eta }_{s}(z;\eta _{0}^{W_{n}}))_{s \in {[0,t]}}$ is the
same for all but finitely many $n$. Indeed, assuming this holds true, a.s.\ it
will be the case simultaneously for all the sequences
$(B^{y}_{n})_{n}$ as well as some deterministic increasing sequence
$(W_{n})_{n}$ that has infinitely many terms in common with each one of
them. This in turn implies that the limit
$(\bar{\eta }_{s})_{s \in [0,t]}$ is a.s.\ the same for each of them, for
every $t>0$.

\medskip
To prove this a.s.\ eventually constant limit, we study how the ``influence''
of a particle addition propagates through the system by time $t$. We show
that the set of sites $z$ for which the configuration
$\bar{\eta }_{t}(z)$ is affected for some $s \in [0,t]$ by a particle addition
at $x$ is stochastically dominated by a branching random walk started with
a single particle at $x$.

We make this precise now. If $\xi (x)>0$, define the event
$ x \overset{\xi ,t}{\leadsto }z $ that
\begin{equation*}
\bar{\eta }_{s}(z;\xi -\delta _{x},\mathbf{X},
\boldsymbol{\mathcal{P}}) \ne \bar{\eta }_{s}(z;\xi ,\mathbf{X},
\boldsymbol{\mathcal{P}}) \text{ for some } s\in [0,t].
\end{equation*}
Define the random set
\begin{equation*}
Z^{x}_{t}(\xi )= \big \{z\in \mathbb{Z}^{d} \,:\, x
\overset{\xi ,t}{\leadsto }z \big \},
\end{equation*}
which is the set of sites influenced during $[0,t]$ by the removal of the
last particle at $x$ from configuration $\xi $. We will later on prove
that
%
\begin{equation}
\label{eq:brwdom}
\mathbb{P}\big ( z \in Z^{x}_{t}(\xi ) \big ) \leqslant \mathbb{P}
\big ( U_{t}(z-x) \geqslant 1 \big )
\end{equation}
for all $z,x,t,\xi $, where $(U_{t})_{t \geqslant 0}$ denotes the following
branching process.

At $t=0$, let $U_{0} = \delta _{\boldsymbol{0}}$ be the configuration
with a single particle at ${\boldsymbol{0}}$. For each $t>0$, a transition
$U_{t} \to U_{t} + \delta _{x}$ occurs at rate $\lambda U_{t}(x)$, and
a transition $U_{t} \to U_{t} + 2 \delta _{y}$ occurs at rate
$\sum _{x} U_{t}(x) p(y-x)$. In words, each particle produces at rate
$\lambda $ a new copy at the same site, and at rate $1$ two new copies
at a site chosen at random. Particles never disappear in
$(U_{t})_{t}$.

\medskip
Before showing~\eqref{eq:brwdom}, let us derive Theorem~\ref{thm:welldefined}
by proving the a.s.\ eventual constant limit as
$W_{n} \uparrow \mathbb{Z}^{d}$. We can assume that $|W_{n}|=n$ and write
$W_{n}=\{x_{1},\dots ,x_{n}\}$. We want to rule out that, for infinitely
many $n$,
%
\begin{equation}
\label{eq:differ}
\bar{\eta }_{s}(z ; \eta _{0}^{W_{n}}) \ne \bar{\eta }_{s}(z ; \eta _{0}^{W_{n-1}})
\text{ for some } s \in [0,t].
\end{equation}
For each $n$, occurrence of the above event implies the occurrence of
\begin{equation}
\nonumber
z \in Z^{x_{n}}_{t}(\eta _{0}^{W_{n-1}} + k \delta _{x_{n}})
\text{ for some } k = 1,\dots , \eta _{0}(x_{n})
\end{equation}
in case $\eta _{0}(x_{n})\in \mathbb{N}_{0}$, or
$z \in Z^{x_{n}}_{t}(\eta _{0}^{W_{n-1}} + \delta _{x_{n}} \cdot
\mathfrak{s})$ if $\eta _{0}(x_{n})=\mathfrak{s}$. For simplicity we
ignore the $\mathfrak{s}$ case. The implication is true because
$\eta _{0}^{W_{n-1}}$ can be obtained from $\eta _{0}^{W_{n}}$ by removing
all the $|\eta _{0}(x_{n})|$ particles at $x$, one by one, and in case
none of them affects the configuration at site $z$ by time $t$, the event
in~\eqref{eq:differ} cannot occur. Now the bound~\eqref{eq:brwdom} holds
for each fixed $\xi $, whence
\begin{equation*}
\mathbb{P}\Big ( z \in Z^{x_{n}}_{t}(\eta _{0}^{W_{n-1}} + k \delta _{x_{n}})
\,\Big |\, |\eta _{0}(x_{n})| \geqslant k \Big ) \leqslant \mathbb{P}
\big ( U_{t}(z-x_{n}) \geqslant 1 \big ) .
\end{equation*}
Thus,
\begin{align*}
\mathbb{E}\Big [\#\big \{n\in \mathbb{N}\,:\, & z \in Z^{x_{n}}_{t}(
\eta _{0}^{W_{n-1}} + k \delta _{x_{n}}) \text{ for some } k
\leqslant |\eta _{0}(x_{n})| \big \} \Big ]
\\
& \leqslant \sum _{n}\sum _{k} \mathbb{P}\big ( |\eta _{0}(x_{n})|
\geqslant k , z \in Z^{x_{n}}_{t}(\eta _{0}^{W_{n-1}} + k \delta _{x_{n}})
\big )
\\
& \leqslant \zeta _{\mathrm{max}} \, \sum _{n} \mathbb{E}\big [ U_{t}(z-x_{n})
\big ]
\\
& = \textstyle \zeta _{\mathrm{max}} \, \mathbb{E}\big [ \sum _{y} U_{t}(y)
\big ] = \zeta _{\mathrm{max}} \times e^{(2+\lambda )t} < \infty .
\end{align*}
Hence, by Borel-Cantelli, a.s.\ the event in~\eqref{eq:differ} occurs for
finitely many $n$, concluding the proof of Theorem~\ref{thm:welldefined}.

\medskip
In the rest of this subsection we prove~\eqref{eq:brwdom}. We aim at bounding
the set of particles influenced before time $t$ by the presence or absence
of an extra particle.

Recall that particles behave independently except when one of them gets
reactivated by another, including the case where it is prevented from deactivating.
Hence, the influence only propagates in these cases, and only if the reactivation
may not be attributed to the presence of a third, uninfluenced particle.
Thus, a particle is influenced by the presence of an extra particle during
$[0,t]$ if either it is the extra particle itself, or if it was reactivated
at some time $s \leqslant t$ by sharing the same site with particles that
had been themselves influenced before time $s$.

Once a particle is influenced, its inner time becomes uncertain. In order
to keep track of every possibility and bound the presence of influenced
particles, the idea is to consider all potential paths simultaneously.
In some sense, we let each such particle both stay put at its current site
as well as jump to the next one, by means of replicating such particles.

Suppose $\xi ^{+}$ and $\xi ^{-}$ are finite and differ by the presence
of a single particle at a given site. By translation invariance we can
assume that this particle is $({\boldsymbol{0}},j^{*})$, so
$\xi ^{-}({\boldsymbol{0}})=j^{*}-1$ and
$\xi ^{+}({\boldsymbol{0}})=j^{*}$ (for simplicity, we omit the case where
$j^{*}=1$ and $\xi ^{+}({\boldsymbol{0}})=\mathfrak{s}$). We write
$Y_{t}^{x,j,\pm }$, $\gamma _{t}^{x,j,\pm }$ and $\bar{\eta }_{t}^{\pm }$ to
denote the processes obtained from
$(\xi ^{\pm },\mathbf{X},\boldsymbol{\mathcal{P}})$.

Initially, the set of potentially affected particles is
$R^{\times }_{0}=\{({\boldsymbol{0}},j^{*})\}$ and the set of unaffected
particles is $R^{\circ }_{0}=\{(x,j):j=1,\dots ,|\xi ^{-}(x)|\}$. Let
$\mathcal{T}^{x,j}$ denote the time when particle $(x,j)$ is removed from
$R^{\circ }_{t}$ and added to $R^{\times }_{t}$. From this time on, clock rings
from $\mathcal{P}^{x,j}$ will be ignored, this particle will jump normally
according to $X^{x,j}$ and also leave a copy behind each time it jumps,
so we can safely ignore its state $\gamma ^{x,j}$. More precisely, we define
$\sigma ^{x,j,\times }(s) = s - \mathcal{T}^{x,j} + \sigma ^{x,j,\pm }(
\mathcal{T}^{x,j})$ and the increasing sets
\begin{equation*}
D^{x,j}_{t} = \Big \{ z \in \mathbb{Z}^{d} : X^{x,j}_{\sigma ^{x,j,
\times }(s)} = z \text{ for some } s \in [\mathcal{T}^{x,j},t] \Big \} .
\end{equation*}
Finally, let $D_{t}$ denote the union of $D^{x,j}_{t}$ over all
$(x,j) \in R_{t}^{\times }$.

On the other hand, each unaffected particle $(x,j)\in R^{\circ }_{t}$ evolves
normally and interacts with other particles
$(x',j') \in R^{\circ }_{t}$ normally, until the first time
$\mathcal{T}^{x,j}$ when it becomes affected. This will occur when (i)
particle $(x,j)$ is sleeping at some site $z$ and site $z$ is added to
$D_{t}$, or (ii) the clock rings for particle $(x,j)$ to sleep at some
site $z \in D_{t}$ and there are no other particles
$(x',j') \in R^{\circ }_{t}$ with $Y^{x',j'}_{t} = z$. So case (ii) is triggered
by a sleep clock ring of particle $(x,j)$ itself, whereas case (i) is triggered
by the jump of an affected particle.

We then define
\begin{equation*}
\bar{\eta }^{\times }_{t}(z) = \big \{ (x,j,i) : (x,j) \in R^{\times }_{t}, i
\in \{1,\sigma \}, z \in D_{t}^{x,j} \big \}
\end{equation*}
as well as
\begin{equation*}
\bar{\eta }^{\circ }_{t}(z) = \{(x,j,i) \in \eta ^{+}_{t}(z) : (x,j)
\in R_{t}^{\circ }\} = \{(x,j,i) \in \eta ^{-}_{t}(z) : (x,j) \in R_{t}^{\circ }\} ,
\end{equation*}
and note that
\begin{equation}
\bar{\eta }^{\circ }_{t}(z) \subseteq \bar{\eta }^{\pm }_{t}(z) \subseteq
\bar{\eta }^{\circ }_{t}(z) \cup \bar{\eta }^{\times }_{t}(z) .
\nonumber
\end{equation}
The above inclusions imply that
$Z^{\boldsymbol{0}}_{t}(\xi ) \subseteq D_{t}$.

Hence, defining
\begin{equation*}
\tilde{U}_{t}(z) = \# \big \{ (x,j) \in R_{t} \,:\, z \in D_{t}^{x,j}
\big \} ,
\end{equation*}
to get~\eqref{eq:brwdom} it suffices to show that
%
\begin{equation}
\label{eq:domlaw}
\big (\tilde{U}_{t}\big )_{t \geqslant 0} \leqslant \big ({U}_{t}\big )_{t
\geqslant 0} \text{ in law}.
\end{equation}

To prove~\eqref{eq:domlaw}, it is enough to show that the transition rates
of $(\tilde{U}_{t})_{t}$ are always dominated by those of
$({U}_{t})_{t}$. The process $(\tilde{U}_{t})_{t}$ can increase in two
situations. First, when only one particle $(x,j) \in R^{\circ }_{t}$ is present
at some site $z \in D_{t}$ and its sleep clock rings. This transition causes
$\tilde{U}_{t}(z)$ to increase by $1$. The rate of this transition is
$\lambda $ at such sites and $0$ elsewhere, so it is bounded by
$\lambda \tilde{U}_{t}(z)$. Second, when a particle
$(x,j) \in R^{\times }_{t}$ jumps, in which case it may add both a new site
to $D^{x,j}_{t}$ and add a new particle $(x',j')$ to $R^{\times }_{t}$. This
transition causes $\tilde{U}_{t}(z)$ to increase by at most $2$, and it
occurs at rate bounded by $\sum _{x} \tilde{U}_{t}(x) p(z-x)$. This concludes
the proof of~\eqref{eq:domlaw}, thus of~\eqref{eq:brwdom} and hence of
Theorem~\ref{thm:welldefined}.

\subsection{Monotonicity and the case of infinite sleep rate}
\label{sub:lambda}

In this subsection we discuss how to adapt the site-wise construction of
\S \ref{sub:swdefined} to produce processes with different values of
$\lambda $. We use this to prove Theorem~\ref{thm:monotone}. We also describe
the site-wise representation for the case $\lambda =\infty $ as mentioned
in \S \ref{sub:resampling}.

Recall the construction of \S \ref{sub:swdefined}. At each site $x$, we
have a sequence $(\mathfrak{t}^{x,j})_{j\in \mathbb{N}}$ and a Poisson
point process which we now denote
$\mathcal{P}_{x} \subseteq [0,\infty )$. The pair
$(\mathcal{P}_{x} , \mathfrak{t}^{x,\cdot })$ can be seen as a marked
Poisson point process where $\mathfrak{t}^{x,j}$ is the mark of the
$j$-th point in $\mathcal{P}_{x}$. Alternatively, we can obtain this marked
process by merging two Poisson processes, $\mathcal{P}^{*}_{x}$ for jump
and $\mathcal{P}^{\lambda }_{x}$ for sleep. In this construction, different
values of $\lambda $ can be coupled in a standard way, so that
$\mathcal{P}^{\lambda }_{x} \subseteq \mathcal{P}^{\lambda '}_{x}$ for
$\lambda ' \geqslant \lambda $.\looseness=1

Similarly to the proof that $L_{t}(x)$ is non-decreasing in $\xi $, one
can show that $L_{t}(x)$ is also non-increasing in $\lambda $ (details
omitted). Since fixation is a.s.\ equivalent to
$L_{t}({\boldsymbol{0}})$ remaining bounded as $t \to \infty $,
$\zeta _{c}$ is non-decreasing in $\lambda $.
\smallskip

We conclude by discussing the case of $\lambda =\infty $.
\smallskip

As we increase $\lambda $, the sleep clock rings more and more often. In
the limiting case $\lambda =\infty $, it is ringing permanently and
$\mathcal{P}^{\infty }_{x}$ is a.s.\ dense on $[0,\infty )$. So the spontaneous
transition $A \to S$ happens immediately, and only the process
$\mathcal{P}_{x}^{*}$ is used in the construction. If other particles
are present, the reaction $A+S \to 2A$ has priority and it overrides this
urgency of particles to fall asleep. That is, $A \to S$ only occurs when
an $A$-particle is alone. Sites with $n \geqslant 2$ particles send a particle
away at rate $n$, and when the second last particle jumps out, the remaining
particle falls asleep immediately. In order to define a site-wise representation
with stacks $(\mathfrak{t}^{x,j})_{j\in \mathbb{N}}$ as in \S
\ref{sub:diaconis}, we can consider the stacks with only the jump instructions,
and declare particles to be sleeping whenever they are alone. The properties
mentioned above still hold in this case, and in particular the proofs of
Theorems~\ref{thm:equivalence} and~\ref{thm:monotone} work for
$\lambda \in [0,\infty ]$.

\subsection{A hybrid construction}
\label{sub:monotone_particle}

In this subsection we prove Lemma~\ref{lem:monotone_particle}. The inequality
may seem obvious, and our intuition says it should be a trivial consequence
of monotonicity properties in the spirit of the site-wise construction.
However, in the unlabeled system it is not possible to distinguish the
particles that have exited and re-entered $V_{n}$ from the particles which
have met them after their re-entrance. To solve this, in we will introduce
a two-color site-wise construction.

Recall that $n$ is fixed. For finite $V \supseteq V_{n}$, define
\begin{equation*}
\eta _{0}^{{}_{\square }}(x) =
\begin{cases}
\eta _{0}(x) ,& x \in V_{n},
\\
- \infty ,& x \not \in V_{n},
\end{cases}
\qquad \eta _{0}^{V}(x) =
\begin{cases}
\eta _{0}(x) ,& x \in V,
\\
0 ,& x \not \in V.
\end{cases}
\end{equation*}
Let $M_{t}^{{}_{\square }}$ and $M_{t}^{V}$ count the number of labeled particles
that exit $V_{n}$ by time $t$ in the system with initial configuration
$\eta _{0}^{{}_{\square }}$ and $\eta _{0}^{V}$. Lemma~\ref{lem:monotone_particle}
then becomes\looseness=1
\begin{equation*}
\mathbb{E}\lim _{t\to \infty } M_{t}^{{}_{\square }} \leqslant
\mathbb{E}\lim _{t\to \infty } \lim _{V \uparrow \mathbb{Z}^{d}} M_{t}^{V}
\end{equation*}
for some fixed sequence of finite boxes $V \uparrow \mathbb{Z}^{d}$, where
the limit in $V$ exists by Theorem~\ref{thm:welldefined}. Since
$M_{t}^{V}$ is bounded by $\| \eta _{0} \|_{V}$, using dominated convergence
theorem it is enough to show that $M_{t}^{{}_{\square }}$ is stochastically
dominated by $M_{t}^{V}$ for each $t>0$ and $V \supseteq V_{n}$ fixed.
This will be shown with an explicit coupling.

In this construction, particles which started in $V_{n}$ and have not yet
exited $V_{n}$ are colored purple, and all other particles are colored
yellow. There are two stacks of instructions at each site, one for each
color. This way one can distinguish the particles which have not yet exited
$V_{n}$ from those who have, and use this distinction to define
$M_{t}^{V}$ without the need to look at individual labels.\looseness=1

More precisely, we consider a two-colored particle system, where we are
only interested in particle counts color by color: the configuration at
time $t$ is
$\eta _{t}=(\eta ^{\mathrm{p}}_{t},\eta ^{\mathrm{y}}_{t})\in (
\mathbb{N}_{\mathfrak{s}}\times \mathbb{N}_{\mathfrak{s}})^{
\mathbb{Z}^{d}}$. Initially, purple particles are the particles that start
on $V_{n}$, and yellow particles are the particles that start outside
$V_{n}$. Purple particles become yellow when they exit $V_{n}$. Sample
two independent families $\mathcal{I}^{\mathrm{p}}$ and
$\mathcal{I}^{\mathrm{y}}$ of instructions, to be respectively used by
purple particles and yellow particles (so we never use
$\mathcal{I}^{\mathrm{p}}$ outside $V_{n}$). Sample two collections
${\mathcal{P}}^{\mathrm{p}}$ and ${\mathcal{P}}^{\mathrm{y}}$ of independent
Poisson point processes attached to each site, and use them to trigger
purple and yellow topplings, respectively.

Similarly to the construction of \S \ref{sub:swdefined}, this enables to
construct the process $(\eta ^{\mathrm{p}},\eta ^{\mathrm{y}})$ from a
finite initial configuration, with the difference that two clocks now run
at each site, at speeds given by the numbers of purple and yellow particles,
and each clock triggers a toppling of the respective color. Furthermore,
purple topplings have the additional effect of changing the color of the
jumping particle if it jumps out of $V_{n}$. Note that topplings affect
the state of particles of both colors because yellow and purple at same
site share activity: a yellow particle can prevent a purple particle from
falling asleep, and vice-versa, etc. This construction yields a natural
coupling between systems with any finite initial configuration, each system
alone having the same distribution as the one obtained from the particle-wise
construction.

We remark that this bi-color construction is not Abelian and the local
times $L_{t}^{\mathrm{p}}(x)$ are not generally monotone with respect to
the initial configuration $\xi $. Nevertheless, since each color uses a
different stack of instructions, adding yellow particles has the only effect,
regarding purple particles, of enforcing activation of some of them at
some times. The proof that $L_{t}(x)$ is non-decreasing in $\xi $ given
in \S \ref{sub:swdefined} can be reproduced with nearly no modifications
to show that $L_{t}^{\mathrm{p}}(x)$ is larger with initial configuration
$\eta _{0}^{V}$ than with $\eta _{0}^{{}_{\square }}$. Since yellow topplings
do not change the total number of purple particles present in the system,
and purple topplings can only make that number decrease, the total number
of purple particles present in the system after stabilization will be lower
with initial configuration $\eta _{0}^{V}$ than with
$\eta _{0}^{{}_{\square }}$. Hence, the number of particles that become yellow
(\textit{i.e.}\ the particles that ever exit $V_{n}$) is higher for
$\eta _{0}^{V}$. So in this coupling
$M_{t}^{V} \geqslant M_{t}^{{}_{\square }}$ for every
$V \supseteq V_{n}$, concluding the proof of Lemma~\ref{lem:monotone_particle}.

\section{Historical remarks and extensions}
\label{sec:extensions}

We conclude these lecture notes by making a historical account of how the
main results in this field appeared in the literature, giving appropriate
credit for the arguments presented in previous sections. We then comment
on the extent at which these results can be generalized to other graphs,
initial distributions, jump distributions etc. We finally mention some
of these arguments that have meaningful counter-parts for the Stochastic
Sandpile Model.

\subsection{Historical remarks}

Conditions~\eqref{eq:conditionb} and~\eqref{eq:conditionu} for the ARW
were first used in~\cite{RollaSidoravicius12}. They rely on a construction
of the evolution which is based on the particle-wise discrete representation
introduced in~\cite{DiaconisFulton91,Eriksson96}, formalized in~\cite{RollaSidoravicius12}
and completed in~\cite{RollaTournier18}. Uniqueness of the critical density
was proved in~\cite{RollaSidoraviciusZindy19}, on which \S
\ref{sec:universality} is based.

Results in dimension $d=1$ appeared earlier. The case of directed (\textit{i.e.}\ totally
asymmetric) walks marks a special case. It was shown by~\cite{HoffmanSidoravicius04}
and published in~\cite{CabezasRollaSidoravicius14} that
$\zeta _{c} = \frac{\lambda }{1+\lambda }$, and the proof of Theorem~\ref{thm:zetacexact}
is based on~\cite{CabezasRollaSidoravicius14}. The first result on fixation
for general jump distributions appeared in~\cite{RollaSidoravicius12},
showing that $\zeta _{c} \geqslant \frac{\lambda }{1+\lambda }$. In terms
of the phase space, $\zeta _{c}>0$ for all $\lambda $, and
$\zeta _{c}\to 1$ as $\lambda \to \infty $. Although it can be obtained
via a much simpler argument as in \S \ref{sec:weak}, we give the original
proof in \S \ref{sec:traps} because: the proof can be adapted to more general
graphs such as regular trees, it is used to study the fixed-energy dynamics
in~\S \ref{sec:cycle}, and with simple modifications it gives
$\zeta _{c} > \frac{\lambda }{1+\lambda }$ unless the walks are directed
(Corollary~\ref{cor:depends}).

Still in $d=1$, using Condition~\eqref{eq:conditionu} it is easy to show
that there is no fixation at $\zeta =1$, so in particular
$\zeta _{c} \leqslant 1$ for every $\lambda \leqslant \infty $, see Theorem~\ref{thm:firstcount}.
It was shown in~\cite{Taggi16} that, when the jump distribution is biased,
$\zeta _{c}<1$ for every $\lambda <\infty $, and $\zeta _{c}\to 0$ as
$\lambda \to 0$. The toppling procedure was significantly simplified and
extended to higher dimensions by the introduction of Condition~\eqref{eq:conditione}
in~\cite{RollaTournier18}, and further simplified using Theorem~\ref{thm:universality}.
The proof of Theorem~\ref{thm:taggi} is adapted from~\cite{RollaTournier18},
which was in turn adapted from~\cite{Taggi16}.

For the unbiased case, it was shown in~\cite{BasuGangulyHoffman18} that
$\zeta _{c} \to 0$ as $\lambda \to 0$, and the proof given in \S
\ref{sec:netflow} is adapted from~\cite{BasuGangulyHoffman18}. Existence of
slow and fast regimes for the fixed-energy model on $\mathbb{Z}_{n}$ was
proved in~\cite{BasuGangulyHoffmanRichey19}. The proof given in \S
\ref{sec:cycle} is adapted from~\cite{BasuGangulyHoffmanRichey19}, whose
analysis is substantially based on results and arguments
from~\cite{BasuGangulyHoffman18,Janson18,KingmanVolkov03,RollaSidoravicius12}.

In dimensions $d \geqslant 2$, it was shown that
$\zeta _{c} \leqslant 1$ by~\cite{Shellef10}, using the technique of ghost
walks previously introduced in the context of Internal DLA by~\cite{LawlerBramsonGriffeath92}.
This technique was later used in~\cite{Taggi16} who considers biased jump
distributions and shows that $\zeta _{c}<1$ for small $\lambda $. Ghost
walks are usually useful for introducing independence in order to control
variance and hence bootstrap from a certain counter having high expectation
to being large in probability. The introduction of Condition~\eqref{eq:conditione}
dispenses the use of ghost walks on $\mathbb{Z}^{d}$ or other amenable
graphs, but this technique is still useful in the non-amenable setting,
as used in~\cite{Shellef10,StaufferTaggi18}. The use of ghost walks in
\cite{Shellef10,Taggi16} is surveyed in~\href{https://arxiv.org/abs/1507.04341}{arXiv:1507.04341}.

An alternative proof that $\zeta _{c}\leqslant 1$ was given
in~\cite{AmirGurel-Gurevich10}, where the inequality follows from the mass
transport principle after establishing the property that site fixation is
equivalent particle fixation, under the assumption that the particle-wise
construction of the process is well-defined. This assumption was proved to
hold true in~\cite{RollaTournier18}, where moreover the equivalence property
was used to obtain Condition~\eqref{eq:conditione}. As an application, it is
shown that $\zeta _{c} < 1$ for all $\lambda <\infty $ and $\zeta _{c} \to 0$
as $\lambda \to 0$ for biased walks in any dimension. The construction in \S
\ref{sub:pwconstr} as well as the arguments of \S \ref{sub:conditione} and \S
\ref{sub:particleexistence} are adapted from~\cite{RollaTournier18}. The
proofs in \S \ref{sub:resampling} and \S \ref{sub:agg} are adapted
from~\cite{CabezasRollaSidoravicius14,CabezasRollaSidoravicius18}
and~\cite{AmirGurel-Gurevich10}, respectively. The presentation and proofs in
\S \ref{sub:condbu} and \S \ref{sub:swdefined} are adapted
from~\cite{RollaSidoravicius12} following an important observation
from~\cite{RollaTournier18}.

The technique of weak stabilization was introduced in~\cite{StaufferTaggi18},
where it was shown that, in $d \geqslant 3$, $\zeta _{c} < 1$ for
$\lambda $ small enough, and $\zeta _{c} \to 0$ as $\lambda \to 0$. This
result was strengthened in~\cite{Taggi19} where it was shown that
$\zeta _{c}<1$ for all $\lambda < \infty $. The idea of strong stabilization
was already implicit in~\cite{StaufferTaggi18} but was formalized and better
exploited in~\cite{Taggi19} under a different name. The proofs in \S
\ref{sec:weak} are adapted from these two articles.

The first proof of fixation for $d \geqslant 2$ appeared in~\cite{Shellef10},
who showed that $\zeta _{c}>0$ when $\lambda =\infty $ using results from~\cite{Martin02}
about greedy lattice animals. This was extended in~\cite{CabezasRollaSidoravicius14,CabezasRollaSidoravicius18}
who also consider the extreme case $\lambda =\infty $ and show that
$\zeta _{c} = 1$. In~\cite{SidoraviciusTeixeira17}, it was shown that when
the jump distribution is unbiased, $\zeta _{c}>0$ for every
$\lambda >0$. This was significantly extended in~\cite{StaufferTaggi18},
where the idea of weak stabilization was used to show that
$\zeta _{c} \geqslant \frac{\lambda }{1+\lambda }$ in general.

Many of the problems listed in~\cite{DickmanRollaSidoravicius10} are still
open. From the enormous gap between the description given in \S
\ref{sub:predictions} and the results compiled in \S
\ref{sub:results}, the reader can find countless challenging open problems
in this topic. Besides the problems already listed throughout these notes,
we remark that more direct or otherwise insightful alternatives to some
of the arguments presented here are certainly welcome.

\subsection{Extension to other graphs and settings}
\label{sub:extensions}

Properties of the site-wise representation stated in \S
\ref{sub:diaconis} are deterministic and hold for any graph.

Many proofs about fixation and activity involved re-indexing a sum and
thus rely on translation invariance. A more general context is that of
jump distributions invariant under a transitive unimodular graph automorphism
subgroup with infinite orbits, or \emph{unimodular walks} for short. Translation-invariant
walks on $\mathbb{Z}^{d}$ and uniform nearest-neighbor walks on regular
trees or other Cayley graphs are good examples. Ergodicity in this case
should be understood with respect to this same subgroup.

Well-definedness in the sense of~\eqref{eq:approx} as provided through
\S \ref{sub:swdefined} and \S \ref{sub:particleexistence} still works for
unimodular walks. Theorems~\ref{thm:equivalence},~\ref{thm:universality}
and~\ref{thm:agg} hold with the same or essentially the same proofs. Theorem~\ref{thm:conditione}
remains true~\cite{RollaTournier18} but requires the graph to be amenable,
and it is false on trees.

Theorem~\ref{thm:firstcount} remains true for unimodular walks through
the proof of Theorem~\ref{thm:abactive}. The proof of Theorem~\ref{thm:taggi}
can also be extended directly to unimodular walks on amenable graphs~\cite{RollaTournier18}
and partially to non-amenable graphs using ghost walks~\cite{StaufferTaggi18}.

The argument shown in \S \ref{sec:traps} can be adapted to trees~\cite{StaufferTaggi18}.
The argument for the universal bound
$\zeta _{c} \geqslant \frac{\lambda }{1+\lambda }$ shown in \S
\ref{sub:weakcor} works on any amenable graph, and it is possible to extended
it to non-amenable graphs. Proofs of Theorems~\ref{thm:subcritstauffertaggi}
and~\ref{thm:supcritstauffertaggi} are given in~\cite{StaufferTaggi18,Taggi19}
assuming uniform nearest-neighbor walks, but they also work for unimodular
walks.

In \S \ref{sec:cycle} we use Poisson thinning which shortens the argument
a bit, but other non-Poisson initial distributions with good concentration
inequalities should suffice with little extra work.

The proof that $\zeta _{c} \geqslant 1$ given in~\cite{Shellef10} in fact
shows activity for deterministic initial configurations having empirical
average larger than unit, and does not even assume any symmetries on the
graph besides having bounded degree.

The assumption of nearest-neighbor jump simplified the exposition, but
it is really needed only in \S \ref{sec:traps}, \S \ref{sec:netflow} and
\S \ref{sec:cycle}. For these sections, dropping such assumption would
decrease the range of parameters for which the proofs work, or make the
constructions more complicated, or both. The argument briefly outlined
in \S \ref{sec:multiscale} extends directly to a bounded-range or very
light-tail jump distribution, as long as it remains unbiased. The arguments
shown in the other sections do not require any assumption on the jump distribution
besides translation invariance (or unimodularity). In this case, the notion
of a walk being directed is the one stated in the proof of Corollary~\ref{cor:directed}.

\subsection{Stochastic Sandpile Model}
\label{sub:ssm}

For the Stochastic Sandpile Model (SSM), there is a continuous-time process
with properties analogous to those of \S \ref{sub:diaconis}, see~\cite{RollaSidoravicius12}.
Theorem~\ref{thm:equivalence} holds with a similar proof. The proof of
Theorems~\ref{thm:conservation},~\ref{thm:welldefined}~and~\ref{thm:equivalent},
and hence that of Theorem~\ref{thm:conditione}, can be adapted to the SSM
without substantial changes.

In one dimension, it has been shown that
$\zeta _{c} \geqslant \frac{1}{4}$ using a toppling procedure similar to
the one shown in \S \ref{sec:traps}, see~\cite{RollaSidoravicius12}. Also,
it should not be too complicated to combine the toppling procedure from
Section~6 of~\cite{RollaSidoravicius12} with a urn process similar to Lemma~\ref{lemma:urn}
to show fast fixation for $\zeta <\frac{1}{4}$.

For higher dimensions, the multi-scale argument described in \S
\ref{sec:multiscale} works without any modification to show that
$\zeta _{c}>0$ for the SSM with unbiased walks, see~\cite{SidoraviciusTeixeira17}.
The proof of Theorem~\ref{thm:subcritstauffertaggi} can be partially adapted
to show that $\zeta _{c}>0$ for more general jump distributions.

To the best of our knowledge, this is all that is known rigorously about
the SSM, as far as predictions in the spirit of \S
\ref{sub:predictions} are concerned.

\section*{Acknowledgments}

These lecture notes developed throughout the years in parallel with the
literature, and were used in courses given at Institute for Advanced Study
of Aix-Marseille University, University of Buenos Aires, Chinese Academy
of Sciences, Henan University, National University of Singapore, and Peking-Tsinghua-\break Beijing
Normal Universities joint seminar. Much of the current narrative originated
from discussions with mathematicians and physicists, especially Deepak
Dhar and P~K~Mouhanti during the workshop Universality in Random Structures:
Interfaces, Matrices, Sandpiles, held on the International Centre for Theoretical
Sciences, Bangalore, 2019 (Code: ICTS/urs2019/01). These events were supported
by A*MIDEX grant ANR-11-IDEX-0001-02 and the National Natural Science Foundation
of China grant 11688101.

The author's current understanding of Activated Random Walks and related
models is the result of discussions and exchanges with many colleagues
throughout more than a decade, it would be hard to name them all. Let us
just mention Vladas Sidoravicius, Ronald Dickman, Manuel Cabezas, Lionel
Levine and Laurent Tournier among those who had the deepest impact.

While these notes were being finalized, Shirshendu Ganguly and Moumanti
Podder contributed many remarks that improved the clarity and readability.
Laurent Tournier and Vittoria Silvestri went through most of these notes
and made a large number of keen remarks that substantially improved the
presentation. Special thanks to Alexandre Stauffer and Jacob Richey who
read the entire material thoroughly and made thoughtful comments and suggestions,
without which the quality of these notes would not have been the same.
All the mistakes, omissions or obscure explanations that remain are the
author's responsibility.

\section*{Update}

Some of the open problems mentioned in these notes or similar ones have
been solved in~\cite{AsselahSchapiraRolla19,CabezasRolla20,HoffmanRicheyRolla20,PodderRolla20,Taggi20}.






\end{document}